\documentclass[11 pt]{amsart}
\usepackage[foot]{amsaddr}
\usepackage[left=1in,top=1.3in,right=1in,bottom=1in,footskip = 0.333in]{geometry}

\RequirePackage{natbib}

\RequirePackage[OT1]{fontenc}
\usepackage[hang,small,bf]{caption}
\RequirePackage{amsthm,amsmath,natbib,graphicx,enumitem,comment,bbm}
\RequirePackage[colorlinks,citecolor=blue,urlcolor=blue]{hyperref}
\RequirePackage{hypernat}
\usepackage{amssymb,tabularx,multicol,multirow,booktabs,setspace}
\usepackage{mhequ}
\usepackage[usenames,dvipsnames]{color}
\usepackage{epstopdf}
\usepackage{subcaption}

\usepackage[colorinlistoftodos]{todonotes}

\input binary_macro.tex
\def \bbeta{\boldsymbol{\beta}}

\def \bX{\mathbf{X}}
\def \bbX{\mathbb{X}}
\def \bW{\mathbf{W}}
\def \bbW{\mathbb{W}}
\def \tlambda{\tilde{\lambda}}
\def \tV{\tilde{V}}

\def \be{\begin{equs}}
	\def \ee{\end{equs}}

\def \E{\mathbb{E}}
\def \P{\mathbb{P}}

\def \fhat{\hat{f}}

\def \sumn{\sum_{i=1}^n}
\def \bmu{\mathbf{\mu}}
\def \R{\mathbb{R}}

\def \btheta {\boldsymbol{\theta}}

\def \bmu {\boldsymbol{\mu}}

\newtheorem{assn}[thm]{Assumption}
\def \into {\mathrm{in}}
\def \out {\mathrm{out}}
\def \ind {\rm ind}

\begin{document}

\title[PC regression]{PC adjusted testing for low-dimensional parameters}

\author[S. Bhattacharya]{Sohom Bhattacharya}
\address{S. Bhattacharya\hfill\break
	Department of Statistics\\ University of Florida\\ Gainesville, FL, USA.}
\email{bhattacharya.s@ufl.edu}
\author[R. Dey]{Rounak Dey}
\address{R. \ Dey\hfill\break
	Senior Data Scientist, Insitro Inc., San Francisco, CA, USA.}
\email{rdey@insitro.com}
\author[R. Mukherjee]{Rajarshi Mukherjee}
\address{R. \ Mukherjee\hfill\break
	Departments of Biostatistics\\ Harvard T.H. Chan School of Public Health\\ Boston, MA, USA.}
\email{ram521@mail.harvard.edu}
\thanks{\texttt{R. Mukherjee was supported by NSF CAREER 8529216-01.}}

\begin{abstract} 
	In this paper, we investigate the impact of high-dimensional Principal Component (PC) adjustments on inferring the effects of variables on outcomes, with a focus on applications in genetic association studies where PC adjustment is commonly used to account for population stratification. We consider high-dimensional linear regression in the regime where the number of covariates grows proportionally to the number of samples. In this setting, we provide an asymptotically precise understanding of when PC adjustments yield valid tests with controlled Type I error rates. Our results demonstrate that, under both fixed and diverging signal strengths, PC regression often fails to control the Type I error at the desired nominal level. Furthermore, we establish necessary and sufficient conditions for Type I error inflation based on covariate distributions. These theoretical findings are further supported by a series of numerical experiments.

\end{abstract}

\maketitle




\section{Introduction} 
Testing statistical hypotheses, where one needs to take special care of complications of high dimensional data structures is a staple in modern data science. Indeed, a quintessential example of this  naturally arise in the context of large genetic association studies. In the context of genetic association studies, such adjustments are motivated by accounting for hidden confounding due to population stratification. In this regard, the seminal paper of \cite{price2006principal} has made a strong case for Principle Component (PC) adjustments to account for population substructure based confounding in genetic association studies. The success of this method have thereafter provided impetus for other applied researchers to appeal to a similar philosophy. Specifically, as another specific example, \cite{barfield2014accounting} explores the applicability of the same principle in epigenome wide association studies. A major appeal of this method, that we will try lend a theoretical lens to in this paper, is availability of the EIGENSTRAT package of \cite{price2006principal} that allows fast implementation of this principle and has therefore been immensely popular as a standard adjustment procedure to account for confounding bias in such studies.

\textcolor{black}{In spite of the immense practical popularity of PC adjustment for high dimensional nuisance parameters,  many basic theoretical questions remain unanswered for this high popular and regularly used principle. For example, the most traditional statistical questions that arise as a part of planning a real data analysis pertains to sample size and power calculations. Interestingly, the answers to even such simple questions seems to have remain unexplored. Indeed, this proposed two stage method involving a principal component adjustment before testing for  associations in large scale studies, implies interesting statistical phenomenon in high dimensional problems where the number of genetic variants under study is large ~\cite{patterson2006population}
}. 

In contrast to such implied caveats and potential inconsistencies in high dimensional PCA ~\cite{johnstone2009consistency}, empirical success of the procedure in practical problems in turn go far towards making a strong case for the PC based adjustments in these applications ~\cite{yang2014principal}. This paper is aimed at providing first steps towards a theoretical formalism and reconciling these issues and thereby potentially resolve questions regarding when and why such PC-based adjustments are reliable and mathematically valid. 

In this study, we examine the behavior of PC adjustment-based testing procedures across various regimes of population PC signal strengths. Specifically, we first show that, without additional assumptions, this method fails to control Type I error under fixed signal strength regimes. While this phenomenon has been empirically observed before ~\cite[Table 1]{zhang2015principal}, we aim to reconcile the method’s practical success with its Type I error issues by analyzing regimes of diverging population PC strengths. In doing so, we establish the necessary and sufficient conditions on population PC strengths required for Type I error control. To the best of our knowledge, this work offers the first theoretical explanation for this phenomenon in the setting where the number of covariates grows proportionally to the number of samples. Importantly, the goal of this paper is not to introduce new methods for handling population stratification or high-dimensional nuisance adjustment. Instead, by providing necessary and sufficient conditions for the effectiveness of a widely used method, we offer guiding principles for practitioners on the assumptions required when relying on PC-based adjustments for downstream inference in high-dimensional settings.

\section{The Method and Models} In this section we introduce the necessary mathematical formalism to describe the method ~\cite{price2006principal,patterson2006population} we aim to rigorously study and motivate models and assumptions to study its theoretical behavior. We first describe the data structure followed by the testing procedure espoused in ~\cite{price2006principal,patterson2006population}. Subsequently, we discuss a population analog of the method to develop intuition on the models and assumptions under which we will operate.

We consider data on $(Y_i, \mathbf{X}_i)_{i=1}^n\stackrel{\rm i.i.d.}{\sim}\P$ where $Y_i\in \mathbb{R}$ denotes an outcome of interest and $\mathbf{X}_i\in \mathbb{R}^p$ collects predictor variables under study. In many case, it is of interest to test the effect of specific variables/ components in $\mathbf{X}$ on $Y$ while conditioning on other variables in $\mathbf{X}$. More precisely, we will partition  $\bX$ as $\bX=[A, \mathbf{W}]^T$ where $A\in \mathbb{R}$ is the variable whose effect on $Y$ is of interest and $\mathbf{W}$ collects confounders one wishes to adjust while decoding this relationship.  Our main question of interest concerns a specific way of inferring the effect of $A$ on $Y$ -- namely PC adjustment for the confounding variables in $\bX$.   Specifically, we shall study the statistical properties (Type I and Type II error) of testing for a specific component $A$ of $\bX$ on $Y$. We introduce the method following traditional GWAS analysis where it is a standard practice to run PCA on the whole data matrix of $\bX$ (instead on the matrix obtained by removing $A$-observations from it) and include them in the analyses  -- instead of running a new PCA every time for testing for a new variable in $\bX$. Here we will explore both the cases when the PC is run on the whole of $\bX$ (a case more relevant for GWAS type analyses ~\cite{price2006principal}) and for the case when PC adjustments are performed on $\bW$ (a case relevant in certain epigenetic studies where $A$ contains DNA methylation levels of a particular cytosine-phosphate-guanine 
(CpG) site of interest and one performs PCA on genetic variants $\bW$ to account for population stratification \cite{barfield2014accounting}).

To initiate a theoretical study of this method, we begin by setting up the necessary mathematical notation to describe the procedure. As discussed above, one standard approach in practice is to perform PC analysis on the data matrix of $\bX$ and only use the top principal component directions while adjusting for confounders in the regression of $Y$ on $A$. Mathematically, this can be written as follows through the data vectors $\mathbf{Y}=(Y_1,\ldots,Y_n)^T$, $\mathbf{A}=(A_1,\ldots,A_n)^T$ and the data matrix of covariates $\mathbb{X}=[\mathbf{X}_1:\cdots:\mathbf{X}_n]^T$:
\begin{enumerate}[wide, labelwidth=!, labelindent=0pt]
	\item [(i)] Given $k\in \mathbb{N}$, let $\widetilde{\mathbf{V}}_{k,\into}$ denote the matrix composed of the first $k$ right singular vectors of $\mathbb{X}$.
	
	\item [(ii)] Perform a linear regression of $\mathbf{Y}$ on $(\mathbf{A},\mathbb{X}\widetilde{\mathbf{V}}_{k,\into})$ to get a score or level of significance for $\delta$.
\end{enumerate}
Here the subscript $\into$ refers to the fact that $A$ is a part of the variables \textit{in} $\bX$ on which the PC was performed. For any cut-off $t$, we consider the test that rejects for value of the likelihood ratio larger than $t$:
\be
\varphi_{k,\into}(t)&:=\mathbf{1}\left(\mathrm{LR}_{k,\into}>t\right),\quad \text{where}\\
\mathrm{LR}_{k,\into}&=\mathbf{Y}^TP_{\mathcal{C}({\mathcal{A}_{k, \into}})}\mathbf{Y},\quad {\mathcal{A}_{k, \into}}=(I-P_{\mathcal{C}(\mathbb{X}{\widetilde{\mathbf{V}}_{k,\into}})})\mathbf{A}.\label{eqn:lr_in}
\ee
Here for any matrix $M$ we use $\mathcal{C}(M)$ to denote its column space and $P_{\mathcal{C}(M)}$ to denote the projection matrix onto $\mathcal{C}(M)$. An analogous test can be constructed where the PC adjustment was performed only on $\mathbb{W}:=[\bW_1:\cdots:\bW_n]^T$ -- which we shall refer to as
\be
\varphi_{k,\out}(t)&:=\mathbf{1}\left(\mathrm{LR}_{k,\out}>t\right),  \hspace{0.25em} \text{where }
\mathrm{LR}_{k,\out}=\mathbf{Y}^TP_{\mathcal{C}({\mathcal{A}_{k, \out}})}\mathbf{Y},\quad
 {\mathcal{A}_{k,\out}}=(I-P_{\mathcal{C}(\mathbb{W}\widetilde{\mathbf{V}}_{k,\out})})\mathbf{A}\label{eqn:lr_out}
\ee
and $\widetilde{\mathbf{V}}_{k,\out}$ now collects the top right singular vectors of $\mathbb{W}$.

In order to study the properties of test statistics $\mathrm{LR}_{k,\ind}$ for $\ind \in \{0,1\}$ , we now delineate data generating models which serve as a template for the population analog of the method described above. To this end, the test based on $\mathrm{LR}_{k,\out}$ operates through a regression of $Y_i$'s on $A_i$'s and $\widetilde{\mathbf{V}}^\top_{k,out}\bW_i$. However the $j^{\rm th}$-element of $\widetilde{\mathbf{V}}^\top_{k,out}\bW_i$ is exactly $\bW_i^T\hat{v}_j$, the $j^{\rm th}$ sample PC score, where $\hat{v}_j$ is the $j^{\rm th}$-eigenvector of $\frac{1}{n}\sum_{i=1}^n \bW_i\bW_i^T$. Therefore, a natural population counterpart of the problem motivates a regression of $Y$ on $A$ and top $k$ population PC scores given by $\bW^Tv_1,\ldots,\bW^Tv_k$ with $v_j$ being the $j^{\rm th}$-eigenvector of $\mathbb{E}(\bW\bW^T)$. However, mathematically this implies positing a linear model: $\mathbb{E}(Y|A,\bW)=A\delta+\sum_{j=1}^k \alpha_j\bW^Tv_j=A\delta+\bW^T\bbeta$ with $\bbeta\in \rm{Span}(v_1,\ldots,v_k)$. This motivates us to consider the following linear model:
\be
Y_i=A_i\delta+\bW_i^T\bbeta+\epsilon_i, \quad \epsilon_i\sim N(0,\sigma^2_{\varepsilon}),\label{eqn:model_gaussian_linear_outcome_regression}
\ee
and thereby consider the hypothesis testing problem for $\delta$ defined as 
\be
H_0: \delta=0 \quad vs \quad H_1(h):\delta=\frac{h}{\sqrt{n}}.\label{eqn:main_hypo}
\ee 

The choice of the alternative converging to $0$ at $1/\sqrt{n}$ rate is driven by the fact that the detection threshold for testing $\delta$ scales accordingly in the asymptotic regimes we will work with in this paper. Hence, with some abuse of notation, we will denote the distribution $\P$ of $Y|A,\bW$ in model \eqref{eqn:model_gaussian_linear_outcome_regression} as $\P_{\delta}$, highlighting the central focus of the hypothesis test, $\delta$. Finally, throughout we assume that $\|\bbeta\|$ is bounded away from $0$ to reflect a scenario where the effect of $A$ on $Y$ experiences non-trivial confounding through $\bW$.

The standard test for \eqref{eqn:main_hypo} is the Generalized Likelihood Ratio Test (GLRT)  -- which for known $\sigma_y^2$\footnote{We will throughout assume this and show through numerical experiments that the results for unknown $\sigma^2_{\varepsilon}$ should be qualitatively similar.} reduces to the chi-square test with $1$-degree of freedom. Although the same test has non-trivial local power for testing \eqref{eqn:main_hypo} even for growing dimension $p$: (1), it is often believed that suitable dimension reduction techniques on the regression of $\bX$ on $Y$ by relying on the intuition that dimension reduction techniques capture underlying structures of the problem (e.g. sparsity, low rank etc) and thereby yields more power for the testing problem compared to vanilla GLRT; and (2)  performing GLRT at the scale of modern high dimensional problems such as GWAS is computationally expensive (${O}(\min\{p^3,np^2,n^2p\})$) and performing dimension reduction on $\bX$ substantially reduces this burden. The method described above in \eqref{eqn:lr_in} and \eqref{eqn:lr_out} through $\mathrm{LR}_{k,\ind},\varphi_{k,\ind}$  was proposed to address such issues while accounting for confounding in genetic association studies, is  similar to Principal Component Regression (PCR), and raises a few  immediate questions:
    \begin{enumerate}
        \item [(i)] What is the behavior of $\varphi_{k,\mathrm{\mathrm{ind}}}(t)$, $\ind \in \{\into,\out\}$ under $H_0$ and does there exist a 
        cut-off $t_{\alpha}$ such that the Type I error of $\varphi_{k,\mathrm{\ind}}(t_{\alpha})$ converges to $\alpha$?
        
        \item [(ii)] How to characterize the power function of $\varphi_{k,\mathrm{ind}}(t)$, $\ind \in \{\into,\out\}$ for any given $(k,t)$?
        

    \end{enumerate}



\textbf{Our Contribution:} We characterize the behavior of generalized LRT for PCR in the regime $\frac{p}{n} \rightarrow \gamma>0$, i.e., number of covariates $p$ grows proportionally with sample size $n$, can potentially be larger than sample size. In this regard, the main results of this paper can be summarized as follows:
\begin{enumerate}[wide]
	\item [(i)] Under a generalized spiked model for $\bW$ (see Definition \ref{defn:generalized_spike_model}) we provide exact asymptotic behavior for testing \eqref{eqn:main_hypo} after PC adjustments. The precise description of the testing errors can be elaborated as follows.
	
	\begin{enumerate}
	    \item [(a)] There is an important distinction between fixed effects and random effects regimes on $\bbeta$. In particular, when $\bbeta$ in \eqref{eqn:main_hypo} is considered fixed and unknown the Type I error of the standard test for \eqref{eqn:main_hypo} after PC adjustments converges to $1$. In contrast, when $\bbeta$ is a random effect the Type I error of the same test is bounded away from both $1$ and the nominal desired level. (See Theorem \ref{thm:spiked_model_fixed_strength_out}\ref{thm:both_fixed_fixed_strength_out_reg}-\ref{thm:one_fixed_fixed_strength_out_reg}).
	    
	    \item [(b)] For random $\bbeta$, we derive the power function of the test for \eqref{eqn:main_hypo} after PC adjustments and show that without further assumptions it might not be possible to estimate the power function exactly to either correct for the Type I error inflation or perform sample size/power calculations (See Theorem \ref{thm:spiked_model_fixed_strength_out}\ref{thm:one_fixed_fixed_strength_out_reg}).
	    
	\end{enumerate} 
	
\item [(ii)] Moving beyond the spiked model, we show that when $\bW$ is derived from a mixture of mean-shift type distributions, a similar phenomenon (Type-I error inflation) persists in terms of the behavior of the test for \eqref{eqn:main_hypo} after PC adjustments. We further verify similar behavior through extensive simulations even when the distribution of $\bX$ arises from mixture of discrete distributions -- a situation that mirrors the case of large scale genetic association testing more closely. 

\item [(iii)] Moving beyond fixed signal strenths, we examine the type-I error beahvior for PC adjusted tests unde a classical spiked model ~\cite{johnstone2001distribution} $\bW$ with diverging leading spike strengths. We precisely characterize the effect of spike  strength and the angle between the $\bbeta$ and leading population eigenvector on the procedures $\varphi_{k,\mathrm{ind}}$ noted in points (i) and (ii) -- see Theorem \ref{thm:diverging_spike_out}-\ref{thm:diverging_spike_delocalized}.

\item[(iv)] We provide extensive numerical experiments that complement our theoretical findings (See Section \ref{sec:numerical_experiments}) even for moderate sample sizes. Our simulations exhibit universality of our findings under diverse covariate distributions.

\end{enumerate}
The rest of the paper, organized as follows, makes the above claims precise, rigorous, and elaborate. We first introduce the notation, definitions, and assumptions to set up the main building blocks of our results in Section \ref{sec:tech_prep}. We divide the main theoretical results in Section \ref{sec:main_results} into three parts as follows. In Section \ref{sec:fixed_strength} we present asymptotic results on the power function of $\varphi_{k,\ind},\ind \in \{\into,\out\}$ for fixed spike/signal strength. Subsequently, we present the implications of these results in the context of GWAS type problems in Section \ref{sec:gwas_implications}. Section \ref{sec:diverging_strength} is devoted to understanding the effect of diverging signal strength and thereby providing benchmarks to decide when and how such PC adjustment based testing procedures are valid in practice. Finally, in Section \ref{sec:numerical_experiments} we present detailed numerical experiments to not only verify our theoretical findings in  simulation examples but also explore the continued validity of our analytical results beyond the specific conditions assumed for the theoretical derivations.


\subsection{\bf Technical Preparations}\label{sec:tech_prep} In this subsection we present definitions, assumptions and related discussions, and notation that will be used throughout the rest of the paper.

\subsubsection{\bf Definitions}
To present our main results of this section, we first need a few definitions. We start with the definition of empirical spectral measure of a Hermitian matrix.
\begin{defn}\label{defn:spectral_distribution}
For a $p$-dimensional Hermitian matrix $\Sigma_p$ with eigenvalues $\lambda_{1,p},\lambda_{2,p},\ldots,\lambda_{p,p}$, the Empirical Spectral Distribution (ESD) $H_p$ is defined as the distribution function corresponding to the probability measure 
\begin{align*}
    H_p(x)=\frac{1}{p}\sum_{i=1}^{p} \mathbbm{1}_{\{x \ge \lambda_{i,p}\}}. 
\end{align*}
For a sequence $\{\Sigma_p\}_{p\geq 1}$ of deterministic Hermitian matrices, if the corresponding sequence $\{H_p\}_{p\geq 1}$ of ESDs converges weakly to a deterministic probability distribution $H$ as $p\rightarrow\infty$, then $H$ is defined as the Limiting Spectral Distribution (LSD) of the sequence $\{\Sigma_p\}_{p\geq 1}$.
\end{defn}

Next we define generalized spiked models as proposed by \cite{bai2012gsp}, which considers the base covariance matrix to be arbitrary, instead of the identity matrix as in the classical spiked model ~\cite{johnstone2001distribution}.

\begin{defn}\label{defn:generalized_spike_model}
A sequence $\{\Sigma_p\}_{p\geq 1}$ of deterministic Hermitian matrices is called a sequence of generalized spiked population matrices if the following hold.
\begin{enumerate}[label=(\roman*)]
    \item $\Sigma_p$ can be written as, $$\Sigma_p=\begin{bmatrix}\Sigma^{(A)}_m & 0\\ 0 & \Sigma^{(B)}_{p-m}\end{bmatrix},$$ where $\Sigma^{(A)}_m$ and $\Sigma^{(B)}_{p-m}$ are non-negative and non-random Hermitian matrices of dimensions $m\times m$ ($m$ is finite and fixed) and $(p-m)\times (p-m)$, respectively.
    \item The sequence $\{H_p\}_{p\geq 1}$ of ESDs corresponding to $\{\Sigma_p\}_{p\geq 1}$ converges weakly to a non-random probability distribution $H$.
    \item Let $\Gamma_H$ be the support of $H$, and let the sets of eigenvalues of the submatrices $\Sigma^{(A)}_m$ and $\Sigma^{(B)}_{p-m}$ be $\{\alpha_{1},\geq\ldots\geq\alpha_{m}\}$ and $\{\beta_{1,p},\geq\ldots\geq\beta_{p-m,p}\}$, respectively. Then, $\alpha_i\notin \Gamma_H$ for $i=1,\ldots,m$, and $\max_{1\leq i\leq p-m}{d\left (\beta_{i,p},\Gamma_H\right )}\rightarrow 0$ as $p\rightarrow \infty$, where $d(x,S_A)=\inf_{y\in S_A}|x-y|$ is a distance metric from a point $x$ to a set $S_A$.
\end{enumerate}

In this case, the eigenvalues of $\Sigma^{(A)}_m$ are called the generalized spikes, and the eigenvalues of $\Sigma^{(B)}_{p-m}$ are called the non-spikes. The distribution $H$ is same as the LSD of the sequence $\{\Sigma^{(B)}_{p-m}\}$. 
\end{defn}

Note that the traditional spiked model ~\cite{johnstone2001distribution} is a special case of the generalized spiked model where $H$ is the degenerate distribution at unity.
Next, we will introduce the phase transition boundaries for the generalized spikes analogous to the phase transition boundary ~\cite{baik2005phase} established in the spiked model. {\color{black}Recall that, $\frac{p}{n} \rightarrow \gamma >0$.}
\begin{defn}\label{defn:generalized_bbp}
For $\alpha\notin \Gamma_H, \alpha\neq 0$, we define the following function:
$$\psi(\alpha)=\alpha+\gamma\alpha\int{\frac{\lambda dH(\lambda)}{\alpha-\lambda}}.$$
Then, a generalized spike $\alpha_i$ is called a ``distant spike" if $\psi ' (\alpha_i)>0$, and ``close spike" if $\psi '(\alpha_i)\leq 0$, where $\psi '$ is the first derivative of $\psi$.
\end{defn}
Depending on whether a generalized spike is a distant or a close spike, the asymptotic convergence of the corresponding sample eigenvalues differ according to Theorem 4.1 and Theorem 4.2 of ~\cite{bai2012gsp}. Therefore, the sample eigenvalues corresponding to the generalized spikes go through a phase transition at the boundaries where $\psi ' (\alpha)=0$. Unlike in the spiked model where the phase transition happens only at the two boundaries $(1\pm \sqrt{\gamma})$, the generalized spiked model can have more than two phase transition boundaries. The next assumption is intended to simplify our mathematical derivations by only allowing one phase transition boundary $\alpha^*>\sup{\Gamma_H}$, and by assuming that all the generalized spikes have multiplicity one and are above that boundary (i.e., all distant spikes). The extension of our derivations to allow close spikes, and generalized spikes around multiple phase transition boundaries are similar but tedious, and hence omitted in this paper. Finally we define covariate data distributions we will working with throughout.

\begin{defn}\label{defn:data_mat_generalized spiked}
We say that a data matrix $\widetilde{\mathbb{X}}_{n\times p}$ follows a generalized spiked distribution with $k^*$ spike eigenvalue-eigenvector pairs $\left\{(\lambda_j,v_j)\right\}_{j=1}^{k^*}$ if  $\widetilde{\mathbb{X}}=\mathbf{Z}\Lambda_p^{1/2}V_p$, where $\mathbf{Z}$ is an $n\times p$ random matrix with i.i.d. sub-Gaussian elements such that $\E(Z_{ij})=0,\E(|Z_{ij}|^2)=1$, and $\Lambda_p^{1/2}$ is the positive definite Hermitian square root of $\Lambda_p$ such that
\begin{enumerate}[label=(\alph*)]
         \item $\{\Sigma_p\}_{p\geq 1}$ is a sequence of real symmetric positive definite matrices with spectral decomposition $\Sigma_p=V_p\Lambda_pV_p^\top$, and $\Lambda_p$ satisfies the generalized spiked assumptions as outlined above (Definition~\ref{defn:generalized_spike_model}) with $k^*$ (finite) generalized spikes. The eigenvalues and eigenvectors of $\Sigma_p$ are given by $\lambda_1,\ldots,\lambda_p$ and $v_1,\ldots,v_p$, respectively.
    \item The sequence $\{\|\Sigma_p\|\}$ of spectral norms is bounded.
    
    \item $\lambda_1> \lambda_2> \ldots>\lambda_{k^*}>\sup{\Gamma_H}$ denote the generalized spikes of $\Sigma_p$, and \textcolor{black}{$\psi'(\lambda_i)>0$} for $i=1,\ldots,k^*$.
\end{enumerate}
In this case we denote $\widetilde{\mathbb{X}}\sim \mathrm{GSp}(\left\{(\lambda_j,v_j)\right\}_{j=1}^{k^*};\Gamma_H;n,p)$
\end{defn}

Some of our results are derived under the classical spiked model ~\cite{johnstone2001distribution}. To that end we will use the following definition.

\begin{assn}
We say that a data matrix $\widetilde{\mathbb{X}}_{n\times p}$ follows a spiked distribution with $k^*$ spike eigenvalue-eigenvector pairs $\left\{(\lambda_j,v_j)\right\}_{j=1}^{k^*}$ if  $\widetilde{\mathbb{X}}=\mathbf{Z}\Sigma^{1/2}_p$, where $\mathbf{Z}$ is an $n\times p$ random matrix with i.i.d. sub-Gaussian elements such that $\E(Z_{ij})=0,\E(|Z_{ij}|^2)=1$, and {\color{black}$\Sigma_p=I+\sum_{j=1}^{k^*} \lambda_j v_jv_j^T$ }with $\lambda_1\geq \ldots\lambda_{k^*}>\sqrt{\gamma}$ where $\gamma=\lim p/n$ and $v_1,\ldots,v_{k^*}$ are orthonormal vectors. In this case we denote $\widetilde{\mathbb{X}}\sim \mathrm{Sp}(\left\{(\lambda_j,v_j)\right\}_{j=1}^{k^*};n,p)$
\end{assn}


\subsubsection{\bf Assumptions}
Using the above definitions, we can now state our assumptions various subsets of which will be useful for our theoretical results.

\begin{assn}\label{assn:spike_and_dimension}
\begin{enumerate}
    \item  [(a)]$ n\rightarrow\infty, p= p(n) \rightarrow\infty, p/n\rightarrow\gamma\in (0,\infty)$.
   
    \item [(b)] $\bbW\sim \mathrm{GSp}(\left\{(\lambda_j,v_j)\right\}_{j=1}^{k^*};\Gamma_H;n,p)$.
    
    \item [(b')] $\bbW\sim \mathrm{Sp}(\left\{(\lambda_j,v_j)\right\}_{j=1}^{k^*};p)$.

    \item[(b'')] $\bbW = \mathbf{Z}\Sigma^{1/2}_p$, the rows of $\mathbf{Z}$ are identically distributed, and the elements of $\mathbf{Z}$ are independent sub-Gaussian elements with mean $0$, variance $1$, and uniformly bounded sub-Gaussian norm. $\Sigma_p$ satisfies $\lambda_{k^{\star}} \rightarrow \infty$ and $\lambda_{k^{\star}+1}= O(1)$.
    
    \item [(c)] ${\bbX}:=[\mathbf{A},\bbW]\sim  \mathrm{GSp}(\left\{(\lambda_j,v_j)\right\}_{j=1}^{k^*};\Gamma_H;n,p+1)$.
    
    \item [(c')] ${\bbX}:=[\mathbf{A},\bbW]\sim  \mathrm{Sp}(\left\{(\lambda_j,v_j)\right\}_{j=1}^{k^*};p+1)$.
\end{enumerate}
\end{assn}
These assumptions specify the growth of covariates in the model along with classical or generalized spiked model structure of the covariates.


{\color{black} We will mostly assume the following dependence between $\mathbf{A}$ and $\mathbb{W}$.}
\begin{assn}\label{assn:prop_projection}
Let $\mathbf{A}=\mathbb{W}\btheta+\boldsymbol{\eta}$ with $\boldsymbol{\eta} \sim N(0,\sigma^2_g I)$.
\end{assn}
Although Assumption \ref{assn:prop_projection} seems restrictive, we keep it here for keeping our arguments short and only for the \textit{precise} analysis of the test $\varphi_{k,\out}$ -- this assumption is often not needed either for analysis of $\varphi_{k,\into}$ or while demonstrating a lower bound on the Type I error of $\varphi_{k,\out}$ instead of deriving a precise limit. Indeed, this result is easy to extend to sub-Gaussian $A$'s and we additionally verify through extensive simulations (see Section \ref{sec:numerical_experiments}) that the intuitions from our main theorems continue to be valid even without assuming the above specific conditional distribution of $A|\bX$. 

Finally we note that for the analysis in epigenetic studies ~\cite{barfield2014accounting}, where $A$ contains DNA methylation levels of a particular CpG site of interest and one performs PCA on genetic variants $\bW$ to account for population stratification, Assumption \ref{assn:prop_projection} is often reasonable. 

Our last assumption removes this restrictions in some of our results at the cost of obtaining bounds instead of precise limiting objects.  
\begin{assn}\label{assn:prop_projection_nonlin}
Let $A_i$'s be centered and sub-Gaussian with parameter $\sigma^2_a$. 
\end{assn}

Assumption \ref{assn:prop_projection_nonlin} will be used for deriving lower bounds on Type I error for tests based on $\mathrm{LR}_{k,\out}$ and obtain partial sharp limiting Type I error for tests based on $\mathrm{LR}_{k,\into}$ (see e.g. Theorem \ref{thm:thm:spiked_model_fixed_strength_out_nogaussian_a}).

\subsubsection {\bf Notation}
The results in this paper are mostly asymptotic (in $n$) in nature and thus requires some standard asymptotic  notations.  If $a_n$ and $b_n$ are two sequences of real numbers then $a_n \gg b_n$  (and $a_n \ll b_n$) implies that ${a_n}/{b_n} \rightarrow \infty$ (and ${a_n}/{b_n} \rightarrow 0$) as $n \rightarrow \infty$, respectively. Similarly $a_n \gtrsim b_n$ (and $a_n \lesssim b_n$) implies that $\liminf_{n \rightarrow \infty} {{a_n}/{b_n}} = C$ for some $C \in (0,\infty]$ (and $\limsup_{n \rightarrow \infty} {{a_n}/{b_n}} =C$ for some $C \in [0,\infty)$). Alternatively, $a_n = o(b_n)$ will also imply $a_n \ll b_n$ and $a_n=O(b_n)$ will imply that $\limsup_{n \rightarrow \infty} \ a_n / b_n = C$ for some $C \in [0,\infty)$). If $C>0$ then we write $a_n=\Theta(b_n)$.  If  $a_n/b_n\rightarrow 1$, then we  say $a_n \sim b_n$. 
A sequence of random variable $X_n$ is called $\Omega_{\P}$, $\mathrm{a.s.} \ \P$ if $\exists C>0$ such that $\P(X_n>C) \rightarrow 1$.

\section{Main Results} \label{sec:main_results}
The behavior of the tests $\varphi_{k,\ind}(t), \ind \in \{\into,\out\}$,  crucially depends on the ``strength" of principal components in the population model which in turn is reflective of the strength of population stratification under consideration for the GWAS studies. To capture this, we first present results for fixed  strength of principal components in Section \ref{sec:fixed_strength}. 
Subsequently, in Section \ref{sec:diverging_strength} we also present a sharp phase transition depending on spike strength in  diverging spike strength of population principal components while considering the traditional spiked model  ~\cite{johnstone2001distribution}. The implications of these results to GWAS type problems is discussed in Section \ref{sec:gwas_implications} 
Finally, \ref{sec:gwas_implications} also presents results when the studies come from  a mixture of populations and thereby directly addressing population stratification type issues in GWAS.

\subsection{\bf Fixed Strength of Population PCs}\label{sec:fixed_strength}

We begin by stating and discussing the results for $\mathrm{LR}_{k,\mathrm{ind}}$ fixed stratification strength, i.e., $\lambda_1= O(1)$ whenever $\bbX$ or $\bbW$ follow a generalized spiked distribution (as in Definition \ref{defn:data_mat_generalized spiked}) with top spiked eigenvalue $\lambda_1$. 


\begin{theorem}\label{thm:spiked_model_fixed_strength_out}
 Consider testing \eqref{eqn:main_hypo} under \eqref{eqn:model_gaussian_linear_outcome_regression} using $\varphi_{k,\ind}(t)$. Then the following hold with any fixed $k\geq k^*$ under Assumptions~\ref{assn:spike_and_dimension} (a), (b), and Assumption~\ref{assn:prop_projection} for $\ind =\out$. Moreover, the same conclusion holds under Assumption~\ref{assn:spike_and_dimension} (a), (c) and Assumption~\ref{assn:prop_projection} for $\ind =\into$. 

\begin{enumerate}[label=\textbf{\roman*}.]
\item \label{thm:both_fixed_fixed_strength_out_reg} {\color{black} Suppose that $\bbeta,\btheta$ are fixed effects.  Then for any fixed cut-off $t \in \mathbb{R}$ and fixed $M>0$ one has
\begin{align*}
    & \liminf\limits_{n\rightarrow \infty:\atop p/n \rightarrow \gamma>0}\sup\limits_{\|\bbeta\|,\|\boldsymbol{\theta}\|\leq M}\P_{\delta=0}\left(\mathrm{LR}_{k,\out}>t\right)=1, \\ \text{and} & \liminf\limits_{n\rightarrow \infty:\atop p/n \rightarrow \gamma>0}\sup\limits_{\|\bbeta\|\leq M}\P_{\delta=0}\left(\mathrm{LR}_{k,\into}>t\right)=1.
\end{align*}}

\item \label{thm:one_fixed_fixed_strength_out_reg} \textcolor{black}{Assume that $\bbeta\sim N(\mathbf{0},\sigma^2_{\beta}{I}/p)$. Then $\exists$ function  $\upsilon_{\out}\left(\cdot,\cdot\right):\mathbb{R}^2\rightarrow \mathbb{R}_+$ such that for any $h,t \in \mathbb{R}$ and $\delta=h_1/\sqrt{n}$ and fixed $h_2$ one has  
\begin{align*}
   & \lim \P_{\delta=h/\sqrt{n}}(\mathrm{LR}_{k,\out}>t)= \upsilon_{\out}\left(t,h\right) \quad \forall \btheta \quad \text{s.t.} \quad \|\btheta\|\geq h_2;\\
    \text{and} & \lim \P_{\delta=h/\sqrt{n}}(\mathrm{LR}_{k,\into}>t)= \upsilon_{\into}\left(t,h\right)
\end{align*}}
The function $\upsilon_{\out}$ depends on $\{\gamma,\sigma^2_{y}, \sigma^2_{\beta},k^*,k,\Gamma_{H},\left\{\langle\mathbf{v}_j,\btheta\rangle,\lambda_j\right\}_{j=1}^{k^*}\}$ and the function $\upsilon_{\into}$ depends on $\{\gamma,\sigma^2_{y}, \sigma^2_{\beta},k^*,k,\Gamma_{H},\left\{\lambda_j\right\}_{j=1}^{k^*}\}$. \footnote{the precise forms of the $\upsilon_{\ind}$ can be found in the proof where its dependence on the sequence of tuples $\{\langle\mathbf{v}_j,\btheta\rangle,\{\lambda_j\}_{j=1}^{k^*}\}$ is suitably defined in a limiting sense. Throughout we shall implicitly assume that suitable quadratic forms of the type $\sum_{j=1}^p \psi(\lambda_j)\boldsymbol{\zeta}^Tv_jv_j^T\boldsymbol{\zeta}$ for certain deterministic functions $\psi$ and sequence of vectors $\boldsymbol{\zeta}$ of bounded norm, that arise in the proof of the theorem, exists. Otherwise, our results can be written more tediously in terms of upper and lower bounds on the power function involving liminf's and limsup's of similar quadratic forms.} and is larger than $\alpha$ at $h=0, t=\chi^2_{1}(\alpha)$ (the $\alpha$-quantile of $\chi^2_1$). \textcolor{black}{ A qualitatively similar result holds for either $\bbeta$ non-random and $\btheta$ random or both $\bbeta,\btheta$ random effects} \footnote{with analogous but different limiting functions.}
\end{enumerate}
\end{theorem}

{\color{black}Theorem \ref{thm:spiked_model_fixed_strength_out} demonstrates that under generalized spike model --- specifically, Assumptions~\ref{assn:spike_and_dimension} (a), (b) for $\ind =\out$ and Assumptions~\ref{assn:spike_and_dimension} (a), (c) for $\ind =\into$ ---  Type-I error inflation is inevitable when spike strength is $O(1)$. This inflation exhibits a dichotomy depending on whether fixed-effect or random-effect regression is considered.}


\begin{remark}
Although some levels of Type I error is expected for fixed and low spike strength ~\cite{patterson2006population}, Theorem \ref{thm:spiked_model_fixed_strength_out} implies that the Type I error can be arbitrarily close to $1$ when using any fixed cut-off. This automatically covers any distribution based cut-off. Indeed, $\mathrm{LR}_{k,\ind}$ diverges in this set-up, and as our proof will suggest, this is not an artifact of specific adversarial choices of $\bbeta,\btheta$, but actually holds for uncountably many choices of $\bbeta,\btheta$. Our proof also suggests that the Type I error is somewhat less pathological when $\bbeta$ is orthogonal to the leading population eigenvectors. This provides intuitive justification of the results in the random effects parts of Theorem  \ref{thm:spiked_model_fixed_strength_out} since mean $0$ random $\bbeta$ is ``on average orthogonal" to any fixed directions.  

\end{remark}

\begin{remark} 
In the notation $\P_{\delta}$ we have suppressed the dependence of $\P$ on $\bbeta,\btheta$ to not only maintain succinct notation but also to avoid confusion when one of $\bbeta$ and $\btheta$ is random. Consequently, in statements of the theorem whenever a supremum over $\bbeta$ or $\btheta$ appears, it implicitly acknowledges the dependence of $\P_{\delta}$ on the respective parameters.
\end{remark}

\begin{remark} 
\textcolor{black}{We note that the results for $\mathrm{LR}_{k,\ind}$ does not involve a supremum over $\btheta$. This is because assuming a spiked model on $\bX$ does not allow a free choice of $\btheta$ since by representing $\btheta$ as the population least squares coefficients implies that $\btheta$ is a function of $\Sigma$. Also, we assume $\|\btheta\|\geq h_2>0$ is bounded away from $0$ to reflect a scenario where the effect of $A$ on $Y$ experiences non-trivial confounding through $\bW$.}
\end{remark}

\begin{remark}
The scaling of the variance of the random effects is imposed to maintain finite signal to noise ratio. Moreover,  the results above do not change if one changes the distribution random effects from Gaussian to other suitable sub-Gaussian distributions as long as the leading population eigenvectors $v_1,\ldots,v_{k^{*}}$ are delocalized. However, the results can be quite specific to various possibilities of localizations and distributions of the random effects otherwise. We do not explore this aspect here in detail.  
\end{remark}

\subsubsection{Application in spike model}\label{remark:classical_spiked_model_example}

We consider a simple application of Theorem \ref{thm:spiked_model_fixed_strength_out} in the case where population covariance matrix is the classical spike model ~\cite{johnstone2001distribution} i.e. $$\Sigma=I+\lambda_1 e_1e^\top_1,$$ 
Note that, we require $\lambda_1>\sqrt{\gamma}$ for the spike to be separate from the bulk eigenvalues. Here, we exhibit how \ref{thm:spiked_model_fixed_strength_out} can be used to characterize precise type-I error of the LR test for \eqref{eqn:main_hypo}.

By Lemma \ref{lem:distribution_lr} one has
\be
    \mathrm{LR}_{k, \rm out}|\mathbf{A},\mathbb{W}\sim \chi_1^2(\hat{\kappa}^2_{\rm out}),\quad  \hat{\kappa}^2_{k,\rm out}=\frac{\left((\bbeta^T\mathbb{W}^\top+\mathbf{A}^\top \delta)\left(I-{P}_{\mathcal{C}(\mathbb{W}\widetilde{\mathbf{V}}_{k,\rm out})}\right)\mathbf{A}\right)^2}{\mathbf{A}^\top\left(I-{P}_{\mathcal{C}(\mathbb{W}\widetilde{\mathbf{V}}_{k,\rm out})}\right)\mathbf{A}}.
\ee
Note that,
$$\mathbb{P}_{\delta}(\mathrm{LR}_{k,\out}> t_\alpha)=\mathbb{E}\Big(\bar{\Phi}(t_\alpha-\hat{\kappa}_{k,\out}^2)-\Phi(-t_\alpha-\hat{\kappa}_{k,\out}^2)\Big)^2,$$
where $\Phi$ is the standard normal cdf and $\bar{\Phi}=1-\Phi$. Hence, it will be sufficient to derive asymptotic distribution of $\hat{\kappa}_{k,\out}^2$ followed by an application of uniform integrability principle to derive the asymptotic behavior of the likelihood ratio test.

If $\bbeta, \btheta$ are both fixed effects, we can show by the proof of Theorem \ref{thm:spiked_model_fixed_strength_out} (proof details deferred to Section \ref{sec:proofs}), with probability $1-o(1)$, one has $\hat{\kappa}_{k,\out}^2 \ge Cn$ for some $C>0$. This implies, for \textit{any} fixed cutoff $t_\alpha$, $\mathbb{P}_{\delta}(\mathrm{LR}_{k,\out}> t_\alpha) \rightarrow 1$. This implies for fixed effects, no finite cutoffs can yield type-I error of the LR test to a desired level $\alpha$.

Turning to the random effects regression, set $\bbeta \sim N\left(0,\frac1p \sigma^2_\beta\right)$ and  $\btheta:= e_1$. We have shown, in \eqref{eq:theta_fixed_ncp},
\begin{align*}
    \hat{\kappa}^2_{k,\out} \xrightarrow{D} \frac{\sigma^2_\beta(\sigma^2_g m_1+c_4)}{c_0 +\sigma^2_g}\chi^2(h^2(c_0 +\sigma^2_g)^2):=\mathcal{L},
\end{align*}
where the explicit constants are given by
$$c_0= 
\lim_n\sum_{j=k+1}^{n\wedge p}\hat{\lambda}_j\btheta^T\hat{v_j}\hat{v}_j^T\btheta, \quad c_4=\lim_n \frac{1}{np} \sum_{j=k+1}^{n \wedge p} \hat{d_j}^4 \btheta^T\hat{v_j}\hat{v_j}^T\btheta, \quad m_1=\lim_n \frac1p \sum_{j=k+1}^{n \wedge p} \hat{\lambda}_j,$$ and $m_1=m_{1,\gamma}$ is the mean of Mar\v{c}enko-Pastur distribution \cite{bai2010spectral} where $p/n \rightarrow \gamma$. 

To compute the precise values of $c_0,c_4$, note that the limiting spectral distribution of $\Sigma$ is the non-random distribution which is degenerate at $1$. Now we invoke Lemma \ref{lem:spikes_non_spikes} to obtain that $c_0= (1+o(1))\sum_{j=1}^{n \wedge p} \phi_{1j} \btheta^\top v_jv^\top_j \btheta$ and $c_4= (1+o(1))\frac{1}{\gamma}\sum_{j=1}^{n \wedge p} \phi_{2j} \btheta^\top v_jv^\top_j \btheta$, where
\begin{align*}
    \phi_{1j}&=\begin{cases}
    \gamma \frac{\lambda+1}{\lambda^2},\quad j=1\\
    1,\quad j >1 ,\end{cases}\\
    \phi_{2j}&=\begin{cases}
    \gamma \frac{\lambda+1}{\lambda^2}+ \gamma^2 \frac{(\lambda+1)^2}{\lambda^3},\quad j=1,\\
    1+\gamma, \quad j >1.\end{cases}
\end{align*}

Note that, $\Sigma= (1+\lambda_1) e_1e^\top_1+\sum_{j = 2}^{n \wedge p} e_je^\top_j$. So, $v_j=e_j$, $c_0=\frac{\lambda+1}{\lambda^2}$, and $c_4=\gamma \frac{\lambda+1}{\lambda^2}+ \gamma^2 \frac{(\lambda+1)^2}{\lambda^3}$. Hence,
$$\mathbb{P}_{\delta}(\mathrm{LR}_{k,\out}> t_\alpha) \rightarrow \mathbb{E}\Big(\bar{\Phi}(t_\alpha-\mathcal{L})-\Phi(-t_\alpha-\mathcal{L})\Big)^2.$$

Since the distribution of random variable $\mathcal{L}$ is exactly computed here, one can choose $t_\alpha \in \mathbb{R}$ such that the RHS equals $\alpha$. This shows, although for the random effects model, the classical $\chi^2$-based cutoffs do not control type-I error for the LR tests in $p \propto n$ regime, one can find alternative cutoffs to control the level of the test. We demonstrate that this was not the case for fixed effects model where no fixed finite cutoff worked.

\subsubsection{Implications of Theorem \ref{thm:spiked_model_fixed_strength_out}}
Here, we make a few further remarks on the assumptions and implications of Theorem \ref{thm:spiked_model_fixed_strength_out} regarding the behavior of $\varphi_{k,\ind}(t)$ for $\ind \in \{\into,\out\}$. First, Theorem \ref{thm:spiked_model_fixed_strength_out} (i) implies that under fixed effects model on the regression of $Y|A,\bX$ and $A|\bX$ one has pathological behavior of PC adjustment based tests in the sense that for any fixed cut-off $t$, the LRT following PC adjustments has size converging to $1$ (This was also seen in Example \ref{remark:classical_spiked_model_example}). In Section \ref{sec:numerical_experiments} we further present simulation results towards universality of this phenomenon by verifying that this is indeed not an artifact of the particular choice of regression of $A|\bX$ in our analytical explorations. 

Similarly, Theorem \ref{thm:spiked_model_fixed_strength_out} (ii) demonstrates that although the pathology of the size is somewhat diluted for random effects in the regressions of $Y|A,\bX$ and $A|\bX$, the resulting size is still strictly above the desired level $\alpha$ while using the standard $\chi_1^2(\alpha)$ type cut-offs. Indeed, in view of the proportional asymptotic regime (i.e. $p/n\rightarrow \gamma \in (0,\infty)$) and associated inconsistency of PCA ~\cite{johnstone2009consistency}, one might expect some level of discrepancy in the size and power of the PC adjustment based test $\varphi_{k,\ind}$ considered -- and Theorem \ref{thm:spiked_model_fixed_strength_out} provides precise nature of this discrepancy. 

Subsequently,  it is natural to ask whether one can recover the desired $\alpha$-level by suitably inverting the asymptotic power function of the test $\varphi_{k,\ind}(t)$. In this regard,  we first focus on the case of random $\bbeta$ and a close look at the formulae obtained for the power functions $\upsilon_{\ind}$ in Theorem \ref{thm:spiked_model_fixed_strength_out} reveals their intricate dependence on $\left\{\langle\mathbf{v}_j,\btheta\rangle,\lambda_j\right\}_{j=1}^{k^*}\}$ -- which cannot in general be estimated consistently for generalized spiked models with fixed spike strengths \cite{cai2018rate,perry2018optimality}.

In some cases, when $\upsilon_{\ind}$ does not depend on the projection of $\btheta$ in the directions of population eigenvectors (see e.g. Example \ref{remark:classical_spiked_model_example} for the very special case of classical spiked model ~\cite{johnstone2001distribution} and specific assumption on $\btheta$) estimation of the power function is possible -- however in general this seems to be impossible without assuming further structure on $\btheta$ and population $\mathbf{v}_j$'s. In contrast when $\btheta$ is random, one has irrespective of $\bbeta$ that the power function only depends on the $\Gamma_{H}$ (the population spectral distribution for the distribution of $\bX$ or $\bW$) and hence can be potentially estimated. To keep our discussions focused we do not further pursue this avenue here.

To further argue that the pathologies noted above are not an artifact of Assumption \ref{assn:prop_projection}, we now present analogues of the results from Theorem \ref{thm:spiked_model_fixed_strength_out} for the case when $A|\bW$ is not a linear regression. Indeed this is the case when $A$ is discrete (as is the situation for GWAS) and as a result Theorem \ref{thm:spiked_model_fixed_strength_out} does not apply immediately. However as we will see that the qualitative nature of the problem conveyed by Theorem \ref{thm:spiked_model_fixed_strength_out} continues to hold as specified by the lower bounds provided by our next result.  

\begin{theorem}\label{thm:thm:spiked_model_fixed_strength_out_nogaussian_a}
  Consider testing \eqref{eqn:main_hypo} under \eqref{eqn:model_gaussian_linear_outcome_regression} using $\varphi_{k,\ind}(t)$. Also assume that $k\geq k^*$ is fixed and Assumptions 2.6 (a) and 2.8 hold. 
  
  \begin{enumerate}[label=\textbf{\roman*}.]
\item \label{spiked_model_fixed_strength_out_nogaussian_a_out} For fixed effect $\beta$, the following holds under Assumption 2.6 (b) for any $M>0$.
      \begin{align*}
         & \liminf\limits_{n\rightarrow \infty:\atop p/n \rightarrow \gamma>0}\sup\limits_{\|\bbeta\|\leq M}\P_{\delta=0}\left(\mathrm{LR}_{k,\out}>\chi_1^2(\alpha)\right)>\alpha.
      \end{align*}
      
\item \label{spiked_model_fixed_strength_random_effects_out_nogaussian_a_out} Suppose $\bbeta\sim N(0,\sigma^2_{\beta}I/p)$ and that Assumption 2.6(b) holds. Then 
\begin{align*}
    \lim\limits_{n\rightarrow \infty:\atop p/n \rightarrow \gamma>0} \P_{\delta=0}(\mathrm{LR}_{k,\out}>\chi_1^2(\alpha))>\alpha. 
\end{align*}
  
\item \label{spiked_model_fixed_strength_out_nogaussian_a_in} Suppose that Assumption 2.6 (c) holds. Let $\bbeta_0:=I_{-1}\bbeta$ with $I_{-1}=[e_2:\ldots:e_{p+1}]$ and  

\begin{equation} \label{eq:define_cstar}
c^*_p(\bbeta_0)=\sum_{j=1}^{p}\phi_{1j}\bbeta^\top_0v_jv_j^\top e_1
\end{equation} with $\phi_{1j}$ defined in Lemma \ref{lem:spikes_non_spikes} and $e_j, j\geq 1$ canonical basis $\mathbb{R}^{p+1}$.
\begin{enumerate}
    \item If $\bbeta$ is such that $\liminf |c^*_p(\bbeta_0)|> 0$ then for any $t\in \mathbb{R}$
     \begin{align*}
         \liminf\limits_{n\rightarrow \infty:\atop p/n \rightarrow \gamma>0}\P_{\delta=0}\left(\mathrm{LR}_{k,\into}>t\right)=1.
      \end{align*}
      
      \item If $\bbeta$ is such that $\limsup |c^*_{p}(\bbeta_0)|=0$ and $\liminf |\sqrt{p} c^*_p(\bbeta_0)|>0$ then
      
      \begin{align*}
         & \liminf\limits_{n\rightarrow \infty:\atop p/n \rightarrow \gamma>0}\P_{\delta=0}\left(\mathrm{LR}_{k,\into}>\chi_1^2(\alpha)\right)>\alpha.
      \end{align*}
      
      \item If $\bbeta$ is such that $\limsup |\sqrt{p} c^*_p(\bbeta_0)|=0$ then the following holds if $\liminf \lambda_p^*>0$.
      
      \begin{align*}
         & \limsup\limits_{n\rightarrow \infty:\atop p/n \rightarrow \gamma>0}\P_{\delta=0}\left(\mathrm{LR}_{k,\into}>\chi_1^2(\alpha)\right)=\alpha.
      \end{align*}
\end{enumerate} 
\item \label{spiked_model_fixed_strength_random_effects_out_nogaussian_a_in} Suppose $\bbeta\sim N(0,\sigma^2_{\beta}I/p)$ and that Assumption 2.6(b) holds. Then 
\begin{align*}
    \lim\limits_{n\rightarrow \infty:\atop p/n \rightarrow \gamma>0} \P_{\delta=0}(\mathrm{LR}_{k,\into}>\chi_1^2(\alpha))>\alpha.
\end{align*}

\end{enumerate}
\end{theorem}

The above Theorem shows Type-I error inflation under generalized spike model (Assumption~\ref{assn:spike_and_dimension}) for $A|\bW$ is not a linear regression. This implies that it is necessary to modify the traditional $\chi^2$-based cut-offs for testing using $\mathrm{LR}_{k,\out}$. Moreover, the result for $\mathrm{LR}_{k,\into}$ can remain pathological owing to the fact of double use of $A$ in both as a regressor for $Y$  as well as the in the PC's used in the same regression. 

To provide insight about the nature of the quantity driving the behavior of $\mathrm{LR}_{k,\into}$, we elaborate using the special case of classical spiked model with $\Sigma=I+\lambda v v^T$. The behavior of the test depends on the localization structure of $v$, which is captured via $c^\star_p(\bbeta_0)$ \eqref{eq:define_cstar}. 

First if $v\perp e_1$ then $c_p^*(\bbeta_0)=0$ for any $\bbeta_0$ and hence a $\chi_1^2$ cut-off is valid for $\mathrm{LR}_{k,\into}$. This is indeed expected since $v\perp e_1$ implies no confounding between $Y$ and $A$ through $\bW$. 

Suppose now $\liminf|\langle v,e_1\rangle |>0$ and $\liminf|\langle v,\bbeta\rangle |>0$. Now, we land in the situation of part (a) of  Theorem \ref{thm:thm:spiked_model_fixed_strength_out_nogaussian_a}\ref{spiked_model_fixed_strength_out_nogaussian_a_in} (a) since using the fact that $\bbeta_0^Te_1=0$ in this case we have $\liminf |c_p^*(\bbeta_0)|=\liminf |(\phi_{11}-\phi_{12})\bbeta_0^Tvv^Te_1 +\bbeta_0^Te_1|=\liminf |(\phi_{11}-\phi_{12})\bbeta_0^Tvv^Te_1|>0$. 

Finally if $|\langle v,e_1 \rangle|=\Theta(1/\sqrt{p})$ we can appeal to Theorem \ref{thm:thm:spiked_model_fixed_strength_out_nogaussian_a}\ref{spiked_model_fixed_strength_out_nogaussian_a_in} (b) and a Type error inflation occurs while using traditional $\chi_1^2$ cut-off.
\subsection{\bf Implications for GWAS type Analyses}\label{sec:gwas_implications}

As mentioned in the introduction, the method analyzed above has especially gathered immense popularity in GWAS type analysis where adjusting for variables in $\bX$ crucially adjusts for population stratification and unmeasured confounding ~\cite{price2006principal}.  In particular, 
a classical problem of modern GWAS is that of testing for individual genetic variants while adjusting for other variants and environmental factors ~\cite{visscher201710}. 

One of the main challenges in modern genetic association studies is population stratification. In particular, the presence of population sub-structure where the dataset consists of several sub-populations with different phenotypic means and varying allele frequencies can confound the relationship between disease and genes. As a result, many commonly used association tests may be severely biased.

To address this issue, several methods have been developed to account for population structure. These methods can be organized into three broad themes:
\begin{enumerate}
    \item[(i)]
    Methods that use linear combinations of  other genetic markers (which capture population sub-structure in their distribution) 
as covariates in the analysis using linear or logistic regression. These markers serve as surrogates for the underlying strata \cite{chen2003qualitative,zhang2003semiparametric,zhu2002association,price2006principal}. Owing to the seminal paper of \cite{price2006principal} have emerged as the most popular of these available methods. This is the method we are analyzing in the paper in high-dimensional covariates regime.

    \item[(ii)] Method of genomic control ~\cite{devlin1999genomic}, which addresses the effects that unknown  population-admixture might have on the variance of commonly used test test statistics.
    
    \item[(iii)] Methods that rely on a model and data
on the additional genetic markers ~\cite{pritchard2000inference} to infer the latent population structure and subsequently incorporates this model into down stream analysis.
\end{enumerate}

The implications of Theorem \ref{thm:mixture_model_fixed_strength} can be best understood through its connection to a simple population admixture model as follows. Suppose that $\bX_i =(X_{i1},\ldots,X_{ip})$ is such that there exists $S\subset \{1,\ldots,p\}$ coordinates where the distributions of $X_{ij},j\in S$ are a mixture between two population and for $j\in S^c$ the $X_{ij}$'s are same between the two populations. 

For studying GWAS type problems, under LD pruning ~\cite{hartl1997principles, laird2011fundamentals} of the genotypes, one such situation can be idealized through independently having $X_{ij}\sim \frac{1}{2}\mathrm{Bin}(2,q_{1j})+\frac{1}{2}\mathrm{Bin}(2,q_{2j})$ with $q_{1j}=q_{2j}=q$ for $j\in S$ and $q_1=q_{1j}\neq q_{2j}=q_2$ for $j\in S^c$. It is easy to check that in this case, $\mathrm{Var}(\bX)=\sigma^2 I+(q_1-q_2)^2|S|\mathbf{v}\mathbf{v}^T$ for some $\sigma^2>0$ and unit vector $\mathbf{v}$. Consequently, 
the covariance matrix of $\bX$ indeed follows a spiked model with spike strength proportional to $|S|$.

Theorem \ref{thm:diverging_spike_out} pertains to the traditional spiked model as an intuitive simplification from the motivating mixture model described above and establishes precise phase transition for the performance of EIGENSTRAT type procedures based on population stratification strength -- as captured by the leading spike strength. Indeed, this provides an useful benchmark to check the validity of using PC adjustment based testing for low dimensional parameters.

Our next result is targeted towards showing that we can translate the results from the generalized spike model set up to a more relevant case of mixture models on the rows of $\mathbb{X}$. In particular, it is natural to assume such a mixture distribution for the genotypes in genetic association studies while considering population admixture models ~\cite{price2006principal}. 

\begin{theorem}\label{thm:mixture_model_fixed_strength}
{\color{black} Consider testing \eqref{eqn:main_hypo} under \eqref{eqn:model_gaussian_linear_outcome_regression} using $\varphi_{k,\ind}(t)$. Then the same conclusion as of Theorem \ref{thm:spiked_model_fixed_strength_out} holds with any fixed $k\geq 1$ under Assumptions 2.6 (a), and 2.7 for $\ind =\out$ when $\mathbf{W} \sim \sum_{i=1}^{L} w_i F(\bmu_i,\Sigma)$ for any fixed $L$ with $w_i \ge 0$, with $\sum_{i=1}^{L} w_i =1$, $\sup_i \|\mu_i\| \le 1$, $\|\Sigma\|=O(1)$  where $F(\bmu,\Sigma)$ is the distribution of a random vector $\Sigma^{1/2}\mathbf{Z}+\bmu$ where $\mathbf{Z}$ is a random vector with mean zero i.i.d. sub-Gaussian coordinates. When $\ind = \into$, the same conclusion holds under Assumption 2.6 (a) when $\mathbf{X} \sim \sum_{i=1}^{L} w_i F(\bmu_i,\Sigma)$. }

\end{theorem}
It is worth noting that mixture models on $\bX$ implies spiked models on the variance covariance matrix of $\bX$ and thereby the above result is somewhat intuitive given the validity of Theorem \ref{thm:spiked_model_fixed_strength_out}. Moreover, in our numerical experiments we shall additionally verify the universality of the results for non-mean-shift type mixture models on $\bX$ (e.g. mixtures of binomial distributions on the coordinates of $\bX$ as more natural candidates for modeling genotypes under admixture of populations under Hardy-Weinberg equilibrium).

\subsection{\bf Diverging Strength of Population PCs}\label{sec:diverging_strength}
Section \ref{sec:fixed_strength} explored the necessity of potential caution while using EIGENSTRAT type procedure for PC adjustment testing of low-dimensional components. However, these results are not able to reconcile the empirical success and popularity of this procedure. Moreover, the results might suggest that the primary issue may be the bias in eigenvectors when spike strengths are fixed. Indeed, the fact that PC based analysis of genetic variations have been able to reveal known population structures ~\cite{novembre2008genes}leads one to believe that diverging spike strengths may be key to the success of PC adjustment-based testing procedures.

However, as we will demonstrate, simply having diverging spike strengths is insufficient to account for the anomalies noted earlier. Instead, a specific signal strength can be calibrated to identify the regimes in which EIGENSTRAT-type procedures are effective.

In this regard, we  now present first results in this direction to derive necessary and sufficient conditions on the success of this procedure under  the classical spiked model \cite{johnstone2001distribution} with diverging spike strengths. In the regime of diverging signal strength, it turns out that we can additionally take into account the effect of the distances between $\bbeta,\btheta$ and the population spike eigenvectors. This seems to be a reasonable consideration to explore since using the sample PC directions as the regressors finds its natural analogue in the population when $\bbeta$ belongs to (or is close to) the linear span of the population spiked eigenvectors.  

To this end, we will use the following natural notion of distance between vectors and sub-spaces. 
In particular,  given a set $V$ and a fixed vector $\beta$ we define $d(\beta,V)=1-\frac{\|P_{V}(\bbeta)\|^2}{\|\bbeta\|^2}$ (with $P_{V}$ denoting the orthogonal projection operator onto $V$) and 

\begin{equation}\label{eq:define_c_tau}
    \mathcal{C}_\tau(V)=\{\beta:   \|\beta\|\le 1, d(\beta,V)\le p^{-\tau}\}.
\end{equation}
Note that, if $\tau_1>\tau_2$, $\mathcal{C}_{\tau_1}(V) \subseteq \mathcal{C}_{\tau_2}(V)$ and larger $\tau$ implies lesser angle of the vector $\bbeta$ with elements of $V$. 

Finally, the results in this case for $\mathrm{LR}_{k,\ind}$ additionally depend on the nature of localization of the coordinates of the spikes. This is intuitive since certain natures of localized coordinates implies a highly strong effect of coordinates of $\bW$ on $A$ whereas delocalized eigenvectors imply otherwise. To avoid confusion, we therefore first present the results for $\mathrm{LR}_{k,\out}$.  


\begin{theorem}\label{thm:diverging_spike_out}
Consider testing \eqref{eqn:main_hypo} under \eqref{eqn:model_gaussian_linear_outcome_regression} using $\varphi_{k,\out}(t)$ and suppose that Assumptions \ref{assn:spike_and_dimension} (a), (b'') and Assumption \ref{assn:prop_projection} hold. 
Further assume that $\lambda_1,\lambda_{k^*}= \Theta(p^{\tau_0})$ with $\tau_0>0$. Then the following hold for any $\tau>0$
\begin{enumerate}[label=\textbf{\roman*}.]
\item \label{thm:both_fixed_diverging_spike_out} Suppose that $\bbeta, \btheta$ are not random.
\begin{enumerate}
    \item If $\min\{\tau_0,\tau\}<1/2$, then for any fixed cut-off $t\in \mathbb{R}$ and $\gamma<1$,
    \[\liminf\limits_{n\rightarrow \infty:\atop p/n \rightarrow \gamma>0} \sup_{\bbeta,\btheta \in \mathcal{C}_{\tau}(S_{V_{k^*}})} \mathbb{P}_{\delta=0}(\mathrm{LR}_{k^*,\out}>t)=1.
    \]
    \item If $\min\{\tau_0,\tau\} >1/2$, then for any $k\geq k^*$ 
     \[\liminf\limits_{n\rightarrow \infty:\atop p/n \rightarrow \gamma>0} \sup_{\bbeta,\btheta \in \mathcal{C}_{\tau}(S_{V_{k^*}})} \mathbb{P}_{\delta=0}(\mathrm{LR}_{k,\out}>\chi^2_1(\alpha))=\alpha.
    \]
\end{enumerate}
\item \label{thm:beta_random_diverging_spike_out}
{\color{black}Suppose $\bbeta \sim N(0, \frac{1}{p}\sigma^2_\beta I_p)$. Then for any $\tau_0\geq 0$ and $\gamma<1$ one has
\[\liminf\limits_{n\rightarrow \infty:\atop p/n \rightarrow \gamma>0} \sup_{\btheta \in \mathcal{C}_{\tau}(S_{V_{k^*}})} \mathbb{P}_{\delta=0}(\mathrm{LR}_{k^\star,\out}>\chi^2_1(\alpha))>\alpha.
    \]
  Further, if $\btheta \sim N(0,\frac{1}{p}\sigma^2_\theta I_p)$,
  \[\liminf\limits_{n\rightarrow \infty:\atop p/n \rightarrow \gamma>0} \mathbb{P}_{\delta=0}(\mathrm{LR}_{k^\star,\out}>\chi^2_1(\alpha))>\alpha.
    \].}
   
\item \label{thm:theta_random_beta_fixed_diverging_spike_out}
Assume $\btheta \sim N(0,\frac{1}{p}\sigma^2_\theta I_p)$ but $\bbeta$ is not random.
If $\tau_0 >0$, then 
     \[\liminf\limits_{n\rightarrow \infty:\atop p/n \rightarrow \gamma>0} \sup_{\bbeta \in \mathcal{C}_{\tau}(S_{V_{k^*}})} \mathbb{P}_{\delta=0}(\mathrm{LR}_{k,\out}>\chi^2_1(\alpha))=\alpha.
    \]
\end{enumerate}
Above $S_{V_{k^*}}$ refers to the subspace spanned by $v_1,\ldots,v_{k^*}$.
\end{theorem}

Theorem~\ref{thm:diverging_spike_out} describes the phase transitions for out-regression when spike strength diverges. Note that the assumption on $\bbW$ is weaker than the spiked model assumption. We show, for fixed-effects regression, Type-I error inflation occurs when $\min{\tau_0, \tau} > \frac{1}{2}$. Here, the spike strength is given by $\lambda_1 = p^{\tau_0}$, and $\tau$ represents the angle between $\beta$ and the top $k^\star$ eigenvectors (defined by \eqref{eq:define_c_tau}). Importantly, the phase transitions in the Type-I error of $\varphi_{k,\out}$ depend not only on whether $\btheta$ and $\bbeta$ are random effects but also on the value of $\tau$.

Note that, Theorem \ref{thm:diverging_spike_out} provides a precise phase transition when $\tau > 0$, which imposes restrictions on $\bbeta$ and $\btheta$. However, if $\tau = 0$, meaning $\bbeta$ and $\btheta$ are unrestricted, our next result demonstrates the absence of a phase transition.
\begin{prop}\label{thm:tau_zero_out_reg}
Consider testing \eqref{eqn:main_hypo} under \eqref{eqn:model_gaussian_linear_outcome_regression} using $\varphi_{k,\out}(t)$.  Further assume that $\lambda_1= \Theta(p^{\tau_0})$, $\tau_0>0$ and $\gamma <1$. Then the following hold under Assumptions 2.6 (a), (b'):
 \be 
 \liminf\limits_{n\rightarrow \infty:\atop p/n \rightarrow \gamma>0} \sup_{\bbeta,\btheta \in \mathcal{C}_{0}(S_{V_{k^*}})} \mathbb{P}_{\delta=0}(\mathrm{LR}_{k^\star,\out}>\chi^2_1(\alpha))>\alpha.
 \ee
 Further, when $\bbeta \sim N(0, \frac1p \sigma^2_\beta I_p)$, the following holds under Assumptions 2.6 (a), (b)':
 \be 
 \liminf\limits_{n\rightarrow \infty:\atop p/n \rightarrow \gamma>0} \sup_{\btheta \in \mathcal{C}_{0}(S_{V_{k^*}})} \mathbb{P}_{\delta=0}(\mathrm{LR}_{k^\star,\out}>\chi^2_1(\alpha))>\alpha.
 \ee
  If $\bbeta \sim N(0,\frac{1}{p}\sigma^2_\beta I_p)$ and $\btheta \sim N(0,\frac{1}{p}\sigma^2_\theta I_p)$, then  \[\liminf\limits_{n\rightarrow \infty:\atop p/n \rightarrow \gamma>0} \mathbb{P}_{\delta=0}(\mathrm{LR}_{k^\star,\out}>\chi^2_1(\alpha))>\alpha.
    \]
However, if $\btheta \sim N(0,\frac{1}{p}\sigma^2_\theta I_p)$, for any $\gamma >0$,
     \[\liminf\limits_{n\rightarrow \infty:\atop p/n \rightarrow \gamma>0} \sup_{\bbeta \in \mathcal{C}_{0}(S_{V_{k^*}})} \mathbb{P}_{\delta=0}(\mathrm{LR}_{k,\out}>\chi^2_1(\alpha))=\alpha.
    \]
\end{prop}

This concludes our analysis of out-regression under diverging spike strength. Specifically, the Type-I error remains at the nominal level $\alpha$ only when $\btheta$ is treated as random effects and $\bbeta$ is unrestricted. Our next result explores the extent to which our findings depend on the assumption of linear regression (Assumption \ref{assn:prop_projection}) and Gaussianity. 

\begin{prop}\label{hm:diverging_spike_out_nongaussian}
Consider testing \eqref{eqn:main_hypo} under \eqref{eqn:model_gaussian_linear_outcome_regression} using $\varphi_{k,\out}(t)$.  Further assume that $\lambda_1= \Theta(p^{\tau_0})$ and $\tau,\tau_0>0$ and $\gamma<1$. Then the following hold under Assumptions 2.6 (a), (b'), and 2.8 whenever $\min\{\tau,\tau_0\}<\frac{1}{2}$.
 \be 
 \liminf\limits_{n\rightarrow \infty:\atop p/n \rightarrow \gamma>0} \sup_{\bbeta \in \mathcal{C}_{\tau}(S_{V_{k^*}})} \mathbb{P}_{\delta=0}(\mathrm{LR}_{k^\star,\out}>\chi^2_1(\alpha))>\alpha.
 \ee
 {\color{black} If $\bbeta \sim N(0, \frac{1}{p}\sigma^2_\beta I_p)$. Then one has
\[\liminf\limits_{n\rightarrow \infty:\atop p/n \rightarrow \gamma>0} \sup_{\btheta \in \mathcal{C}_{\tau}(S_{V_{k^*}})} \mathbb{P}_{\delta=0}(\mathrm{LR}_{k^\star,\out}>\chi^2_1(\alpha))>\alpha.
    \]}
\end{prop}
The result shows for generalized spike model, the Type-I error inflation is an universal phenomena persisting beyond linear regression model.

Now we turn our attention to in-regression $\mathrm{LR}_{k,\into}$ under divergence spikes. As noted above, this behavior is subtle and depends on the nature of localization of the coordinates of the leading eigenvectors of $\Sigma$. Since the complete characterization of all the possible cases that can arise is beyond the scope of the paper, we focus on a simple special case of single spiked model below with completely delocalized spiked eigenvector.

\begin{theorem}\label{thm:diverging_spike_delocalized}
Consider testing \eqref{eqn:main_hypo} under \eqref{eqn:model_gaussian_linear_outcome_regression} using $\varphi_{k,\into}(t)$ and suppose that Assumptions 2.6 (a), (c'), and 2.7 hold.

\begin{enumerate}[label=\textbf{\roman*}.]
\item \label{thm:diverging_spike_fixed_effects_in} Assume that $k^*=1$, $\lambda_1= \Theta(p^{\tau_0})$ with $\tau_0>0$ and that $\inf_j |v_1(j)|,\sup_{j} |v_1(j)|=\Theta(1/\sqrt{p})$. Then the following hold with any fixed $k\geq 1$ and fixed $M>0$ 
\be 
 \liminf\limits_{n\rightarrow \infty:\atop p/n \rightarrow \gamma>0} \sup_{\bbeta\in S_{v_1}:\|\bbeta\|\leq M} \mathbb{P}_{\delta=0}(\mathrm{LR}_{k,\into}>\chi^2_1(\alpha))=\alpha.
 \ee
 
 \item \label{thm:diverging_spike_random_effects_in}
 If $\bbeta \sim N(0,\frac{\sigma^2_\beta}{p}I_p)$,
 \be 
 \liminf\limits_{n\rightarrow \infty:\atop p/n \rightarrow \gamma>0}  \mathbb{P}_{\delta=0}(\mathrm{LR}_{1,\into}>\chi^2_1(\alpha))>\alpha.
 \ee
\end{enumerate}
\end{theorem}

The result shows Type-I error inflation persists for random effect regression. However, if the top eigenvector is decolalized, i.e., $\|v_1\|_\infty= \Theta(1/\sqrt{p})$, Type-I error is controlled at nominal level $\alpha$. This finding indicates that, in the context of diverging spikes, using matrix $A$ twice (once in the regression of $Y$ on $A$ and again for adjusting confounding through PCA) yields better control of Type-I error compared to $\mathrm{LR}_{k,\out}$. This contrasts with the existing intuition, which primarily focuses on the power of the procedures while overlooking Type-I error considerations~\cite{mai2021understanding,listgarten2012improved,yang2014advantages}.

\section{Numerical Experiments:}\label{sec:numerical_experiments} In this section, we present detailed numerical experiments to not only verify our theoretical results but also to provide additional evidence on the potential universality of the main narrative behind our results that persists beyond the working assumptions of our theoretical results.

 We first present results on tests based on $\mathrm{LR}_{k,\out}$. Our simulations are set up to explore the following features of PC adjustment based procedure for testing \eqref{eqn:main_hypo} under true underlying model \eqref{eqn:model_gaussian_linear_outcome_regression} by varying the following features: (i) marginal distribution of $\bW$ and associated strength of confounding as captured by the spiked strength in case of the generalized spiked model or a suitable notion of mixing strength for mixture models on $\bW$; (ii) conditional distribution of $A|\bW$; (iii) aspect ratio between the sample size and the number of confounding variables; (iv) the nature of the regression of $Y|A,\bW$ (i.e. fixed or random effects as driven by the nature of $\bbeta$ in model \eqref{eqn:model_gaussian_linear_outcome_regression}); and (v) the nature of the regression of $A|\bW$ (i.e. once again fixed or random effects as driven by the nature of $\btheta$ in model \eqref{assn:prop_projection}). Below we first provide a description of the variations that different types of cases we consider in this regard.

\begin{itemize}
    \item \textbf{Distribution of $\bW$:} We shall consider two broad sub-cases in this regard with further variations as follows.
    \begin{enumerate}
        \item \textbf{Spiked Model:} In this case we will consider $\bW_i\sim N(0,\Sigma)$ with $\Sigma=I+\sum_{j=1}^{k^*}\lambda_jv_jv_j^T$ with $\lambda_1\geq \ldots,\lambda_k^*>0$. We shall consider $k^*=1$, and $\lambda_1=p^{\tau_0}$ to regulate the spike strength, where $\tau_0\in\{0,0.05,0.10,\ldots,1\}$.
        
        \item \textbf{Mixture Model:} We shall consider two types mixture models: (i) (\textbf{Gaussian Mixture Model}) $\bW_i\sim \frac{1}{2}N(\bmu_1,I)+\frac{1}{2}N(\bmu_2,I)$ with $\bmu_i=(\mu_{i1},\ldots,\mu_{ip})$ for $i=1,2$, and $\mu_{1j}=-\mu_{2j}=1$ for $j=1,\ldots,m$ and $\mu_{1j}=\mu_{2j}=0$ for $j\geq m+1$; and (ii) (\textbf{Binomial Mixture Model}) $\bW_i=(W_{i1},\ldots,W_{ip})$ with $W_{ij}\sim \frac{1}{2}\mathrm{Bin}(2,q_{1j})+\frac{1}{2}\mathrm{Bin}(2,q_{2j})$ with $q_{1j}=0.3$, $q_{2j}=0.7$ for $j=1,\ldots,m$ and $q_{1j}=q_{2j}=0.5$ for $j\geq m+1$. For this model, we vary $m=\lceil{p^{\tau_0}\rceil}$ to regulate the strength of the population stratification, where  $\tau_0\in\{0,0.05,0.10,\ldots,1\}$, and $\lceil\cdot\rceil$ denotes the ceiling function.
    \end{enumerate}
    
    \item \textbf{Regression of $A|\bW$:} We shall consider two cases: (i) (\textbf{Continuous Exposure}) $A_i\in \mathbb{R}$ with $A_i=\bW_i^T\btheta+\eta_i$ where $\eta_i\sim N(0,\sigma_{a}^2)$; and (ii) (\textbf{Binomial Exposure}) $A_i\sim \mathrm{Bin}(2,p_i)$ with $p_i=\frac{\exp(\btheta^T\bW_i)}{1+\exp(\btheta^T\bW_i)}$. We take $\sigma_a^2=1$ for our simulations.

    \item \textbf{Aspect Ratio:} We shall consider two different aspect ratios $\gamma\in\{0.5,2\}$ between the sample size and the number of confounding variables in our simulations. In all our simulations, we shall let $p=1000$ (and thus the sample size $n\in\{500,2000\}$).
    
    \item \textbf{Regression Coefficients for $Y|A,\bW$:} We shall consider both fixed and random effects for $\bbeta$ in the model \eqref{eqn:model_gaussian_linear_outcome_regression}. When $\bbeta$ is a fixed effect, we shall vary the angle/distance metric $d(\bbeta,S_{v_1})$ by setting $\bbeta=av_1+\sqrt{1-a^2} v_2$, where $v_1$ and $v_2$ are the eigenvectors corresponding to the largest and second largest eigenvalues of $\E(\bW^\top \bW/n)$, and $a=1-p^{-\tau_0}$. The specification of $\tau_0$ remains the same as described in the context of varying spike strength. When $\bbeta$ is a random effect, $\bbeta$ is drawn randomly from $N_p(0,1/pI)$.
    
    \item \textbf{Regression Coefficients for $A,\bW$:} We shall consider both fixed and random effects for $\btheta$ in the model \eqref{assn:prop_projection}. The specifications of $\btheta$ as fixed and random effects are exactly the same as described above for $\bbeta$.
\end{itemize}
We now present our numerical experiments which delves into combinations of the setups presented above. We consider three different combinations of the distributional assumptions of $\bW$ and $A|\bW$ -- (1) Spiked Model on $\bW$ and continuous exposure $A$; (2) Gaussian mixture model on $\bW$ and and continuous exposure $A$; and (3) Binomial mixture model on $\bW$ and binomial exposure $A$. In each of these cases and for all combinations of different aspect ratios and different specifications of $\bbeta$ and $\btheta$, we simulate the outcomes under the null model (i.e. $\delta=0$ in \eqref{eqn:model_gaussian_linear_outcome_regression}) by taking $\sigma^2_{\varepsilon}=1$. We performed 2000 replications of the dataset and tested $\varphi_{k,\out}(\chi_1^2(\alpha))$ on each dataset by taking $k=1$ and $\alpha=0.05$. 
Finally, we plot the empirical type I error rate of the test across varying $\tau_0$ in Figure~\ref{fig:out_divspike}.
\begin{figure}
\centering
\begin{subfigure}{5 cm}
    \includegraphics[width=5 cm]{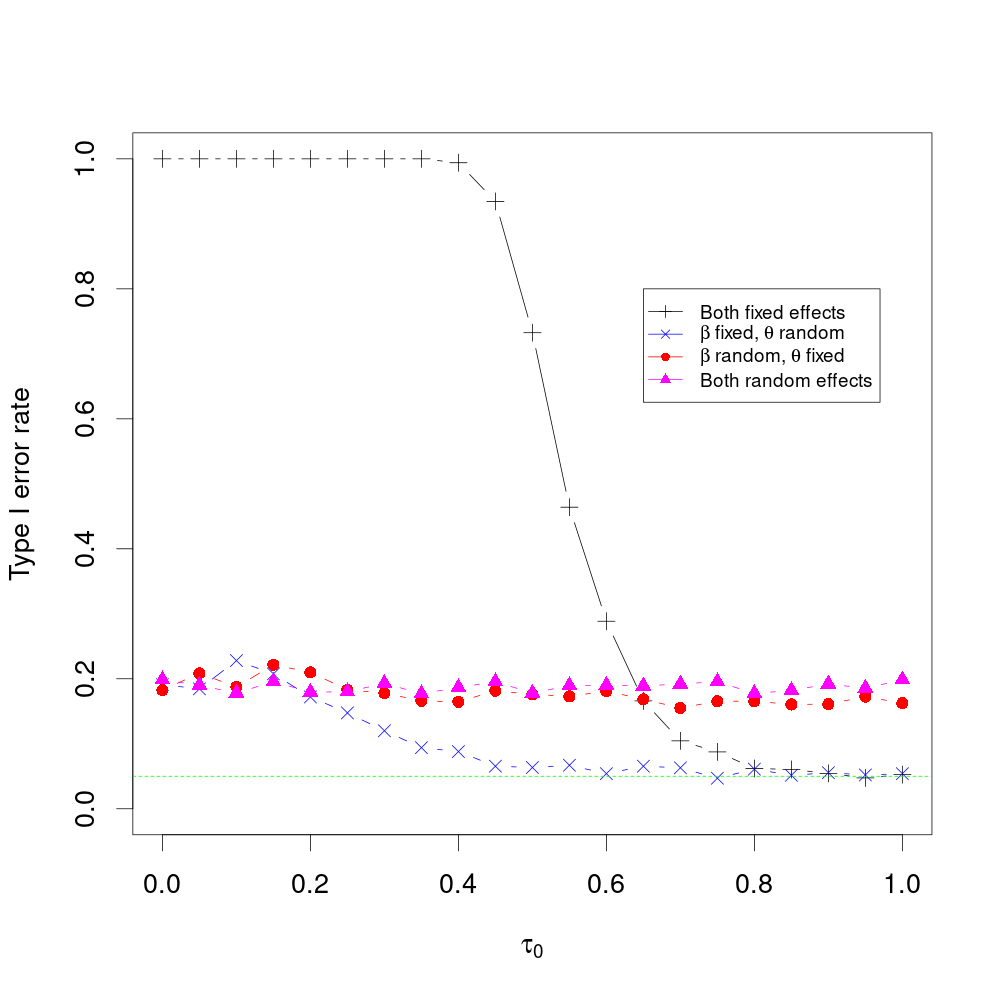}
    \caption{Spiked model, continuous exposure\\($\gamma=2$)}
    \label{fig:divspike_1000_500_gauss_spike_FALSE}
\end{subfigure}
\hfill
\begin{subfigure}{5 cm}
    \includegraphics[width=5 cm]{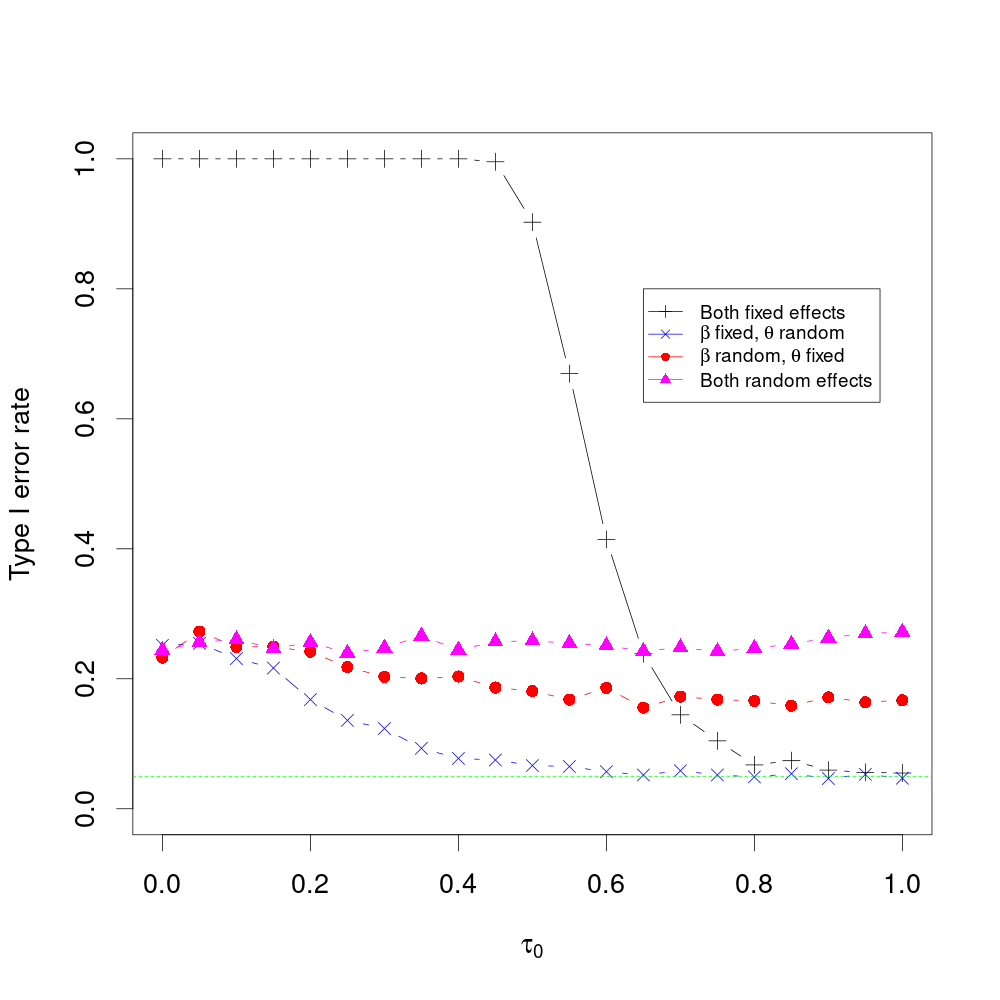}
    \caption{Spiked model, continuous exposure\\($\gamma=0.5$)}
    \label{fig:divspike_1000_2000_gauss_spike_FALSE}
\end{subfigure}
\hfill
\begin{subfigure}{5 cm}
    \includegraphics[width=5 cm]{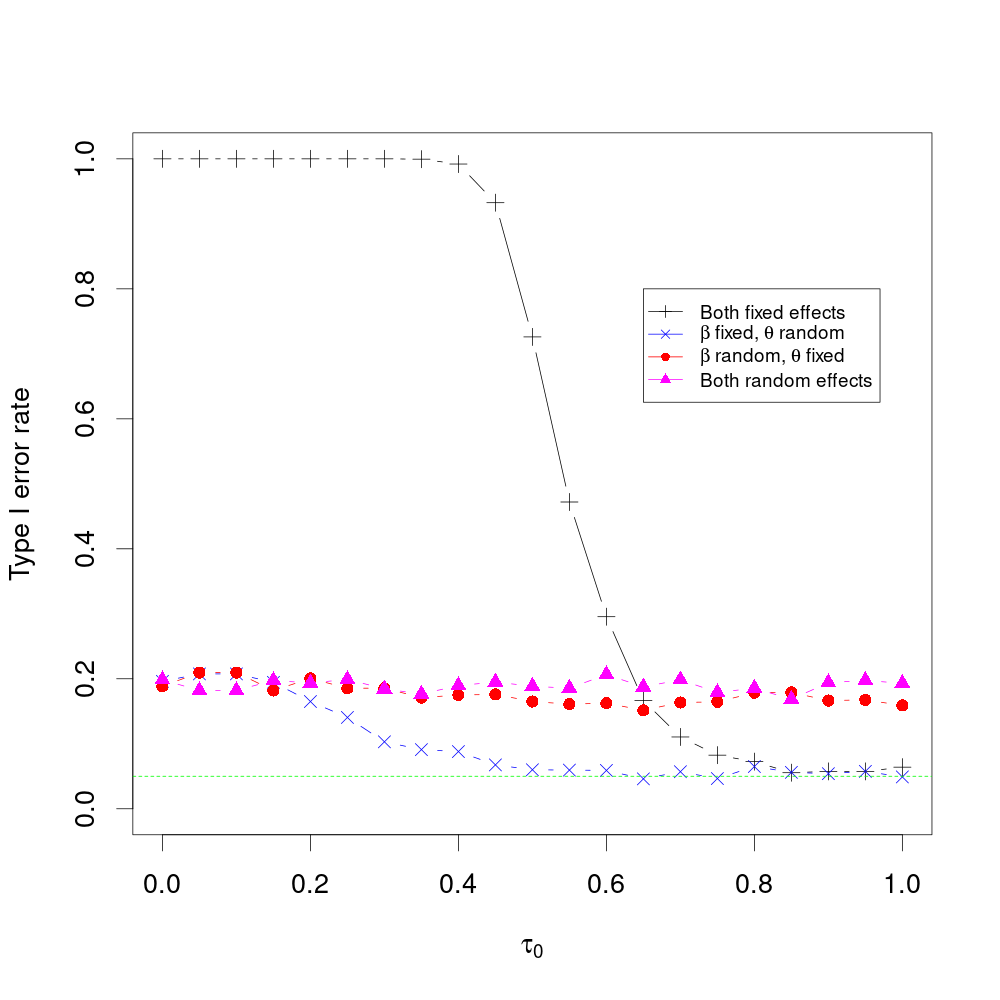}
    \caption{Gaussian mixture, continuous exposure\\($\gamma=2$)}
    \label{fig:divspike_1000_500_gauss_mix_FALSE}
\end{subfigure}
\hfill
\begin{subfigure}{5 cm}
    \includegraphics[width=5 cm]{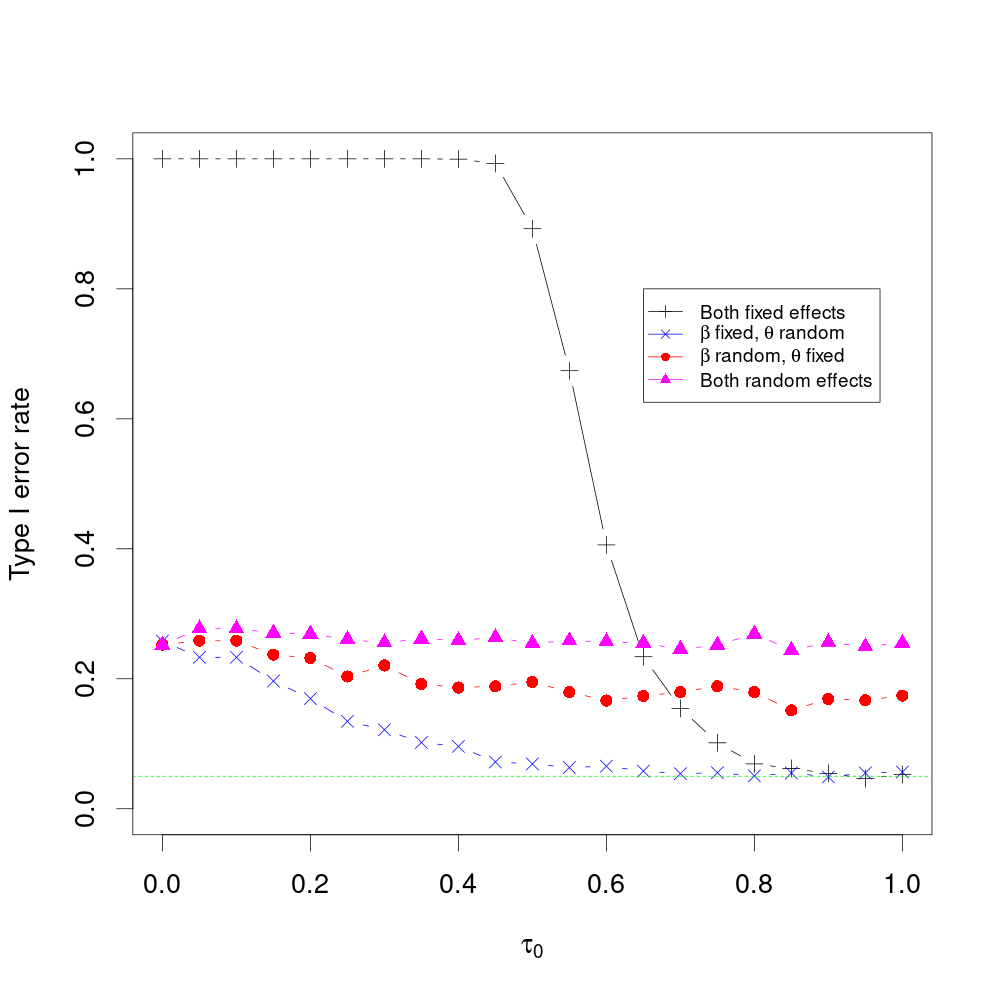}
    \caption{Gaussian mixture, continuous exposure\\($\gamma=0.5$)}
    \label{fig:divspike_1000_2000_gauss_mix_FALSE}
\end{subfigure}
\hfill
\begin{subfigure}{5 cm}
    \includegraphics[width=5 cm]{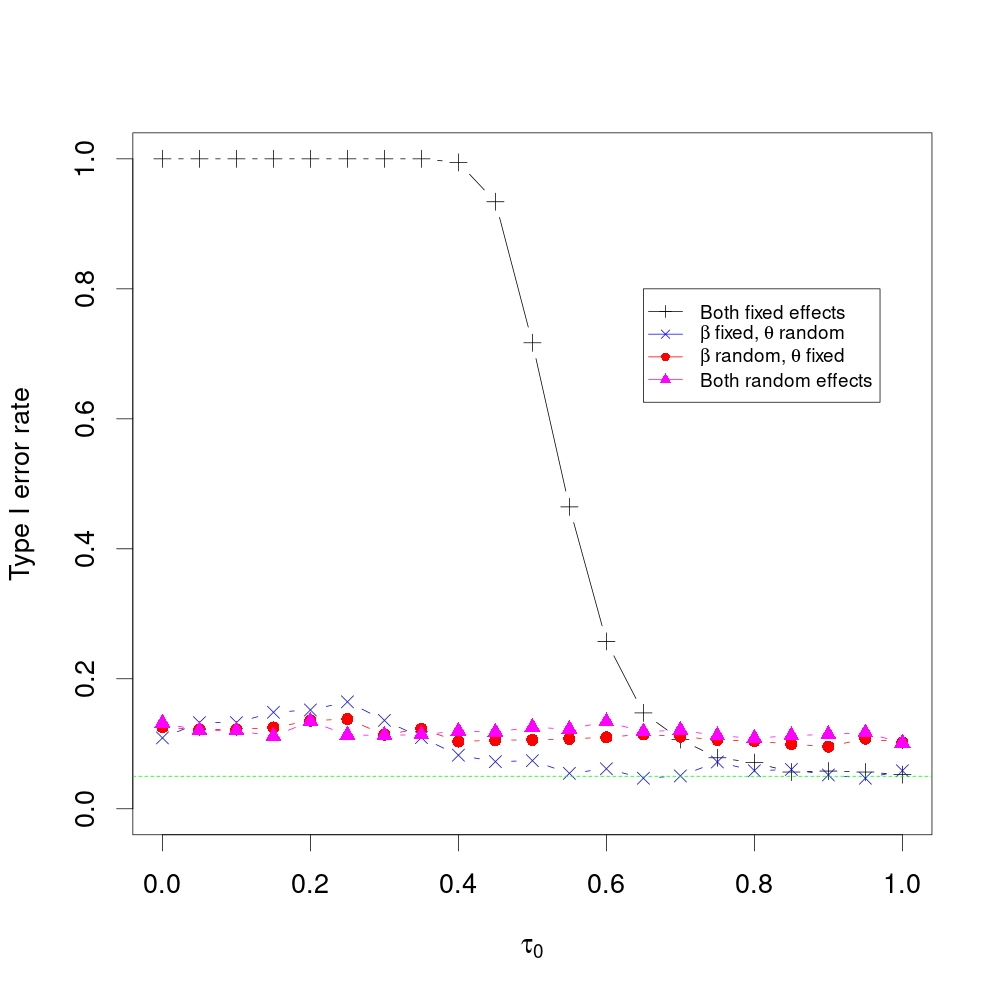}
    \caption{Binomial mixture, binomial exposure\\($\gamma=2$)}
    \label{fig:divspike_1000_500_bin_mix_FALSE}
\end{subfigure}
\hfill
\begin{subfigure}{5 cm}
    \includegraphics[width=5 cm]{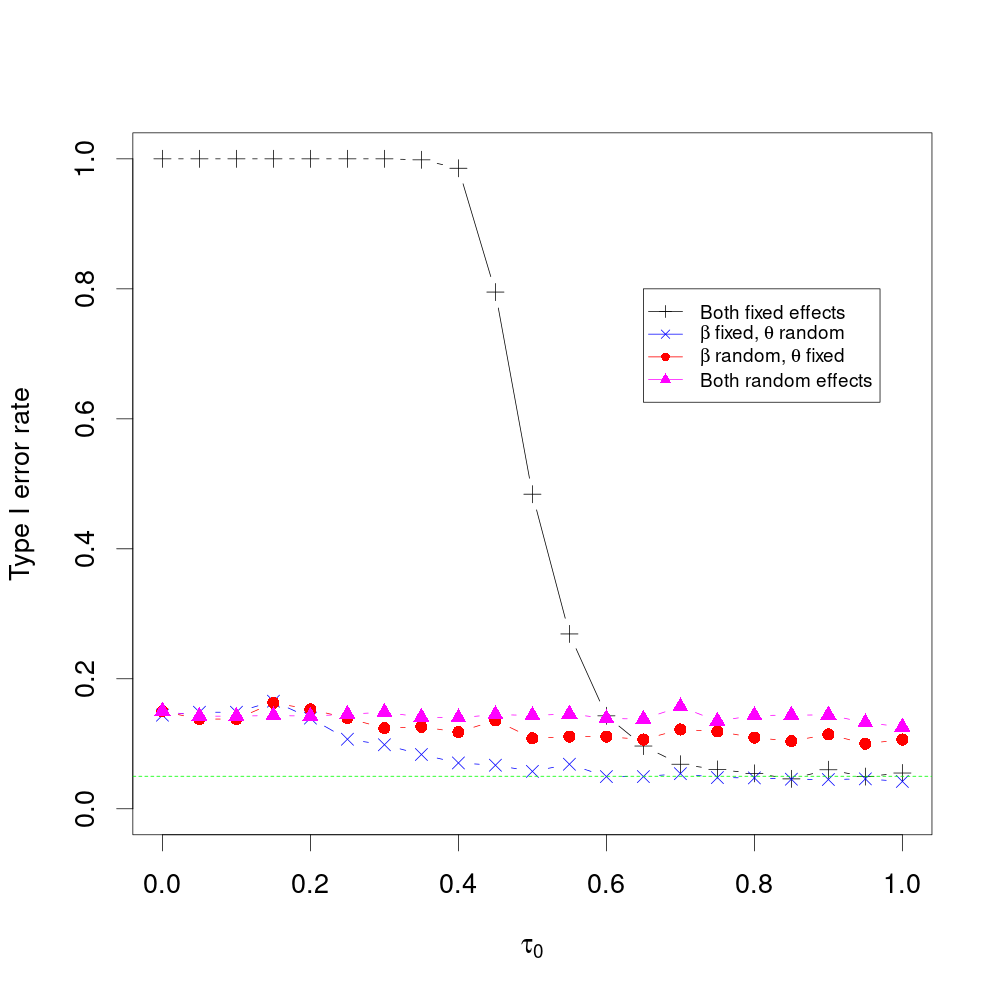}
    \caption{Binomial mixture, binomial exposure\\($\gamma=0.5$)}
    \label{fig:divspike_1000_2000_bin_mix_FALSE}
\end{subfigure}
        
\caption{Simulation results in out-regression for different marginal distributions of $\bW$, conditional distributions of $A|\bW$, aspect ratios, regression coefficients, and strength of the spike (or population stratification)}
\label{fig:out_divspike}
\end{figure}

In Figure~\ref{fig:out_divspike}, the left and the right columns present the results based on the two different aspect ratios, $\gamma=2$ and $0.5$, respectively. From top to bottom, the three rows represent simulation setups (1), (2), and (3) as described above. In each of the plots, the black line corresponds to both $\bbeta$ and $\btheta$ being fixed effects, the blue line corresponds to $\bbeta$ being a fixed effect and $\btheta$ being a random effect, the red line corresponds to $\bbeta$ being a random effect and $\btheta$ being a fixed effect, and the pink line corresponds to both $\bbeta$ and $\btheta$ being random effects. The results verify the claims made in Theorem \ref{thm:diverging_spike_out}, as well as Theorems \ref{thm:spiked_model_fixed_strength_out},\ref{thm:mixture_model_fixed_strength} when $\tau_0=0$.

Next, we present results on tests based on $\mathrm{LR}_{k,in}$. Below we provide a description of different aspects of the simulation procedure that we will consider.
\begin{itemize}
    \item \textbf{Distribution of $\bX$:} We consider a spiked model for $\bX$. In particular, $\bX_i\sim N(0,\Sigma)$ with $\Sigma=I+\sum_{j=1}^{k^*}\lambda_jv_jv_j^T$ with $\lambda_1\geq \ldots,\lambda_k^*>0$. We shall set $k^*=1$, and $\lambda_1=p^{\tau_0}$ to regulate the spike strength, where $\tau_0\in\{0,0.05,0.10,\ldots,1\}$. We  generate $v_1$ by drawing randomly from $N_p(0,1/pI)$ and normalizing afterwards. Next, we shall partition $\bX=\left [A:\bW \right ]$, i.e., take the first column of $\bX$ as $A$ and the rest as $\bW$.

    \item \textbf{Aspect Ratio:} We shall consider two different aspect ratios $\gamma\in\{0.5,2\}$ between the sample size and the number of confounding variables in our simulations. In all our simulations, we shall let $p=1000$ (and thus the sample size $n\in\{500,2000\}$.
    
    \item \textbf{Regression Coefficients for $Y|A,\bW$:} We shall consider both fixed and random effects for $\bbeta$ in the model \eqref{eqn:model_gaussian_linear_outcome_regression}. When $\bbeta$ is a fixed effect, we shall set $\bbeta=\Sigma_{(22)}^{-1}\Sigma_{(21)}$, where $\Sigma_{(21)},\Sigma_{(22)}$ are defined based on the partition of $\Sigma$,
    $$\Sigma=\begin{bmatrix}\Sigma_{(11)} & \Sigma_{(12)}\\ \Sigma_{(21)} & \Sigma_{(22)} \end{bmatrix}.$$
    Here $\Sigma_{(11)}$ is a scalar. When $\bbeta$ is a random effect, $\bbeta$ is drawn randomly from $N_p(0,\frac1p I)$.
\end{itemize}
For 
different aspect ratios and  specifications of $\bbeta$, we simulate the outcomes under the null model ( $\delta=0$ in \eqref{eqn:model_gaussian_linear_outcome_regression}) by taking $\sigma^2_{\varepsilon}=1$. We performed 2000 replications of the dataset and tested $\varphi_{k,in}(\chi_1^2(\alpha))$ on each dataset by taking $k=1$ and $\alpha=0.05$. Finally, we plot the empirical type I error rate of the test across varying $\tau_0$ in Figure~\ref{fig:in_divspike}.
\begin{figure}
\centering
\begin{subfigure}{6.5 cm}
    \includegraphics[width=6.5 cm]{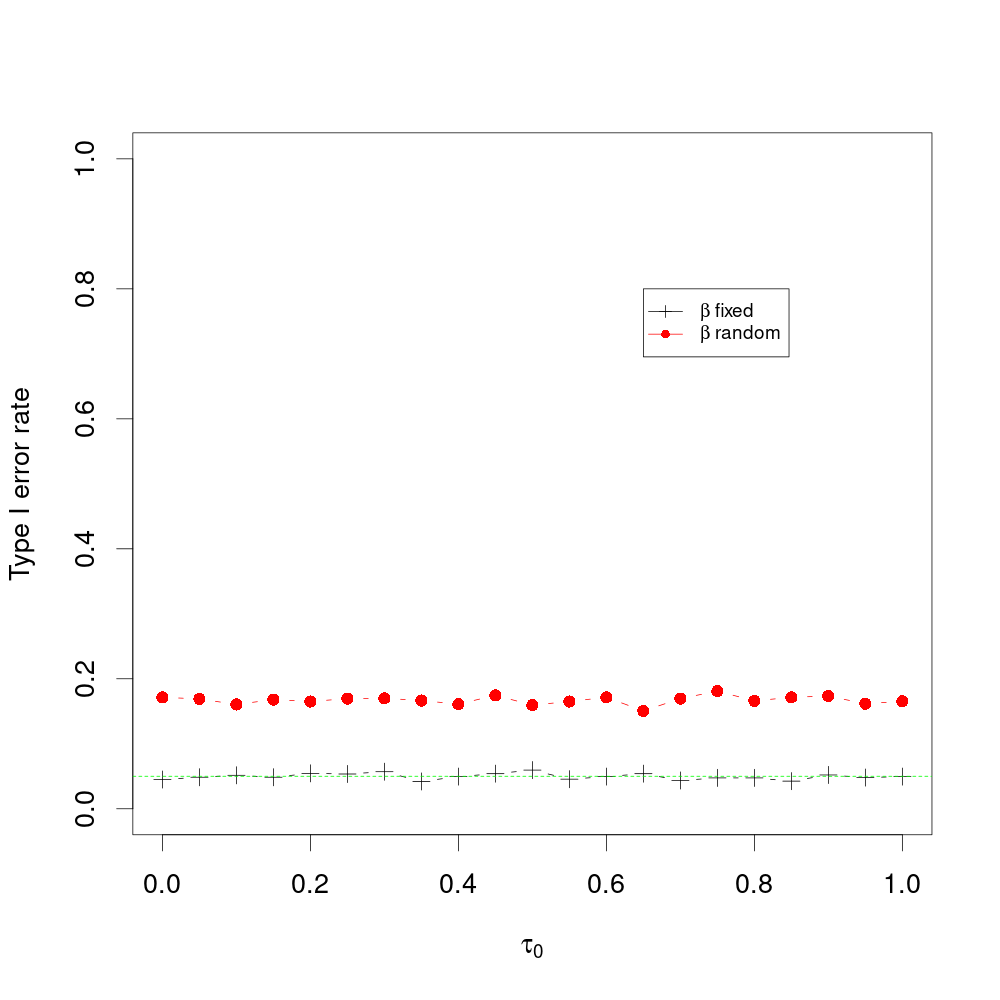}
    \caption{$\gamma=2$}
    \label{fig:divspike_1000_500_gauss_spike_TRUE_beta_is_theta}
\end{subfigure}
\hfill
\begin{subfigure}{6.5 cm}
    \includegraphics[width= 6.5 cm]{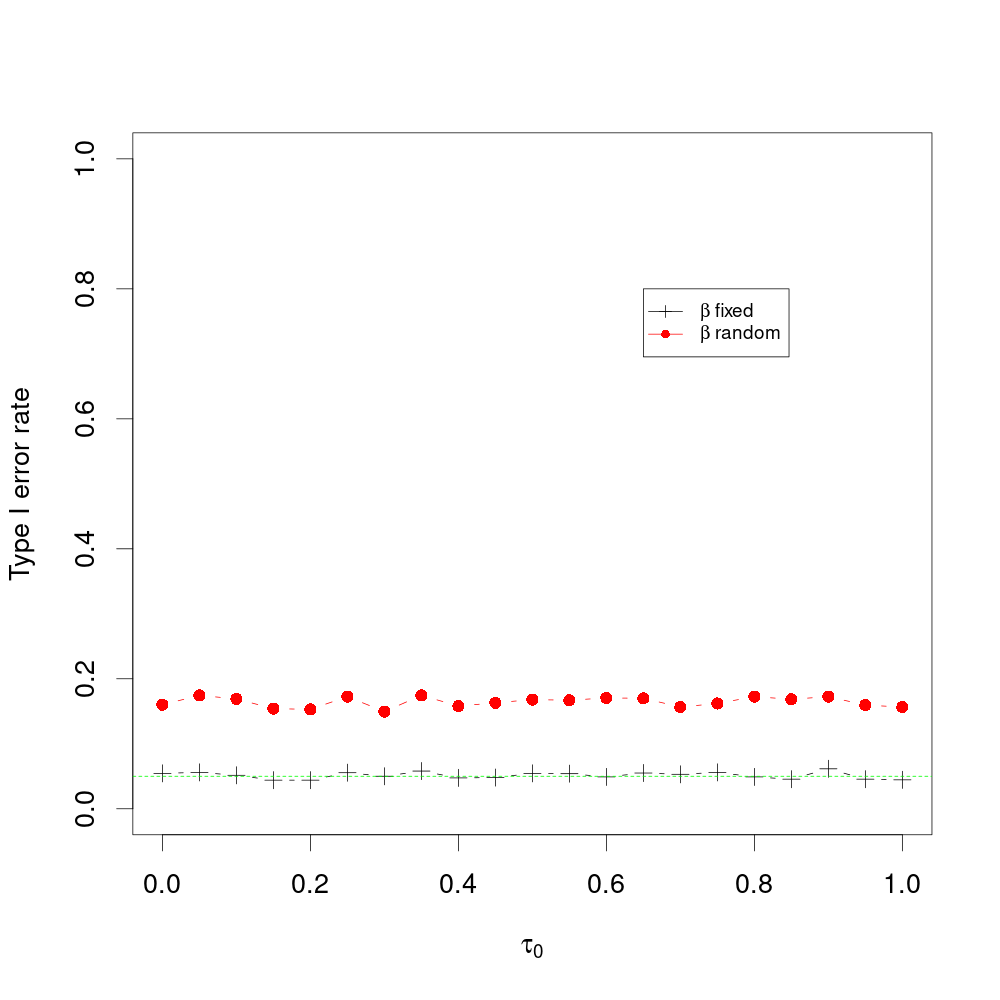}
    \caption{$\gamma=0.5$}
    \label{fig:divspike_1000_2000_gauss_spike_TRUE_beta_is_theta}
\end{subfigure}
        
\caption{Simulation results in in-regression for different marginal distributions of $\bX$, aspect ratios, regression coefficients, and spike strengths}
\label{fig:in_divspike}
\end{figure}

Figures~\ref{fig:divspike_1000_500_gauss_spike_TRUE_beta_is_theta} and \ref{fig:divspike_1000_2000_gauss_spike_TRUE_beta_is_theta} present the results based on the two different aspect ratios, $\gamma=2$ and $0.5$, respectively. In each of the plots, the black line corresponds to $\bbeta$ being fixed effects, and the red line corresponds to $\bbeta$ being random effects. The results verify the claims made in Theorem \ref{thm:diverging_spike_delocalized}.

\section{Discussions} In this paper we consider hypothesis testing of individual parameters in high dimensional linear regression set up under proportional asymptotics while adjusting for additional confounding through principal component analysis. Our analyzed method is regularly employed in large-scale genetic and epigenetic association studies.  Here, we take some initial steps to shed light on necessary and sufficient conditions for the validity of this procedure. Our results show, in fact, such a popular method can often have inflated type-I error causing inaccurate downstream inference.

Here, we list several possible extensions and important future research directions:

\begin{enumerate}
    \item[(i)] Moving beyond proportional asymptotic regime: Our results rely crucially on the explicit calculation of the regime $p \propto n$. However, many of our results of this article can be extended to the ultra-high dimensions $p \gg n$ under appropriate signal strength by understanding behavior of bilinear forms of sample eigenvectors (cf. \cite[Theorem 1]{paul2012augmented} for examples for such analysis).  
    \item[(ii)] Beyond mixture models: In Theorem \ref{thm:mixture_model_fixed_strength}, we investigate the behavior of likelihood ratio under the mixture models, which is a crucial first-step of understanding PC regression beyond spiked covariance models. However, it is important to analyze $\mathrm{LR}_{k,\ind}$ for other genetic admixture models ~\cite{price2006principal}, which are a generalization of mixture models in
which each data point is generated from multiple components.
\item[(iii)] Fixing the type-I error inflation:  Our results show, often for both fixed and diverging spike strength regimes, the inflated type-I error renders the PC regression impractical. A natural next step is to understand the behavior of alternative methods in high-dimension to provide a guideline for practitioners. To this end, we plan to analyze supervised PC-regression \cite{bair2006prediction} in proportional asymptotic regime for a future work.
\item[(iv)] Extension to GLMs: Often, in GWAS, $A|\bW$ follows a multiple logistic regression with Hardy-Weinberg equilibrium proportions depending through the logistic link on $\bW$. It will be interesting to understand the behavior of likelihood ratio tests beyond linear regression.

\item[(v)] Other directions: Our paper shows type-I error inflation making the traditional cut-offs useless. One might be interested in constructing corrections to the cut-offs for $\mathrm{LR}_{k,\ind}$ in regimes where the Type-I error of the tests is inflated. Finally, it will be interesting to perform analogous analysis for case-control studies.
\end{enumerate}


\bibliographystyle{plainnat}
\bibliography{biblio_pca}

\newpage

\section{Proofs of Main Results}\label{sec:proofs}
Throughout our proofs whenever the context is clear, with an abuse of notation, we shall write the SVD of both $\bbW$ and $\bbX$ as $\sum_{j=1}^{n \wedge p} \hat{d}_j \hat{u}_j\hat{v_j}^T$ and use them accordingly in the context of $\mathrm{LR}_{\out}$ and  $\mathrm{LR}_{\into}$ respectively. Also, denote by $\hat{\lambda}_j=\hat{d}^2_j/n$ the eigenvalues of $\frac{1}{n}\bbW^\top \bbW$ or $\frac{1}{n}\bbX^\top \bbX$ respectively.

\subsection{\bf Proof of Theorem \ref{thm:spiked_model_fixed_strength_out}}
We only prove the result for $\mathrm{LR}_{k,\out}$ and note that the proof for $\mathrm{LR}_{k,\into}$ follows by similar arguments by simply replacing $\bbeta,\btheta$ by $\tilde{\bbeta}:=\tilde{I}_{-1}\bbeta$ and  $\tilde{\btheta}:=\tilde{I}_{-1}\btheta$ respectively with $\tilde{I}_{-1}:=[e_2:\cdots:e_p]$ and $e_j$ denotes the $j^{\rm th}$ canonical basis of $\mathbb{R}^{p+1}$. 

We begin by noting that by Lemma \ref{lem:distribution_lr} one has
\be
    \mathrm{LR}_{k, \rm out}|\mathbf{A},\mathbb{W}\sim \chi_1^2(\hat{\kappa}^2_{k,\rm out}),\quad  \hat{\kappa}^2_{k,\rm out}=\frac{\left((\bbeta^T\mathbb{W}^\top+\mathbf{A}^\top \delta)\left(I-{P}_{\mathcal{C}(\mathbb{W}\widetilde{\mathbf{V}}_{k,\rm out})}\right)\mathbf{A}\right)^2}{\mathbf{A}^\top\left(I-{P}_{\mathcal{C}(\mathbb{W}\widetilde{\mathbf{V}}_{k,\rm out})}\right)\mathbf{A}}.
\ee
Therefore, we have,
$$\mathbb{P}_{\delta=0}(\mathrm{LR}_{k,\out}> t_\alpha)=\mathbb{E}\Big(\bar{\Phi}(t_\alpha-\hat{\kappa}_{k,\out}^2)-\Phi(-t_\alpha-\hat{\kappa}_{k,\out}^2)\Big)^2,$$
where $\Phi$ is the standard normal cdf and $\bar{\Phi}=1-\Phi$. Hence, the rest of the proof involves obtaining asymptotic distribution of $\hat{\kappa}_{k,\out}^2$, which coupled with uniform integrability yields the asymptotic behavior of the likelihood ratio test.

\subsubsection*{\bf Proof of Theorem \ref{thm:spiked_model_fixed_strength_out}\ref{thm:both_fixed_fixed_strength_out_reg}}
The non-centrality parameter $\hat{\kappa}_{k,\out}^2$ under $H_0$ is given by
\be\label{eq:null_kap}
    \hat{\kappa}_{k,\out}^2=(\mathbb{W} \bbeta)^\top P_{\mathcal{C}(\mathcal{A}_k)} (\mathbb{W} \bbeta) =\frac{\left(\bbeta^\top \mathbb{W}^\top\left(I-{P}_{\mathcal{C}(\mathbb{W}\widetilde{\mathbf{V}}_{k,\rm out})}\right)\mathbf{A}\right)^2}{\mathbf{A}^\top\left(I-{P}_{\mathcal{C}(\mathbb{W}\widetilde{\mathbf{V}}_{k,\rm out})}\right)\mathbf{A}}:=\frac{T^2_1}{T_2}.
\ee
{\color{black} Since we only deal with $H_0$, we shorten the notation $\P_{\delta=0}$ by $\P$.} We analyze the behavior of $T_2$ first. Note that the coordinates of $A_i$ are i.i.d. with finite $\psi_2$-norm (this follows from Assumption \ref{assn:prop_projection} and \eqref{eqn:model_gaussian_linear_outcome_regression}). Therefore by Bernstein's inequality, $\exists$ $C_1,C_1'>0$ such that $\|\mathbf{A}\|^2_2 \le C_1n$ w.p. $\ge 1- e^{-C_1'n}$. Hence, as $\|I-{P}_{\mathcal{C}(\mathbb{W}\widetilde{\mathbf{V}}_{k,\rm out})}\| \le 1$, we have,
\[
\frac{T_2}{n} \le \frac{\|\mathbf{A}\|^2_2}{n} \le C_1,
\]
w.p. $\ge 1- e^{-C_1'n}$. 

Next we show that there exists a constant $C_2$ and a pair of vectors $\bbeta,\btheta$ such that  with probability converging to $1$ one has $T_1^2= T^2_1(\bbeta,\btheta) \geq 4C^2_2n^2$. To this end note that for any $\bbeta,\btheta$,
\be
    T_1=\bbeta^\top \mathbb{W}^\top\left(I-{P}_{\mathcal{C}(\mathbb{W}\widetilde{\mathbf{V}}_{k,\rm out})}\right)\mathbb{W}\btheta+ \bbeta^\top \mathbb{W}^\top\left(I-{P}_{\mathcal{C}(\mathbb{W}\widetilde{\mathbf{V}}_{k,\rm out})}\right) \boldsymbol{\eta}:= T_{11}+T_{12}. \label{eq:num_null}
\ee
Now, we will show that with high probability both of the following hold: (a) $|T_{11}| \ge 2C_2 n$ and (b) $|T_{12}| = o_{\P} (n)$ which yields the desired result.

 To show (a) note that for any $\bbeta=\btheta$, one has 
\be
  T^2_{11}&= \left( \btheta^\top \mathbb{W}^\top \left(I-{P}_{\mathcal{C}(\mathbb{W}\widetilde{\mathbf{V}}_{k,\rm out})}\right) \mathbb{W} \btheta\right)^2= \left(\sum_{j=k+1}^{n\wedge p}\hat{d}^2_j\btheta^\top\hat{v_j}\hat{v}_j^\top\btheta\right)^2= n^2 \left(\sum_{j=k+1}^{n\wedge p}\hat{\lambda}_j\btheta^T\hat{v_j}\hat{v}_j^T\btheta\right)^2.
\ee
Now using Lemma \ref{lem:spikes_non_spikes}, there exists $\|\btheta\|_2^2 =1$ and $C_2>0$ with  $$\sum_{j=k+1}^{n\wedge p}\hat{\lambda}_j\btheta^T\hat{v_j}\hat{v}_j^T\btheta \rightarrow 2C_2 >0$$ a.s. under $\P$. Hence, for this choice of $\bbeta=\btheta$, $T^2_{11} \ge 4C^2_2 n^2$ with probability converging to $1$. 

To show (b) note that $\E(T_{12}|\bbW)=0$ since $\E(\boldsymbol{\eta})=0$, and for $\bbeta=\btheta$ one has
\[
\text{Var}\Big(\frac{1}{n}T_{12}|\bbW\Big)=\frac{1}{n^2} \E\left[\btheta^\top \mathbb{W}^\top \left(I-{P}_{\mathcal{C}(\mathbb{W}\tilde{\mathbf{V}_k})}\right)\mathbb{W} \btheta|\bbW\right]\leq \hat{\lambda}_{k+1}^2/n
\]
yielding $\frac1n T_{12}= o_{\P}(1)$ since $\hat{\lambda}_{k+1}^2=O_{\P}(1)$ by \cite{shen2013surprising}. This shows $\hat{\kappa}_{k,\out}^2 = \Omega(n)$, thereby completes the proof of Theorem \ref{thm:spiked_model_fixed_strength_out}\ref{thm:both_fixed_fixed_strength_out_reg}. 

\subsubsection*{\bf Proof of Theorem \ref{thm:spiked_model_fixed_strength_out}\ref{thm:one_fixed_fixed_strength_out_reg}}
Here, we have assumed $\bbeta \sim N\Big(0,\frac{\sigma^2_\beta}{p}I_p\Big)$ and $\btheta \in \R^p$ be a fixed vector with $\|\btheta\| \le M$ for some fixed $M>0$. Fix some $h\neq 0$ and $\delta:=\frac{h}{\sqrt{n}}$. {\color{black} We shorten the notation $\P_{\delta=\frac{h}{\sqrt{n}}}$ by $\P$.} Recall that, the non-centrality parameter of the likelihood ratio test under alternative $\delta$ is $\hat{\kappa}^2_{k,\out}$ given by Lemma \ref{lem:distribution_lr}:
\be\label{eq:kap}
    \hat{\kappa}_{k,\out}^2=\frac{\left((\bbeta^T\mathbb{W}^\top+\mathbf{A}^\top \delta)\left(I-{P}_{\mathcal{C}(\mathbb{W}\widetilde{\mathbf{V}}_{k,\rm out})}\right)\mathbf{A}\right)^2}{\mathbf{A}^\top\left(I-{P}_{\mathcal{C}(\mathbb{W}\widetilde{\mathbf{V}}_{k,\rm out})}\right)\mathbf{A}} :=\frac{T^2_1}{T_2}.
\ee
Similar to the previous part, we will compute asymptotic distribution of $\hat{\kappa}_{k,\out}^2$ and thereby the result follows by uniform integrability. 

To this end, we will first show $T_2/n$ converges, in probability, to a positive constant. Note that by Assumption \ref{assn:prop_projection} we have $\mathbf{A}|\mathbb{W} \sim N(\mathbb{W} \btheta, \sigma^2_g I)$. Hence, 
\be\label{eq:theta_1quad}
    \mathbb{E}\left(\frac{T_2}{n}|\mathbb{W}\right) =\frac{1}{n} \btheta^\top \mathbb{W}^\top \left(I-{P}_{\mathcal{C}(\mathbb{W}\widetilde{\mathbf{V}}_{k,\rm out})}\right) \mathbb{W} \btheta+ \frac{\sigma^2_g}{n} \text{Tr}\left(I-{P}_{\mathcal{C}(\mathbb{W}\widetilde{\mathbf{V}}_{k,\rm out})}\right)
\ee
For the second summand above, by definition, $\text{Tr}({P}_{\mathcal{C}(\mathbb{W}\widetilde{\mathbf{V}}_{k,\rm out})})= \text{rank}({P}_{\mathcal{C}(\mathbb{W}\widetilde{\mathbf{V}}_{k,\rm out})}) \le k=o(n)$. Hence, $$\sigma^2_g \ge \frac{\sigma^2_g}{n} \text{Tr}\left(I-{P}_{\mathcal{C}(\mathbb{W}\widetilde{\mathbf{V}}_{k,\rm out})}\right)\ge (1-o(1))\sigma^2_g \Longrightarrow \lim\limits_{n \rightarrow \infty}\frac{\sigma^2_g}{n} \text{Tr}\left(I-{P}_{\mathcal{C}(\mathbb{W}\widetilde{\mathbf{V}}_{k,\rm out})}\right)= \sigma^2_g.$$
Using SVD of $\mathbb{W}$,
\be \label{eq:one_random_denom_mean}
    \frac{1}{n} \btheta^\top \mathbb{W}^\top \left(I-{P}_{\mathcal{C}(\mathbb{W}\widetilde{\mathbf{V}}_{k,\rm out})}\right) \mathbb{W} \btheta= \frac{1}{n}\sum_{j=k+1}^{n\wedge p}\hat{d}^2_j\btheta^\top\hat{v_j}\hat{v}_j^\top\btheta= \sum_{j=k+1}^{n\wedge p}\hat{\lambda}_j\btheta^T\hat{v_j}\hat{v}_j^T\btheta \rightarrow c_0
\ee
a.s. under $\P$, by \eqref{eq:non-spike-mom-i-call} where the specific form of $c_0$ is provided in Lemma \ref{lem:spikes_non_spikes}. Hence, $\mathbb{E}\left(\frac{T_2}{n}|\mathbb{W}\right) \rightarrow c_0 +\sigma^2_g$, a.s. under $\P$. Further,
\be
    \mathrm{Var}\left(\frac{T_2}{n}|\mathbb{W}\right)= \frac{2 \sigma^4_g}{n^2} \text{Tr} \left(I-{P}_{\mathcal{C}(\mathbb{W}\widetilde{\mathbf{V}}_{k,\rm out})}\right)+ \frac{4 \sigma^2_g}{n^2} \btheta^\top \mathbb{W}^\top \left(I-{P}_{\mathcal{C}(\mathbb{W}\widetilde{\mathbf{V}}_{k,\rm out})}\right) \mathbb{W} \btheta.
\ee
The second summand of RHS converges to $0$ a.s. under $\P$ by \eqref{eq:one_random_denom_mean} and the first quantity is deterministic and $O\big(\frac{2 \sigma^4_g}{n}\big)=o(1)$. So,
\begin{equation}\label{eq:t2}
    \frac{T_2}{n} \xrightarrow{\P} c_0+\sigma^2_g,
\end{equation}
unconditionally on $\mathbb{W}$ using Dominated Convergence Theorem.

Next, we derive asymptotic behavior of $T_1/\sqrt{n}$, where $T_1$ is defined by \eqref{eq:kap}. Recall that, $\mathbf{A}= \mathbb{W} \btheta +\boldsymbol{\eta}$ and $$T_1= \bbeta^T\mathbb{W}^\top \left(I-{P}_{\mathcal{C}(\mathbb{W}\widetilde{\mathbf{V}}_{k,\rm out})}\right)\mathbf{A} + \delta \mathbf{A}^\top \left(I-{P}_{\mathcal{C}(\mathbb{W}\widetilde{\mathbf{V}}_{k,\rm out})}\right)\mathbf{A}.$$
Since $\delta= h/\sqrt{n}$, the second summand is:
$$\frac{\delta}{\sqrt{n}}\mathbf{A}^\top  \left(I-{P}_{\mathcal{C}(\mathbb{W}\widetilde{\mathbf{V}}_{k,\rm out})}\right)\mathbf{A}=\frac{h}{n}\mathbf{A}^\top  \left(I-{P}_{\mathcal{C}(\mathbb{W}\widetilde{\mathbf{V}}_{k,\rm out})}\right)\mathbf{A} \xrightarrow{\P} h(c_0 +\sigma^2_g),$$
using \eqref{eq:t2}. Next, define 
\be\label{eq:defn_t3}
    T_3:=\mathbf{A}^\top \left(I-{P}_{\mathcal{C}(\mathbb{W}\widetilde{\mathbf{V}}_{k,\rm out})}\right)\mathbb{W}\mathbb{W}^\top \left(I-{P}_{\mathcal{C}(\mathbb{W}\widetilde{\mathbf{V}}_{k,\rm out})}\right) \mathbf{A}.
\ee
Conditioned on $\mathbb{W}, \mathbf{A}$, 
\be\label{eq:beta_normality}
    \frac{\bbeta^T\mathbb{W}^\top \left(I-{P}_{\mathcal{C}(\mathbb{W}\widetilde{\mathbf{V}}_{k,\rm out})}\right)\mathbf{A}/ \sqrt{n}}{\sqrt{T_3/np}} \sim N(0,\sigma^2_\beta).
\ee
We will show  $\frac{T_3}{np}$ converges, in probability, to a positive constant. To this end note that,
\be\label{eq:theta_quad}
    \ \mathbb{E}\left(\frac{T_3}{np}|\mathbb{W}\right) &= \frac{\sigma^2_g}{np} \text{Tr}\left(\left(I-{P}_{\mathcal{C}(\mathbb{W}\widetilde{\mathbf{V}}_{k,\rm out})}\right)\mathbb{W}\mathbb{W}^\top\left(I-{P}_{\mathcal{C}(\mathbb{W}\widetilde{\mathbf{V}}_{k,\rm out})}\right)\right) \nonumber \\ &+ \frac{1}{np} \btheta^\top \mathbb{W}^\top\left(I-{P}_{\mathcal{C}(\mathbb{W}\widetilde{\mathbf{V}}_{k,\rm out})}\right)\mathbb{W}\mathbb{W}^\top\left(I-{P}_{\mathcal{C}(\mathbb{W}\widetilde{\mathbf{V}}_{k,\rm out})}\right)\mathbb{W} \btheta.
\ee
The first quantity of RHS equals 
\begin{align}\label{eq:trx}
    \frac{\sigma^2_g}{np} \text{Tr}\left(\mathbb{W}^\top \left(I-{P}_{\mathcal{C}(\mathbb{W}\widetilde{\mathbf{V}}_{k,\rm out})}\right)\mathbb{W}\right) &= \frac{\sigma^2_g}{np} \text{Tr}\left(\sum_{j=k+1}^{n \wedge p} \hat{d_j}^2 \hat{v_j}\hat{v_j}^\top\right)\nonumber\\
    &= \frac{\sigma^2_g}{p}\sum_{j=k+1}^{n \wedge p} \hat{\lambda_j} \rightarrow \sigma^2_g m_1,
\end{align}
a.s. under $\P$ using Lemma \ref{lem:lmnt}. The second summand in \eqref{eq:theta_quad} equals \be\label{eq:trx2}
    \frac{1}{np} \sum_{j=k+1}^{n \wedge p} \hat{d_j}^4 \btheta^T\hat{v_j}\hat{v_j}^T\btheta = (1+o(1))\frac{1}{\gamma}\rightarrow c_4 >0,
\ee
a.s. under $\P$ using Lemma \ref{lem:spikes_non_spikes} for some $c_4>0$. Finally,
\be\label{eq:tr2x}
    \ & \mathrm{Var}\left(\frac{T_3}{np}|\mathbb{W}\right)\\ &= \frac{2\sigma^4_g}{n^2p^2} \text{Tr}\left(\left(\mathbb{W}^\top \left(I-{P}_{\mathcal{C}(\mathbb{W}\widetilde{\mathbf{V}}_{k,\rm out})}\right)\mathbb{W}\right)^2\right)+ \frac{4\sigma^2_g}{n^2p^2}\btheta^T \left(\mathbb{W}^\top\left(I-{P}_{\mathcal{C}(\mathbb{W}\widetilde{\mathbf{V}}_{k,\rm out})}\right)\mathbb{W}\right)^3 \btheta \nonumber\\
    &= \frac{2\sigma^4_g}{n^2p^2}\sum_{j=k+1}^{n \wedge p} \hat{d_j}^4+ \frac{4\sigma^2_g}{n^2p^2} \sum_{j=k+1}^{n \wedge p} \hat{d_j}^6\btheta^\top\hat{v_j}\hat{v_j}^T\btheta \nonumber\\
    &\leq \frac{2\sigma^4_g}{p^2}\sum_{j=k+1}^{n \wedge p} \hat{\lambda_j}^2+ \frac{4\sigma^2_g n}{p^2} \sum_{j=k+1}^{n \wedge p} \hat{
    \lambda_j}^3\btheta^\top\hat{v_j}\hat{v_j}^\top\btheta \rightarrow 0.
\ee
 a.s. under $\P$, where the first summand is $o(1)$ by Lemma \ref{lem:lmnt} and the second summand is $o(1)$ using \eqref{eq:non-spike-mom-ii-call} and $\hat{
    \lambda}_{k+1}=O_{\P}(1)$.
Hence, $\frac{T_3}{np} {\xrightarrow{\P}} \sigma^2_g m_1+c_4$ unconditionally using Dominated Convergence Theorem. By Slutsky's theorem \eqref{eq:beta_normality} becomes: 
$$\frac{T_1}{\sqrt{n}} \xrightarrow{D} N(  h(c_0 +\sigma^2_g), \sigma^2_\beta(\sigma^2_g m_1+c_4)).$$
This, coupled with \eqref{eq:t2}, implies
\begin{align}\label{eq:theta_fixed_ncp}
    \hat{\kappa}^2_{k,\out} \xrightarrow{D} \frac{\sigma^2_\beta(\sigma^2_g m_1+c_4)}{c_0 +\sigma^2_g}\chi^2(h^2(c_0 +\sigma^2_g)^2),
\end{align}
where $\hat{\kappa}^2_{k,\out}$ is defined by \eqref{eq:kap}.

Next consider both $\bbeta, \btheta$ random effects, i.e., $\btheta \sim N\big(0,\frac{\sigma^2_\theta}{p} I_p\big)$ and $\bbeta \in \R^p$ be fixed.
Recall that the denominator of $\hat{\kappa}_{k,\out}^2$ is $T_2= \mathbf{A}^\top\left(I-{P}_{\mathcal{C}(\mathbb{W}\widetilde{\mathbf{V}}_{k,\rm out})}\right)\mathbf{A}$ where $\mathbf{A}|\mathbb{W} \sim N(0,\frac{\sigma^2_\theta}{p}\mathbb{W}\mathbb{W}^\top+\sigma^2_g I_p)$. Hence
\be
    \mathbb{E}\left(\frac{T_2}{n}|\mathbb{W}\right)&=\frac{1}{n} \text{Tr}\left(\frac{\sigma^2_\theta}{p}\mathbb{W}\mathbb{W}^\top\left(I-{P}_{\mathcal{C}(\mathbb{W}\widetilde{\mathbf{V}}_{k,\rm out})}\right)+\sigma^2_g\left(I-{P}_{\mathcal{C}(\mathbb{W}\widetilde{\mathbf{V}}_{k,\rm out})}\right)\right)\\
    &=\sigma^2_\theta \frac1p \Big( \sum_{j=k+1}^{n\wedge p}\hat{\lambda}^2_j \Big)+ \sigma^2_g (1+o(1)) \rightarrow \sigma^2_\theta m_2+ \sigma^2_g,
\ee
a.s. $\P$ using Lemma \ref{lem:lmnt}. Moreover, 
\be
   \ & \mathrm{Var}\left(\frac{T_2}{n}|\mathbb{X}\right)\\
   &=\frac{2}{n^2}\text{Tr}\left(\left(I-{P}_{\mathcal{C}(\mathbb{W}\widetilde{\mathbf{V}}_{k,\rm out})}\right)\Big(\frac{\sigma^2_\theta}{p}\mathbb{W}\mathbb{W}^\top+\sigma^2_g I_p\Big)\left(I-{P}_{\mathcal{C}(\mathbb{W}\widetilde{\mathbf{V}}_{k,\rm out})}\right)\Big(\frac{\sigma^2_\theta}{p}\mathbb{W}\mathbb{W}^\top+\sigma^2_g I_p\Big)\right)\rightarrow 0,
\ee
a.s. under $\P$ by proof technique similar to the proof of \eqref{eq:tr2x}. Hence, $\frac{T_2}{n} \xrightarrow{\P} \sigma^2_\theta m_2+\sigma^2_g$. To show asymptotic normality of $\frac{T_1}{\sqrt{n}}$, note that $$\frac{\beta^\top \mathbb{W}^\top \left(I-{P}_{\mathcal{C}(\mathbb{W}\widetilde{\mathbf{V}}_{k,\rm out})}\right)\mathbf{A}}{\sqrt{V_0}}\xrightarrow{D} N(0,1),$$ where
\be
    V_0:= \bbeta^\top \mathbb{W}^\top \left(I-{P}_{\mathcal{C}(\mathbb{W}\widetilde{\mathbf{V}}_{k,\rm out})}\right)\left(\frac{\sigma^2_\theta}{p}\mathbb{W}\mathbb{W}^\top+\sigma^2_g I_p\right)\left(I-{P}_{\mathcal{C}(\mathbb{W}\widetilde{\mathbf{V}}_{k,\rm out})}\right)\mathbb{W}\bbeta.
\ee
Now, observe that, $\frac{V_0}{n}\xrightarrow{ \P} C_1$, for some $C_1>0$ by calculation similar to \eqref{eq:theta_1quad} and \eqref{eq:theta_quad}. Finally 
\[
\frac{\delta}{\sqrt{n}}\mathbf{A}^\top\left(I-{P}_{\mathcal{C}(\mathbb{X}\widetilde{\mathbf{V}}_{k,\rm out})}\right)\mathbf{A}=\frac{h}{n}\mathbf{A}^\top\left(I-{P}_{\mathcal{C}(\mathbb{X}\widetilde{\mathbf{V}}_{k,\rm out})}\right)\mathbf{A} \xrightarrow{\P} h(\sigma^2_\theta m_1+\sigma^2_g).
\] Hence, $\frac{T_1}{\sqrt{n}}\xrightarrow{D} N(h(\sigma^2_\theta m_2+\sigma^2_g),C_1)$ yielding
\be\label{eq:beta_fixed_ncp}
    \hat{\kappa}^2_{k,\out} \xrightarrow{D} \frac{C_1}{\sigma^2_\theta m_2+\sigma^2_g}\chi^2(h^2(\sigma^2_\theta m_2+\sigma^2_g)^2).
\ee
Finally consider $\bbeta \sim N\Big(0,\frac{\sigma^2_\beta}{p}I_p\Big)$ and $\btheta \sim N \Big(0,\frac{\sigma^2_\theta}{p}I_p\Big)$. Once again we want to find the asymptotic distribution of $\hat{\kappa}^2_{k,\out}$, where $\hat{\kappa}^2_{k,\out}$ was defined as \eqref{eq:kap}. To this end, observe that, $T_2$ is free of $\bbeta$ hence by the previous case, we immediately obtain $\frac{T_2}{n} \xrightarrow{\P} \sigma^2_\theta m_2+\sigma^2_g$. \par 
To derive the asymptotic distribution of $T_1$, we will show $\frac{T_3}{np}$ converges, in probability, to a positive constant, where $T_3$ is defined in \eqref{eq:defn_t3}. Note that,
\be
\mathbb{E}\left(\frac{T_3}{np}|\mathbb{W},\btheta\right) &= \frac{\sigma^2_g}{np} \text{Tr}\left(\left(I-{P}_{\mathcal{C}(\mathbb{W}\widetilde{\mathbf{V}}_{k,\rm out})}\right)\mathbb{W}\mathbb{W}^\top\left(I-{P}_{\mathcal{C}(\mathbb{X}\widetilde{\mathbf{V}}_{k,\rm out})}\right)\right)\\&+ \frac{1}{np} \btheta^\top \left(\mathbb{W}^\top\left(I-{P}_{\mathcal{C}(\mathbb{W}\widetilde{\mathbf{V}}_{k,\rm out})}\right)\mathbb{W}\right)^2 \btheta
\ee
The first term in the RHS converges a.s. to $\sigma^2_g m_1$ using \eqref{eq:trx}. Denoting the second term on the RHS by $T_4$, note
\be
    \mathbb{E}(T_4|\mathbb{W})= \frac{\sigma^2_\theta}{np^2} \text{Tr}\left(\left(\mathbb{W}^\top \left(I-{P}_{\mathcal{C}(\mathbb{W}\widetilde{\mathbf{V}}_{k,\rm out})}\right)\mathbb{W}\right)^2\right)=(1+o(1))\frac{\sigma^2_\theta}{\gamma p}\sum_{j=k+1}^{n \wedge p} \hat{\lambda_j}^2 \xrightarrow{\text{a.s.}\P} \frac{\sigma^2_\theta m_2}{\gamma},
\ee
using Lemma \ref{lem:lmnt} and
\be
    \mathrm{Var}(T_4|\mathbb{W})=\frac{2\sigma^4_\theta}{n^2p^4}\text{Tr}\left(\left(\mathbb{W}^\top \left(I-{P}_{\mathcal{C}(\mathbb{W}\widetilde{\mathbf{V}}_{k,\rm out})}\right)\mathbb{W}\right)^4\right)=\frac{1}{n^2p^4}\sum_{j=k+1}^{n \wedge p} \hat{d_j}^8 \xrightarrow{\text{a.s.}\P} 0.
\ee
Therefore, $$\frac{T_1}{\sqrt{n}} \xrightarrow{D} N(  h(\sigma^2_\theta m_1+\sigma^2_g), \sigma^2_\beta(\sigma^2_g m_1+ \sigma^2_\theta m_2)).$$
Hence, using Slutsky's theorem,
\be\label{eq:both_random_ncp}
    \hat{\kappa}^2_{k,\out}\xrightarrow{D} \frac{\sigma^2_\beta(\sigma^2_g m_1+ \sigma^2_\theta m_2)}{\sigma^2_\theta m_1 +\sigma^2_g}\chi^2(h^2(\sigma^2_\theta m_1+\sigma^2_g)^2),
\ee
concluding the proof of the theorem for $\mathrm{LR}_{k,\out}$.

\subsection{\bf Proof of Theorem \ref{thm:thm:spiked_model_fixed_strength_out_nogaussian_a}}
We divide the proof according to the parts of the theorem.
\subsubsection*{\bf Proof of Theorem \ref{thm:thm:spiked_model_fixed_strength_out_nogaussian_a}\ref{spiked_model_fixed_strength_out_nogaussian_a_out}}

We once again begin by noting that by Lemma \ref{lem:distribution_lr} one has
\be
    \mathrm{LR}_{k, \rm out}|\mathbf{A},\mathbb{W}\sim \chi_1^2(\hat{\kappa}^2_{k,\rm out}),\quad  \hat{\kappa}^2_{k,\rm out}=\frac{\left((\bbeta^T\mathbb{W}^\top+\mathbf{A}^\top \delta)\left(I-{P}_{\mathcal{C}(\mathbb{W}\widetilde{\mathbf{V}}_{k,\rm out})}\right)\mathbf{A}\right)^2}{\mathbf{A}^\top\left(I-{P}_{\mathcal{C}(\mathbb{W}\widetilde{\mathbf{V}}_{k,\rm out})}\right)\mathbf{A}}.
\ee 
Therefore, we will only analyze $\hat{\kappa}^2_{k,\rm out}$. The denominator of $\hat{\kappa}^2_{k,\rm out}$ behaves as follows: Since projection contracts norm and $A$ is sub-Gaussian we immediately have
\begin{align*}
    \mathbf{A}^T(I-P_{\mathcal{C}(\mathbb{W}\widetilde{\mathbf{V}}_{k,\out})})\mathbf{A}\leq \|\mathbf{A}\|_2^2\leq cn
\end{align*}
with high probability for some constant $c>0$.
To lower bound the numerator of $\hat{\kappa}_{k,\out}^2$ we consider a specific instance where $A\perp \mathbf{W}$ 
 which implies that
\be
T_1:=\mathbf{A}^T(I-P_{\mathcal{C}(\mathbb{W}\widetilde{\mathbf{V}}_{k,\out})})\mathbb{W}\bbeta\bbeta^T\mathbb{W}^T(I-P_{\mathcal{C}(\mathbb{W}\widetilde{\mathbf{V}}_{k,\out})})\mathbf{A}
\ee
has mean
\be 
\E(T_1|\mathbb{W})&=\mathrm{Tr}\left((I-P_{\mathcal{C}(\mathbb{W}\widetilde{\mathbf{V}}_{k,\out})})\mathbb{W}\bbeta\bbeta^T\mathbb{W}^T(I-P_{\mathcal{C}(\mathbb{W}\widetilde{\mathbf{V}}_{k,\out})})\right)\\
&=\bbeta^T\mathbb{W}^T(I-P_{\mathcal{C}(\mathbb{W}\widetilde{\mathbf{V}}_{k,\out})})(I-P_{\mathcal{C}(\mathbb{W}\widetilde{\mathbf{V}}_{k,\out})})\mathbb{W}\bbeta.\\
\ee 
Also, since coordinates of $\mathbf{A}$ are sub-Gaussian we have by sub-Gaussian moment bounds that for some constant $C>0$
\be 
\mathrm{Var}\left(T_1|\mathbb{W}\right)\leq C \left(\bbeta^T\mathbb{W}^T (I-P_{\mathcal{C}(\mathbb{W}\widetilde{\mathbf{V}}_{k,\out})})\mathbb{W}\bbeta\right)^2
\ee
Therefore by Paley-Zygmund Inequality one has for some $0<\zeta<1$ 
\be 
\P(T_1\geq \zeta \E(T_1|\mathbb{W})|\mathbb{W})\geq \zeta
\ee
Now note that $\bbeta^T\mathbb{W}^T(I-P_{\mathcal{C}(\mathbb{W}\widetilde{\mathbf{V}}_{k,\out})})(I-P_{\mathcal{C}(\mathbb{W}\widetilde{\mathbf{V}}_{k,\out})})\mathbb{W}\bbeta
=n\sum_{j=k+1}^{n\wedge p}\hat{\lambda}_j^2\langle \bbeta,\hat{v}_j\rangle^2$ and that by \eqref{eq:non-spike-mom-ii-call} we have $\sum_{j=k+1}^{n\wedge p}\hat{\lambda}_j^2\langle \bbeta,\hat{v}_j\rangle^2\stackrel{\P}{\rightarrow} c'_0>0$ for some $c'_0>0$. This completes the proof of the desired result.



\subsubsection*{\bf Proof of Theorem \ref{thm:thm:spiked_model_fixed_strength_out_nogaussian_a} \ref{spiked_model_fixed_strength_random_effects_out_nogaussian_a_out}}
The proof follows same steps as the previous part. By the Paley-Zygmund Inequality, it is again enough to that there exists $c_0>0$ such that $\frac1n \bbeta^T\mathbb{W}^T(I-P_{\mathcal{C}(\mathbb{W}\widetilde{\mathbf{V}}_{k,\out})})\mathbb{W}\bbeta \stackrel{\P}{\rightarrow} c_0$. To this end note that,
\begin{align*}\label{eq:subg_random_mean}
\frac1n \E \Big(\bbeta^T\mathbb{W}^T(I-P_{\mathcal{C}(\mathbb{W}\widetilde{\mathbf{V}}_{k,\out})})\mathbb{W}\bbeta \Big)= \frac {\sigma^2_\beta}{p} \sum_{j=k+1}^{n \wedge p} \hat{\lambda}_j \rightarrow m_1 \sigma^2_\beta>0,
\end{align*}
by Lemma \ref{lem:lmnt}. Also,
\[
\text{Var} \Big( \frac1n \bbeta^T\mathbb{W}^T(I-P_{\mathcal{C}(\mathbb{W}\widetilde{\mathbf{V}}_{k,\out})})\mathbb{W}\bbeta \Big) = \frac{ \sigma^4_\beta}{n^2 p} \text{Tr} \Big(\mathbb{W}^T(I-P_{\mathcal{C}(\mathbb{W}\widetilde{\mathbf{V}}_{k,\out})})\mathbb{W}\Big)^2 \rightarrow 0.
\]
This yields the desired conclusion.

\subsubsection*{\bf Proof of Theorem \ref{thm:thm:spiked_model_fixed_strength_out_nogaussian_a}\ref{spiked_model_fixed_strength_out_nogaussian_a_in}}
Recall that, $I_{-1}=[e_2:\dots:e_p]$ and $\mathbb{W}=\mathbb{X} I_{-1}$ and $\mathbf{A}=\mathbb{X}e_1$. By Lemma \ref{lem:distribution_lr},
\begin{align}\label{eq:nogaussian_in_kappa}
    \hat{\kappa}_{k,\into}^2=\frac{\left(\bbeta^\top I^\top_{-1}\mathbb{X}^\top\left(I-{P}_{\mathcal{C}(\mathbb{X}\tilde{\mathbf{V}_{k,\into}})}\right)\mathbb{X}e_1\right)^2}{e^{\top}_1\mathbb{X}^\top\left(I-{P}_{\mathcal{C}(\mathbb{X}\tilde{\mathbf{V}_{k,\into}})}\right)\mathbb{X}e_1}
\end{align}

By Lemma \ref{lem:spikes_non_spikes},
\begin{align*}
 \hat{\kappa}^2_{k,\into}& = (1+o(1))n \frac{\left(\sum_{j=1}^{n \wedge p} \phi_{1j}\bbeta^{\top}_0 v_j v^\top_j e_1\right)^2}{\sum_{j=1}^{n\wedge p} \phi_{1j}e^{\top}_1 v_j v^\top_j e_1}=(1+o(1)) \frac{1}{\gamma}\frac{(\sqrt{p}c^\star_p(\beta_0))^2}{\sum_{j=1}^{n \wedge p} \phi_{1j}e^{\top}_1 v_j v^\top_j e_1}
\end{align*}

Note that, $\exists C_1,C_2>0$ such that $C_1 \le \phi_{1j} \le C_2$ for all $j$. Hence, $C_1 \le \sum_{j=1}^{n \wedge p} \phi_{1j}e^{\top}_1 v_j v^\top_j e_1 \le C_2$. Hence if $|c^\star_p(\beta_0)|>0$, $\hat{\kappa}^2_{k,\into} \xrightarrow{a.s. \P} \infty$, proving (a). If $\liminf \sqrt{p} |c^\star_p(\bbeta_0)| =C_3>0$, then $\exists C_4>0$ such that $\hat{\kappa}^2_{k,\into} \ge C_4>0$ with high probability, proving $(b)$.
Finally, if $\limsup \sqrt{p} |c^\star_p(\bbeta_0)| =0$, $\hat{\kappa}^2_{k,\into}=o_{\P}(1)$, proving (c).

\subsubsection*{\bf Proof of Theorem \ref{thm:thm:spiked_model_fixed_strength_out_nogaussian_a}\ref{spiked_model_fixed_strength_random_effects_out_nogaussian_a_in}}
We want to analyse the behavior of $\hat{\kappa}^2_{k,\into}$ given by \eqref{eq:nogaussian_in_kappa}. By Lemma \ref{lem:spikes_non_spikes}, the denominator is $O_{\P}(n)$. Since $\bbeta \sim N(0,\frac1p \sigma^2_\beta I_p)$, it is enough to show $\frac{T_3}{n^2}=\Omega_{\P}(1)$, where

\begin{align}\label{eq:fixed_in_t3}
    T_3:=e^\top_1\mathbb{X}^\top\left(I-{P}_{\mathcal{C}(\mathbb{X}\widetilde{\mathbf{V}}_{k,\into})}\right)\mathbb{X}I_{-1}I^\top_{-1}\mathbb{X}^\top\left(I-{P}_{\mathcal{C}(\mathbb{X}\widetilde{\mathbf{V}}_{k,\into})}\right)\mathbb{X}e_1.
\end{align}
Note that, $I_{-1}I^\top_{-1}=I-e_1e^\top_1$. Hence,
\begin{align*}
    \frac{T_3}{n^2}&= \frac{1}{n^2}\left(e^{\top}_1\left(\mathbb{X}^\top\left(I-{P}_{\mathcal{C}(\mathbb{X}\widetilde{\mathbf{V}}_{k,\into})}\right)\mathbb{X}\right)^2 e_1- \left(e^{\top}_1\mathbb{X}^\top\left(I-{P}_{\mathcal{C}(\mathbb{X}\widetilde{\mathbf{V}}_{k,\into})}\right)\mathbb{X}e_1\right)^2\right)\\
    &=\sum_{j=2}^{n \wedge p} \hat{\lambda}^2_j \langle e_1,\hat{v}_j \rangle^2 -\left(\sum_{j=2}^{n \wedge p} \hat{\lambda}_j \langle e_1,\hat{v}_j \rangle^2\right)^2\\
    &\ge \left(\sum_{j=2}^{n \wedge p} \hat{\lambda}^2_j \langle e_1,\hat{v}_j \rangle^2 \right) \left(1-\sum_{j=2}^{n\wedge p} \langle e_1,\hat{v}_j \rangle^2 \right) =\Omega_{\P}(1),
\end{align*}
where the first term is $\Omega_{\P}(1)$ using \ref{lem:spikes_non_spikes}. This completes the proof of the Theorem.

\subsection{\bf Proof of Theorem \ref{thm:mixture_model_fixed_strength}}

We begin by stating couple of Lemmas with would be useful for the proof of the Theorem.

The first lemma  establishes a desired moment bounds on $\mathbf{W}$ for $\ind=\out$ and on $\mathbf{X}$ for $\ind=\into$.
\begin{lemma}\label{lem:mesmix}
Let $\mathbf{D} \sim \sum_{i=1}^{L} w_i F(\bmu_i,\Sigma)$ for any fixed $L$ with $w_i \ge 0$, with $\sum_{i=1}^{L} w_i =1$, $\sup_i \|\mu_i\| \le 1$, $\|\Sigma\|=O(1)$ where $F(\bmu,\Sigma)$ is the distribution of a random vector $\Sigma^{1/2}\mathbf{Z}+\bmu$ where $\mathbf{Z}$ is a random vector with mean zero i.i.d. sub-Gaussian coordinates. Let $\mathbf{C} $ be a deterministic $p \times p$ matrix. Define $\mathbf{T}:= \E (\mathbf{D} \mathbf{D}^\top)$. Define $\tilde{\mathbf{C} }= \Sigma^{1/2}\mathbf{C}  \Sigma^{1/2}$. Then for any $q \in \mathbb{N}$,
\begin{equation}\label{eq:mixmom}
    \E[|\mathbf{D}^\top \mathbf{C}  \mathbf{D} - \text{Tr}(\mathbf{C} \mathbf{T})|^{2q}] \leq K_{q} \Big(\big(\text{Tr}(\tilde{\mathbf{C} }\tilde{\mathbf{C} }^T)\big)^q + \text{Tr}\big((\tilde{\mathbf{C} }\tilde{\mathbf{C} }^T)^q\big)+(\text{Tr}(\mathbf{C} ))^{2q}\Big).
\end{equation}
\end{lemma}
\begin{proof}
Note that, $\mathbf{D}= \Sigma^{1/2}\mathbf{Z}+ \sum_{i=1}^L \mu_i\xi_i$, where $\xi=(\xi_1,\ldots,\xi_l) \sim \text{Mult}(1;(w_1,\ldots,w_l))$. Define $\mathbf{R}=\sum_{i=1}^L \mu_i\xi_i$. Notice that, $\mathbf{D}^\top \mathbf{C} \mathbf{D}= \mathbf{Z}^\top \tilde{\mathbf{C} } \mathbf{Z} + \mathbf{R}^\top (\mathbf{C} +\mathbf{C} ^\top) \Sigma^{1/2}\mathbf{Z}+ \mathbf{R}^\top \mathbf{C}  \mathbf{R}$. Also, $\mathbf{T}= \E\Sigma^{1/2}\mathbf{Z}\mathbf{Z}^\top\Sigma^{1/2} + \E \mathbf{R}\mathbf{R}^\top= \Sigma + \E \mathbf{R}\mathbf{R}^\top$. So,
\begin{align*}
    \E[|\mathbf{D}^\top \mathbf{C}  \mathbf{D} - \text{Tr}(\mathbf{C} \mathbf{T})|^{2q}] & \le \tilde{K}_q \Big( \E[|\mathbf{Z}^\top \tilde{\mathbf{C} } \mathbf{Z} - \text{Tr}(\mathbf{C} \Sigma) |^{2q}]+ \E[|\mathbf{R}^\top (\mathbf{C} +\mathbf{C} ^\top) \Sigma^{1/2}\mathbf{Z}|^{2q}]\\ & +\E[|\mathbf{R}^\top \mathbf{C}  \mathbf{R} - \text{Tr}(\mathbf{C} \E(\mathbf{R}\mathbf{R}^\top))|^{2q}]\Big).
\end{align*}
Using \cite[Lemma 3]{mestre2006asymptotic}, the first summand is upper bounded by $$\tilde{M}_q \left(\Big(\text{Tr} (\tilde{\mathbf{C} } \tilde{\mathbf{C} }^\top)\Big)^q+\text{Tr} \Big((\tilde{\mathbf{C} } \tilde{\mathbf{C} }^\top)^q\Big)\right)$$ for some $\tilde{M}_q>0$. Given $\xi$, $\mathbf{R}^\top \mathbf{C}  \Sigma^{1/2}\mathbf{Z}$ is a subgaussian random variable with proxy $\mathbf{R}^\top \mathbf{C}  \Sigma \mathbf{C} ^\top \mathbf{R}$. Since $\|\mathbf{R}\| \le \sum_i w_i \|\bmu_i\| \le 1$, $\exists$ $M_q>0$,
$$\E[|\mathbf{R}^\top \mathbf{C}  \Sigma^{1/2}\mathbf{Z}|^{2q}]= M_q (\mathbf{R}^\top \mathbf{C}  \Sigma C^T \mathbf{R})^q \leq M_q (\text{Tr} \mathbf{C}  \Sigma \mathbf{C} ^T)^q.$$
Finally $\mathbf{R}^\top \mathbf{C}  \mathbf{R}= \sum_{i=1}^{L} \xi_i \bmu^\top_i \mathbf{C}  \bmu_i$. Hence,
\begin{align*}
    \E[|\mathbf{R}^\top \mathbf{C}  \mathbf{R} - \text{Tr}(\mathbf{C} \E(\mathbf{R}\mathbf{R}^\top))|^{2q} \le c_0 \sum_{i=1}^L \left(\bmu^\top_i \mathbf{C}  \bmu_i\right)^{2q} \le c_0 L (\text{Tr}\mathbf{C} )^{2q},
\end{align*}
for some $c_0>0$. This completes the proof.  \end{proof}

Now, we prove analogue of \cite[Theorem 5]{mestre2006asymptotic} for mixture models which yields the limit of bilinear forms.
\begin{lemma}\label{lem:mestre_mixture_bilinear}
Define $\hat{\mathbf{R}}= \frac{1}{n} \mathbf{D}\mathbf{D}^\top$ and $\mathbf{T}=\E \mathbf{D}\mathbf{D}^\top$, where $\mathbf{D}$ is defined as in Lemma \ref{lem:mesmix}. Let the eigenvalues and eigenvectors of $T$ are given by $\lambda_1(\mathbf{T})\ge \lambda_1(\mathbf{T}) \ge \ldots$ and $ v_1(\mathbf{T}), v_2(\mathbf{T}), \ldots$. Then for $z \in \mathbb{C}^+= \{z_0 \in \mathbb{C}: \text{Im}(z_0)>0\}$, we have
\begin{equation*}
    \lim\limits_{n,p \rightarrow \infty} \left\lvert \frac{1}{n} \text{Tr}(\hat{\mathbf{R}}-zI_n)^{-1}-b(z)\right\rvert =0,
\end{equation*}
where $b(z):=b$ is the unique solution of the following equation in the set $\{b \in \mathbb{C}: -(1-c)/b +cz \in \mathbb{C}^+\}:$
\begin{equation*}
    b=\frac{1}{n} \sum_{k=1}^n\frac{1}{\lambda_k(\mathbf{T})(1-c- czb)-z}.
\end{equation*}
Morever, for any fixed vector $s$ with $\|s\| \le 1$, we have 
\begin{equation*}
    \lim\limits_{n,p \rightarrow \infty} \left\lvert s^\top (\hat{\mathbf{R}}-zI_n)^{-1} s - \sum_{k=1}^{n} \frac{s^T v_k(\mathbf{T})v^\top_k(\mathbf{T}) s}{\lambda^\top _k(\mathbf{T})(1-c- czb)-z}\right\rvert =0
\end{equation*}
\end{lemma}
\begin{proof}
Note that the statement of the Lemma is a restatement of \cite[Theorem 5]{mestre2006asymptotic} holds if one replaces the eigenvalues and eigenvectors of $\mathbf{R}$ by $\mathbf{T}$. The assumptions of \cite[Theorem 1.1]{baicol} are satisfied by $\hat{R}$ with $\mathbf{T}=\E \mathbf{D}\mathbf{D}^\top$ since $\|\mathbf{T}\|\le \|\Sigma\|+ \|\mathbb{E}\mathbf{R}\mathbf{R}^\top\|=O(1)$ as we are working under fixed spike strength assumption and using Lemma \ref{lem:mesmix}. Hence, the empirical spectral distribution of $\hat{\mathbf{R}}$ converges with the limit given by \cite[(1.1)]{baicol}. Define $\hat{\mathbf{R}}_{(n)}:= \hat{\mathbf{R}} - \frac{1}{n}d_n d^\top_n$. Then, it is enough to show \cite[Eq 17,18]{mestre2006asymptotic} with $\mathbf{R}$ replaced by $\mathbf{T}$. To this end, note that \cite[Eq 17]{mestre2006asymptotic} follows immediately from \cite[Lemma 6]{mestre2006asymptotic} and convergence of ESD of $\hat{\mathbf{R}}$. The proof of \cite[Eq 25]{mestre2006asymptotic} follows the same way once we define 
$$\theta_n=b^T d_nd^\top_n \Big(\hat{\mathbf{R}}_{(n)}-zI\Big)^{-1}a- b^\top \mathbf{T} \Big(\hat{\mathbf{R}}_{(n)}-zI\Big)^{-1}a,$$
and bound its moments using Lemma \ref{lem:mesmix}. The remaining proof follows once we observe $\frac{1}{n^{2q}}\E\left(\|\mathbf{T}\|^{4q}\right)$ is bounded by a constant only depending on $q$ since $\|\bmu\|=1$.
\end{proof}

Now, Lemma \ref{lem:mestre_mixture_bilinear} and Lemma \ref{lem:spikes_non_spikes} yields the limits of bilinear forms of sample eigenvectors. The rest of the proof can now be done in the exact same way as Theorem \ref{thm:spiked_model_fixed_strength_out}.

\subsection{\bf Proof of Theorem \ref{thm:diverging_spike_out}}
We divide the proof according to the parts of the theorem.
\subsubsection*{\bf Proof of Theorem \ref{thm:diverging_spike_out}\ref{thm:both_fixed_diverging_spike_out}}
Similar to the proof of Theorem \ref{thm:spiked_model_fixed_strength_out} we will only need to analyse the behavior of the non-centrality parameter $\hat{\kappa}^2_{k,\out}$ defined by \eqref{eq:null_kap}. Recall that, $\lambda_1,\lambda_{k^\star} =\Theta(p^{\tau_0})$, $\tau_0>0$. \par
Let $\bbeta, \btheta$ be fixed vectors with $\|\bbeta\|,\|\btheta\| \le 1$. The denominator of $\hat{\kappa}^2_{k,\out}$ is $T_2= \mathbf{A}^\top\left(I-{P}_{\mathcal{C}(\mathbb{W}\tilde{\mathbf{V}_{k,\out}})}\right)\mathbf{A}$ which does not depends on $\bbeta$. Since $\mathbf{A}=\mathbb{W}\btheta+\boldsymbol{\eta}$, write $T_2=D_1+D_2+D_3$, where $D_1=\btheta^\top\mathbb{W}^\top\left(I-{P}_{\mathcal{C}(\mathbb{W}\tilde{\mathbf{V}_{k,\out}})}\right)\mathbb{W}\btheta$, $D_2=2\boldsymbol{\eta}^\top\left(I-{P}_{\mathcal{C}(\mathbb{W}\tilde{\mathbf{V}_{k,\out}})}\right) \mathbb{W}\btheta$, $D_3=\boldsymbol{\eta}^\top\left(I-{P}_{\mathcal{C}(\mathbb{W}\tilde{\mathbf{V}_{k,\out}})}\right)\boldsymbol{\eta}$. Note, $D_3/n \xrightarrow{\P} \sigma^2_g$, $\E(D_2|\mathbb{W})=0$,
\be\label{eq:var_d2}
    \text{Var}\Big(\frac{D_2}{n}|\mathbb{W}\Big)=\frac{4}{n}\sum\limits_{j=k^\star+1}^{n \wedge p} \frac{\hat{d}^2_j}{n} \langle \hat{v}_j, \btheta  \rangle^2 \le \frac{4}{n} \hat{\lambda}_{k^\star+1} \|\btheta\|^2=o_{\mathbb{P}}(1),
\ee
since \cite[Theorem 2.5]{cai2020limiting} implies $\hat{\lambda}_{k^\star+1}=O_{\mathbb{P}}(1)$ and $\|\btheta\| \le 1$. By similar argument $0 \le \frac1n D_1 = O_{\mathbb{P}}(1)$. Therefore, $\frac1n T_2=O_{\mathbb{P}}(1)$ and $\P(T_2 \ge \sigma^2_g) \rightarrow 1$, yielding both upper and lower bounds on $T_2$. \par 
For the numerator, note that 
\[\frac{T_1}{\sqrt{n}}|\mathbb{W}\sim N\left(\frac{1}{\sqrt{n}}\bbeta^\top\mathbb{W}^\top \left(I-{P}_{\mathcal{C}(\mathbb{W}\tilde{\mathbf{V}_{k,\out}})}\right)\mathbb{W}\btheta, \frac{\sigma^2_g}{n}\bbeta^\top \mathbb{W}^\top\left(I-{P}_{\mathcal{C}(\mathbb{W}\tilde{\mathbf{V}_{k,\out}})}\right)\mathbb{W}\bbeta\right)\]
Note that {\color{black} for any $k\ge k^\star$},the conditional variance equals
\[\frac{\sigma^2_g}{n}\bbeta^T\mathbb{W}^T\left(I-{P}_{\mathcal{C}(\mathbb{W}\tilde{\mathbf{V}_{k,\out}})}\right)\mathbb{W}\bbeta \le \sigma^2_g \sum_{j=k^\star+1}^{n \wedge p} \frac{\hat{d}^2_j}{n} \langle \hat{v}_j, \bbeta  \rangle^2 \leq \sigma^2_g\hat{\lambda}_{k^\star+1}\sum_{j=k^\star+1}^{n \wedge p}  \langle \hat{v}_j, \bbeta  \rangle^2 .\]
By \cite[Theorem 2.5]{cai2020limiting}, $\hat{\lambda}_{k^\star+1}=O_{\mathbb{\P}}(1)$. Further,
\cite[Theorem 3.2]{wang2017asymptotics} implies, for $j \le k^\star$,
\be\label{eq:xiucai_quadratic}
    {\langle \hat{v}_j, \bbeta  \rangle^2}= \langle v_j, \bbeta  \rangle^2+O_{\P}\left(\frac{1}{\lambda_j}\right).
\ee
Since $\bbeta \in \mathcal{C}_{\tau}(S_{V_{k^\star}})$, $\tau>0$, and $\lambda_{k^\star} \rightarrow \infty$.
\[\sum_{j=k^\star+1}^{n \wedge p} \langle \hat{v}_j, \bbeta  \rangle^2 \le \|\bbeta\|^2 d(\beta,S_{V_{k^\star}})+O_{\P}\Big(\frac{1}{\lambda_{k^\star}}\Big) \le o_{\P}(p^{-\tau})+O_{\P}\Big(\frac{1}{\lambda_{k^\star}}\Big) \xrightarrow{a.s. \P} 0.\]
For the conditional mean, consider the case $\min\{\tau_0,\tau\} <1/2$ and $k=k^\star$. {\color{black} Since $\gamma<1$, assume $p\le n$. }When $\tau_0 <1/2$,
\be\label{eq:chhoto_tau_0}
    \sup_{\bbeta,\btheta\in \mathcal{C}_{\tau}(S_{V_{k^\star}})}\frac{1}{\sqrt{n}}\bbeta^\top\mathbb{W}^\top\left(I-{P}_{\mathcal{C}(\mathbb{W}\tilde{\mathbf{V}_{k^\star,\rm out}})}\right)\mathbb{W}\btheta \ge \sqrt{n}\hat{\lambda}^2_{ p}\sum_{j=k^\star+1}^{p} \langle \hat{v}_j, v_1  \rangle^2 \xrightarrow{\text{a.s.} \P} \infty,
\ee
{\color{black} by choosing $\bbeta=\btheta=v_1$ since $1-\langle \hat{v}_1, v_1  \rangle^2=\Omega_{\P}(\frac{1}{\lambda})$ and $\langle \hat{v}_1, v_1  \rangle^2= \Omega_{\P}(\frac{1}{\lambda_{k^\star}})$}. Now, if $\tau_0>1/2$ but $\tau<1/2$, fix a vector $\xi$ such that $\xi \in \mathcal{C}_{\tau}(S_{V_{k^\star}})$ and $d(\xi,S_{V_{k^\star}}) =\Omega(p^{-\tau})$. Hence,
\be\label{eq:chhoto_tau}
    \sup_{\bbeta,\btheta\in \mathcal{C}_{\tau}(S_{v_1})}\frac{1}{\sqrt{n}}\bbeta^\top\mathbb{W}^\top\left(I-{P}_{\mathcal{C}(\mathbb{W}\tilde{\mathbf{V}_{k^\star,\out}})}\right)\mathbb{W}\btheta &\ge \sqrt{n}\hat{\lambda}^2_{ p}\sum_{j=k^\star+1}^{ p} \langle \hat{v}_j, \xi  \rangle^2 \nonumber \\ &= \sqrt{n}\hat{\lambda}^2_{p} \left(\Omega\Big(p^{-\tau}\Big)+o(p^{-\tau_0})\right) \xrightarrow{\text{a.s.} \P} \infty.
\ee
Finally, when $\min\{\tau_0,\tau\} >1/2$ and $k \ge k^\star$,
\be 
\ & \frac{1}{\sqrt{n}}\bbeta^T\mathbb{W}^T\left(I-{P}_{\mathcal{C}(\mathbb{W}\tilde{\mathbf{V}_{k,\out}})}\right)\mathbb{W}\btheta \\& \le \left(\frac{1}{\sqrt{n}}\bbeta^T\mathbb{W}^T\left(I-{P}_{\mathcal{C}(\mathbb{W}\tilde{\mathbf{V}_{k,\out}})}\right)\mathbb{W}\bbeta\right)^{\frac12} \left(\frac{1}{\sqrt{n}}\btheta^T\mathbb{W}^T\left(I-{P}_{\mathcal{C}(\mathbb{W}\tilde{\mathbf{V}_{k,\out}})}\right)\mathbb{W}\btheta\right)^{\frac12}.
\ee
Each term of RHS can be upper bounded, as above, by 
\be\label{eq:div_fixed_high_strength}
    \sqrt{n}\hat{\lambda}^2_{k^\star+1} \left(o_{\P}(p^{-\tau})+O_{\P}\Big(\frac{1}{\lambda_{k^\star}}\Big)\right) \xrightarrow{\text{a.s.} \P} 0.
\ee

Hence, if $\min\{\tau_0,\tau\} <1/2$, $\hat{\kappa}^2_{k,\out} \xrightarrow{a.s. \P} \infty$, and if $\min\{\tau_0,\tau\} >1/2$, $\hat{\kappa}^2_{k,\out} \xrightarrow{a.s. \P} 0$. This completes the proof for fixed $\bbeta$, $\btheta$.

\subsubsection*{\bf Proof of Theorem \ref{thm:diverging_spike_out}\ref{thm:beta_random_diverging_spike_out}}
When $\tau_0=0$, the proof follows from \ref{thm:spiked_model_fixed_strength_out}. So, set $\tau_0>0$. Let $\btheta$ be a vector with $\|\btheta\| \le 1$ and $\bbeta \sim N(0,\frac{\sigma^2_\beta}{p}I_p)$. As shown in the proof of Theorem \ref{thm:diverging_spike_out}, $\frac1n T_2=O_{\mathbb{P}}(1)$ and $\P(T_2 \ge \sigma^2_g) \rightarrow 1$. By \eqref{eq:beta_normality}, it is enough to show $\frac{1}{np}T_3$ is stochastically bounded away from $0$, where $T_3$ is defined by,
$$T_3:=\mathbf{A}^\top \left(I-{P}_{\mathcal{C}(\mathbb{W}\tilde{\mathbf{V}_{k,\out}})}\right)\mathbb{W}\mathbb{W}^\top \left(I-{P}_{\mathcal{C}(\mathbb{W}\tilde{\mathbf{V}_{k,\out}})}\right) \mathbf{A}.$$ 
We can upper bound the conditional variance of $T_3$ using \eqref{eq:tr2x} as,
\[ \mathrm{Var}\left(\frac{T_3}{np}|\mathbb{W}\right)\leq \frac{2\sigma^4_g}{p^2}\sum_{j=k^\star+1}^{n \wedge p} \hat{\lambda_j}^2+ \frac{4\sigma^2_g n}{p^2} \sum_{j=k^\star+1}^{n \wedge p} \hat{
    \lambda_j}^3\btheta^\top\hat{v_j}\hat{v_j}^\top\btheta 
    \]
Using $\hat{\lambda}_{k^\star+1}=O_{\P}(1)$ by \cite[Theorem 2.5]{cai2020limiting} and noting that $\btheta \in \mathcal{C}_{\tau}(S_{V_{k^*}})$ with $\tau>0$, we obtain $\mathrm{Var}\left(\frac{T_3}{np}|\mathbb{W}\right) \xrightarrow{\text{a.s.} \P} 0$. For the conditional mean, we invoke \eqref{eq:theta_quad} and \eqref{eq:trx} to obtain $\E(\frac{T_3}{np}|\mathbb{W}) \ge (1-o_{\P}(1))\sigma^2_g \hat{\lambda}_{n \wedge p}$, which is stochastically bounded away from $0$ by \cite{baik2006eigenvalues}, yielding our claim.
\par 
If $\btheta \sim N(0,\frac{\sigma^2_\theta}{p}I_p)$, it is enough to analyze $T_3$ defined by \eqref{eq:defn_t3} where $\mathbf{A}|\mathbb{W}\sim N(0,\frac{\sigma^2_\theta}{p}\mathbb{W}\mathbb{W}^\top+\sigma^2_g I_p)$. To this end, note that,
\be
    \E\left(\frac{T_3}{np}|\mathbb{W}\right) &= \frac{1}{np}\text{tr}\Big[\left(\left(I-{P}_{\mathcal{C}(\mathbb{W}\tilde{\mathbf{V}_{k^\star,\out}})}\right)\mathbb{W}\mathbb{W}^\top \left(I-{P}_{\mathcal{C}(\mathbb{W}\tilde{\mathbf{V}_{k^\star,\out}})}\right)\right)\left(\frac{\sigma^2_\theta}{p}\mathbb{W}\mathbb{W}^\top+\sigma^2_g I_p\right)\Big]\\
    &=\frac{\sigma^2_\theta}{np^2}\sum_{j=k^\star+1}^{n \wedge p} \hat{d}^4_j+\frac{\sigma^2_g}{np}\sum_{j=k^\star+1}^{n \wedge p} \hat{d}^2_j = \Omega_{\P}(1).
\ee
The conditional variance is,
\be
   \text{Var}\left(\frac{T_3}{np}|\mathbb{W}\right)&= \frac{1}{n^2p^2} \text{tr}\Big[\left(\left(\left(I-{P}_{\mathcal{C}(\mathbb{W}\tilde{\mathbf{V}_{k^\star,\out}})}\right)\mathbb{W}\mathbb{W}^\top \left(I-{P}_{\mathcal{C}(\mathbb{W}\tilde{\mathbf{V}_{k^\star,\out}})}\right)\right)\left(\frac{\sigma^2_\theta}{p}\mathbb{W}\mathbb{W}^\top+\sigma^2_g I_p\right)\right)^2\Big]\\& \xrightarrow{\text{a.s.} \P} 0,
\ee
since $\sum_{j=k^\star+1}^{n \wedge p} \hat{d}^8_j =o_{\P}(n^2p^4)$ by Lemma \ref{lem:conv_mhat_moment_integral}, concluding the proof.

\subsubsection*{\bf Proof of Theorem \ref{thm:diverging_spike_out}\ref{thm:theta_random_beta_fixed_diverging_spike_out}} We analyze $\hat{\kappa}_{k,\out}^2$ under $\btheta \sim N(0, \frac{\sigma^2_\theta}{p}I_p)$. For the denominator of $\hat{\kappa}_{k,\out}^2$, using \eqref{eq:var_d2}, $\text{Var}(\frac1n D_2 | \mathbb{W})=o_{\mathbb{P}}(1)$ since $\|\btheta\|^2=O_{\mathbb{P}}(1)$ by law of large numbers. Similar bounds hold for $D_1$ yielding $T_2$ has similar behavior as the case of fixed $\btheta$. Here, first consider $\bbeta$ be a fixed vector with $\|\bbeta\|\le 1$. We have shown before, $\bbeta^T\mathbb{W}^T\left(I-{P}_{\mathcal{C}(\mathbb{W}\tilde{\mathbf{V}_{k,\out}})}\right)\boldsymbol{\eta}/\sqrt{n} \xrightarrow{\P} 0$, so it enough to analyze $T_{11}= \bbeta^T\mathbb{W}^T\left(I-{P}_{\mathcal{C}(\mathbb{W}\tilde{\mathbf{V}_{k,\out}})}\right)\mathbb{W}\btheta$. Defining $T_4:=\bbeta^T\Big(\mathbb{W}^T\left(I-{P}_{\mathcal{C}(\mathbb{W}\tilde{\mathbf{V}_{k,\out}})}\right)\mathbb{W}\Big)^2\bbeta$, we have 
$$\frac{T_{11}/\sqrt{n}}{\sqrt{T_4/np}} \Rightarrow N(0,\sigma^2_\theta).$$
Now,
\be
    \frac{T_4}{np} \le (1+o_{\P}(1))\frac{1}{\gamma} \hat{\lambda}_{k^\star+1}\sum_{j=k^\star+1}^{n \wedge p} \langle \hat{v}_j, \bbeta  \rangle^2 \xrightarrow{\text{a.s.} \P} 0,
\ee
using \eqref{eq:div_fixed_high_strength} for any $\tau_0>0$. This completes the proof of the Theorem.

\subsection{\bf Proof of Proposition \ref{thm:tau_zero_out_reg}}
{\color{black}Recall that, if $\tau_1 > \tau_2$, $\mathcal{C}_{\tau_1}(S_{V_{k^\star}}) \subseteq \mathcal{C}_{\tau_2}(S_{V_{k^\star}})$. Hence, 
\[\sup_{\bbeta\in \mathcal{C}_{0.25}(S_{V_{k^\star}})} \hat{\kappa}^2_{k,\out} \leq \sup_{\bbeta\in \mathcal{C}_{0}(S_{V_{k^\star}})} \hat{\kappa}^2_{k,\out}.\]
Since the asymptotic distribution of $\mathrm{LR}_{k,\out}$ is stochastically increasing function in $\hat{\kappa}^2_{k,\out}$, it is enough to show $\mathcal{C}_{0.25}(S_{V_{k^\star}})$ is stochastically bounded away from $0$ as $n \rightarrow \infty$. To this end, note that the conclusion is immediate via Theorem \ref{thm:diverging_spike_out} for all the cases except when $\btheta \sim N(0, \frac{\sigma^2_\theta}{p}I_p)$ and $\bbeta$ is a fixed vector. So, we will only consider that scenario. Here, $\mathbf{A}|\mathbb{W}\sim N(0,\frac{\sigma^2_\theta}{p}\mathbb{W}\mathbb{W}^\top+\sigma^2_g I_p)$.  Since $\gamma<1$, assume $p<n$. As before, the denominator is $O_\mathbb{P}(n)$. For the numerator, it is enough to analyze the quantity $T_3/n$ where 
\begin{equation}
    T_3:= \bbeta^\top \mathbb{W}^\top \left(I-{P}_{\mathcal{C}(\mathbb{W}\tilde{\mathbf{V}_{k^\star,\out}})}\right)\left(\frac{\sigma^2_\theta}{p}\mathbb{W}\mathbb{W}^\top+\sigma^2_g I_p\right)\left(I-{P}_{\mathcal{C}(\mathbb{W}\tilde{\mathbf{V}_{k^\star,\out}})}\right)\mathbb{W}\bbeta.
\end{equation}
Note that,
\begin{align*}
    \frac{T_3}{n} &= \frac{p \sigma^2_\theta}{n} \sum_{j=k^\star+1}^{p} \hat{\lambda}^2_j \langle \hat{v}_j,\bbeta\rangle^2 + \sigma^2_g \sum_{j=k^\star+1}^{ p} \hat{\lambda}_j \langle \hat{v}_j,\bbeta\rangle^2 \\
    & \ge \left(\frac{p \sigma^2_\theta}{n}\hat{\lambda}^2_{ p}+\sigma^2_g\hat{\lambda}_{ p}\right)(\|\bbeta\|^2- \sum_{j=1}^{k^\star}\langle \bbeta,v_j\rangle^2+O_{\P}(1/\lambda_1)).
\end{align*}
By \cite{baik2006eigenvalues}, $\hat{\lambda}^2_{ p}$ is stochastically bounded away from $0$. Pick $\bbeta$ such that $\|\bbeta\|^2=1$ and $P_{S_{V_{k^\star}}}(\beta)=0$, implying $T_3/n \xrightarrow{a.s. \P} C_1>0$, for some $C_1>0$, yielding the conclusion.}

\subsection{\bf Proof of Proposition \ref{hm:diverging_spike_out_nongaussian}}
Since $\gamma<1$ we have that $p < n$. Note that the size of the test $\mathrm{LR}_{k^*,\out}$ can be lower bounded through the specific instance where $A\perp \mathbf{W}$ and has mean $0$. We now analyze $\hat{\kappa}_{k^*,\out}$ in this regime. As noted in the proof of Theorem \ref{thm:thm:spiked_model_fixed_strength_out_nogaussian_a}, the denominator of $\hat{\kappa}_{k^*,\out}$ satisfies
$\mathbf{A}^T(I-P_{\mathcal{C}(\mathbb{W}\tilde{\mathbf{V}_{k^\star,\out}})})\mathbf{A} \leq cn$
with high probability for some constant $c>0$. We have the numerator of $\hat{\kappa}_{k^\star,\out}^2$ as

\be
T_1:=\mathbf{A}^T(I-P_{\mathcal{C}(\mathbb{W}\tilde{\mathbf{V}_{k^\star,\out}})})\mathbb{W}\bbeta\bbeta^T\mathbb{W}^T(I-P_{\mathcal{C}(\mathbb{W}\tilde{\mathbf{V}_{k^\star,\out}})})\mathbf{A}
\ee
The mean and variance of $T_1$ can bounded as in the proof of Theorem \ref{thm:thm:spiked_model_fixed_strength_out_nogaussian_a}, and by Paley-Zygmund Inequality one has,
\be 
\P(T_1\geq \zeta \E(T_1|\mathbb{W})|\mathbb{W})\geq \zeta
\ee
Recall that, $\bbeta^T\mathbb{W}^T(I-P_{\mathcal{C}(\mathbb{W}\tilde{\mathbf{V}}_{k^\star,\out})})\mathbb{W}\bbeta
=n\sum_{j=k^\star+1}^{p}\hat{\lambda}_j^2\langle \bbeta,\hat{v}_j\rangle^2$. Also, if $\min \{\tau_0,\tau\} <1/2$, using \eqref{eq:chhoto_tau_0} and \eqref{eq:chhoto_tau}, $ \sup_{\bbeta\in \mathcal{C}_{\tau}(S_{V_{k^\star}})} n\hat{\lambda}_{ p}(\|\bbeta\|^2-\sum_{j=1}^{k^\star} \langle\hat{v}_j,\bbeta\rangle^2) \xrightarrow{\text{a.s.} \P} \infty$. This proves the desired result. \par
Finally, when $\bbeta$ is random effects, the proof of Theorem 3.2 ii follows noting
\[\frac1n \E \Big(\bbeta^T\mathbb{W}^T(I-P_{\mathcal{C}(\mathbb{W}\widetilde{\mathbf{V}_{k^\star,\out}})})\mathbb{W}\bbeta \Big)= \frac {\sigma^2_\beta}{p} \sum_{j=k^\star+1}^{n \wedge p} \hat{\lambda}_j \ge \frac {\sigma^2_\beta}{p} \hat{\lambda}_p =\Omega_{\P}(1),
\]
yielding the desired conclusion. 

\subsection{\bf Proof of Theorem \ref{thm:diverging_spike_delocalized}}

\subsubsection*{\bf Proof of \ref{thm:diverging_spike_delocalized}\ref{thm:diverging_spike_fixed_effects_in}}
Recall that, $\mathbf{X}_i=[A_i, \bW_i] \sim N(0,I+\lambda vv^\top$). Writing $A_i=W^\top_i \btheta+\boldsymbol{\eta}$, we obtain that $\btheta= (I+\lambda v_{-1}v^\top_{-1})^{-1}\lambda_1 v(1)v_{-1}$, where $v=(v(1),v_{-1})$. By the Sherman-Morrison formula, we obtain
\begin{equation}\label{eq:theta_in}
    \btheta= \frac{\lambda v(1)}{1+\lambda \|v_{-1}\|^2} v_{-1}.
\end{equation}
So, $\|\btheta\|=\Theta(\frac{1}{\sqrt{p}})$ by our assumption of delocalization. Thereafter, with $k=1$, as before we consider the asymptotic behavior of  $\hat{\kappa}_{1,\into}^2=\frac{T_1^2}{T_2}$ with $T_1=\mathbf{A}^T(I-P_{\mathcal{C}(\bbW \widetilde{\mathbf{V}_{1,\into}})})\bbW\bbeta$ and $T_2=\mathbf{A}^T(I-P_{\mathcal{C}(\bbW \widetilde{\mathbf{V}_{1,\into}})})\mathbf{A}$. 

First, we claim that $T_2=\Omega_{\P}(n)$. To see this, we write $$T_2=\btheta_0^T\bbX^T(I-P_{\mathcal{C}(\bbW \widetilde{\mathbf{V}_{1,\into}})})\bbX\btheta_0+2\boldsymbol{\eta}^T(I-P_{\mathcal{C}(\bbW \widetilde{\mathbf{V}_{1,\into}})})\bbX\btheta_0+\boldsymbol{\eta}^T(I-P_{\mathcal{C}(\bbW \widetilde{\mathbf{V}_{1,\into}})})\boldsymbol{\eta}$$ where  $\btheta_0=I_{-1}\btheta$ with $I_{-1}=[e_2:\cdots:e_{p+1}]$. It is easy to see that the first term in the expansion is non-negative and the last term is $\Omega_{\P}(n)$. Therefore, it is enough to show that the second summand is $o_{\P}(1)$.

To that end note that the second summand has mean $0$ conditional on $\bbX$ and has conditional variance $\leq  n \hat{\lambda_2}\|\btheta_0\|^2=o_{\P}(n)$ by \cite{cai2020limiting}. Therefore the quantity is $o_{\P}(n)$ unconditionally by DCT. Hence, $T_2 =\Omega_{\P}(n)$.

Thereafter, to conclude to the proof of this part of the theorem we claim that $T_1/\sqrt{n}=o_{\P}(1)$, where 
$$T_1=\btheta_0^T\bbX^T(I-P_{\mathcal{C}(\bbW \widetilde{\mathbf{V}_{1,\into}})})\bbX\bbeta_0+\boldsymbol{\eta}^T(I-P_{\mathcal{C}(\bbW \widetilde{\mathbf{V}_{1,\into}})})\bbX\bbeta_0,$$
and $\btheta_0=I_{-1}\btheta$ and $\bbeta_0=I_{-1}\bbeta$ with $I_{-1}=[e_2:\cdots:e_{p+1}]$. The second term in the expression for $T_1$ divided by $\sqrt{n}$ has mean $0$ conditional on $\bbX$ and has conditional variance $\sum_{j=2}^{n \wedge p}\hat{\lambda}_j\langle\bbeta_0,\hat{v}_j\rangle^2\leq \hat{\lambda}_2\|\bbeta\|^2O_{\P}(p^{-\tau_0})=o_{\P}(1)$, again by \cite{cai2020limiting}. Hence this term is $o_{\P}(1)$ unconditionally by DCT. 

Focusing on the first summand of $T_1$, we have by writing $\bbeta=c^*v$ for $|c^*|\leq M$ we note that this term divided by $\sqrt{n}$ can be bounded by $\sqrt{n}\hat{\lambda}_2^2c^*O(\|\btheta_0\|)O_{\P}(p^{-\tau_0})=o_{\P}(p^{-\tau_0})=o_{\P}(1)$ since $\|\btheta_0\|=\|\btheta\|=O(1/\sqrt{p})=O(1/\sqrt{n})$. The proof follows. 

\subsubsection*{\bf Proof of \ref{thm:diverging_spike_delocalized}\ref{thm:diverging_spike_random_effects_in}}
Let $\bbeta \sim (0, \frac{\sigma^2_\beta}{p}I_p)$. Recall that, by \eqref{eq:theta_in}, $\btheta= \frac{\lambda v(1)}{1+\lambda \|v_{-1}\|^2} v_{-1}$. We again consider the asymptotic behavior of  $\hat{\kappa}_{1,\into}^2=\frac{T_1^2}{T_2}$. Note that,
\[
\frac{T_2}{n}= \frac{1}{n}e^T_1 \bbX^T(I-P_{\mathcal{C}(\bbW \widetilde{\mathbf{V}_{1,\into}})})\bbX e_1 \le \hat{\lambda}_2 =O_{\P}(1),
\]
by \cite{cai2020limiting}. To understand the behavior of the numerator, note that, $T_1=\btheta_0^T\bbX^T(I-P_{\mathcal{C}(\bbW \widetilde{\mathbf{V}_{1,\into}})})\bbX\bbeta_0+\boldsymbol{\eta}^T(I-P_{\mathcal{C}(\bbW \widetilde{\mathbf{V}_{1,\into}})})\bbX\bbeta_0$. Now, conditional on $\mathbb{X}$, $$\frac{1}{\sqrt{n}} \btheta_0^T\bbX^T(I-P_{\mathcal{C}(\bbW \widetilde{\mathbf{V}_{1,\into}})})\bbX\bbeta_0 \sim N\left(0, \frac{1}{np}\btheta_0^T\bbX^T(I-P_{\mathcal{C}(\bbW \widetilde{\mathbf{V}_{1,\into}})})\bbX (I-e_1e^\top_1)\bbX^T(I-P_{\mathcal{C}(\bbW \widetilde{\mathbf{V}_{1,\into}})})\bbX \btheta_0\right)$$
The variance above can be upper bounded by $ \frac{1}{np} \btheta_0^T(\bbX^T(I-P_{\mathcal{C}(\bbW \widetilde{\mathbf{V}_{1,\into}})})\bbX)^2\btheta_0 \le \hat{\lambda}^2_2 \|\btheta_0\|=o_{\P}(1)$. Moreover, conditional on $\mathbb{X}$,$\bbeta$,
$$\frac{1}{\sqrt{n}}\boldsymbol{\eta}^T(I-P_{\mathcal{C}(\bbW \widetilde{\mathbf{V}_{1,\into}})})\bbX\bbeta_0 \sim N \left(0,\frac{1}{n}\bbeta_0^T(\bbX^T(I-P_{\mathcal{C}(\bbW \widetilde{\mathbf{V}_{1,\into}})})\bbX)\bbeta_0\right).$$
Denoting the variance above by $T_4$, note that $\E(T_4|\bbX)=\frac{1}{np}\mathrm{Tr}(\bbX^T(I-P_{\mathcal{C}(\bbW \widetilde{\mathbf{V}_{1,\into}})})\bbX) \ge C_1$ a.s., for some $C_1>0$. Now $\mathrm{Var}(T_4|\bbX)=\frac{1}{n^2p^2} \sum_{j=2}^{n \wedge p} \hat{d}^2_j=o_{\P}(1)$. Hence, $\frac{T_1}{\sqrt{n}}=\Omega_{\P}(1)$, yielding the conclusion.

\section{Technical Lemmas:}

\begin{lemma}\label{lem:distribution_lr}
Under Model \ref{eqn:model_gaussian_linear_outcome_regression} one has that 
\begin{align*}
    \mathrm{LR}_{k, \rm out}|\mathbf{A},\mathbb{W}\sim \chi_1^2(\hat{\kappa}^2_{k,\rm out}),\quad  \hat{\kappa}^2_{k,\rm out}=\frac{\left((\bbeta^T\mathbb{W}^\top+\mathbf{A}^\top \delta)\left(I-{P}_{\mathcal{C}(\mathbb{W}\widetilde{\mathbf{V}}_{k,\rm out})}\right)\mathbf{A}\right)^2}{\mathbf{A}^\top\left(I-{P}_{\mathcal{C}(\mathbb{W}\widetilde{\mathbf{V}}_{k,\rm out})}\right)\mathbf{A}}.
\end{align*}
Similarly,
\begin{align*}
    \mathrm{LR}_{k, \rm in}|\mathbb{X}\sim \chi_1^2(\hat{\kappa}^2_{k,\rm in}),\quad  \hat{\kappa}^2_{k,\rm in}=\frac{\left((\bbeta^T\mathbb{W}^\top+\mathbf{A}^\top \delta)\left(I-{P}_{\mathcal{C}(\mathbb{X}\tilde{\mathbf{V}_{k,\rm in}})}\right)\mathbf{A}\right)^2}{\mathbf{A}^\top\left(I-{P}_{\mathcal{C}(\mathbb{X}\tilde{\mathbf{V}_{k,\rm in}})}\right)\mathbf{A}}.
\end{align*}
\end{lemma}
\begin{proof}
The proof follows by standard calculations and Fisher-Cochran Theorems and hence is omitted.
\end{proof}

\begin{lemma}\label{lem:spikes_non_spikes}
Let $s_1$ and $s_2$ be (a sequence of) $p\times 1$ non-random vectors with uniformly bounded norms.  $\widetilde{\mathbb{X}}\sim \mathrm{GSp}(\left\{(\lambda_j,v_j)\right\}_{j=1}^{k^*};\Gamma_H;n,d)$. Suppose $\hat{\Sigma}=\frac1n \widetilde{\mathbb{X}}^\top \widetilde{\mathbb{X}}$ has spectral decomposition $\hat{\Sigma}=\sum_{j=1}^{n \wedge p} \hat{\lambda}_j \hat{v}_j \hat{v}^\top_j$. For $\alpha\notin \Gamma_H,\alpha\neq 0$, let us define the function,
$$\psi(\alpha)=\alpha\left(1+\gamma\int_{\Gamma_H}{\frac{\lambda dH(\lambda)}{\alpha-\lambda}}\right ),$$
and the first derivative is denoted by $\psi '(\alpha)$. Then the following hold under Assumption \ref{assn:spike_and_dimension}(a).

\begin{enumerate}
    \item [(i)] [For Generalized Spikes] \begin{align*}
    \left |\hat{\lambda}_j^r s_1^\top \hat{v}_j\hat{v}_j^\top s_2- \xi_{rj} s_1^\top v_jv_j^\top s_2\right |&\xrightarrow{a.s.}0, \quad r=1,2,\text{ and } j=1,\ldots,k^*,
\end{align*}
where,
\begin{align*}
    \xi_{rj}=\lambda_j\psi(\lambda_j)^{r-1}\psi ' (\lambda_j).
\end{align*}

\item [(ii)] [For Non-Spikes] \begin{align*}
    \left |\sum_{j=k^*+1}^{p}{\hat{\lambda}_j^r s_1^\top \hat{v}_j\hat{v}_j^\top s_2}-\sum_{j=1}^{p}{\phi_{rj} s_1^\top v_jv_j^\top s_2} \right |&\xrightarrow{a.s.}0, \quad r=1,2,
\end{align*}
where,
\begin{align*}
    \phi_{1j}&=\begin{cases}
    \lambda_j\left [1-\psi '(\lambda_j) \right ],\quad j=1,\ldots,k^*\\
    \lambda_j,\quad j=k^*+1,\ldots,p,\end{cases}\\
    \phi_{2j}&=\begin{cases}
    \gamma\lambda_j\int_{\Gamma_H}{ \frac{\lambda^3}{(\lambda_j-\lambda)^2}dH(\lambda)}+\left [\psi(\lambda_j)-\lambda_j \right ] \left [1-\psi ' (\lambda_j) \right ],\quad j=1,\ldots,k^*\\
    \lambda_j^2+\gamma\lambda_j\int_{\Gamma_H}{\lambda dH(\lambda)},\quad j=k^*+1,\ldots,p.\end{cases}
\end{align*}
\end{enumerate}
\end{lemma}

\begin{remark}
    Before we provide the proof of the above Lemma, let us write the implication of the result. The Lemma states if $\widetilde{\mathbb{X}}\sim \mathrm{GSp}(\left\{(\lambda_j,v_j)\right\}_{j=1}^{k^*};\Gamma_H;n,d)$, then there exists $c_0,c'_0>0$ such that,

    \begin{align}
         \sum_{j=k^*+1}^{p}{\hat{\lambda}_j s_1^\top \hat{v}_j\hat{v}_j^\top s_2} \rightarrow c_0. \label{eq:non-spike-mom-i-call}\\
         \sum_{j=k^*+1}^{p}{\hat{\lambda}^2_j s_1^\top \hat{v}_j\hat{v}_j^\top s_2} \rightarrow c'_0.\label{eq:non-spike-mom-ii-call}
    \end{align}

\end{remark}
\begin{proof}[Proof of Lemma \ref{lem:spikes_non_spikes}]
To derive the results, we analyze the following bi-linear forms,
$$\hat{m}_p(z)=s_1^\top \left (\hat{\Sigma}-zI\right )^{-1}s_2=\sum_{j=1}^{p}{\frac{s_1^\top \hat{v}_j \hat{v}_j^\top s_2}{\hat{\lambda}_j-z}}; \quad \forall z\in \mathbb{C}^+.$$
Throughout the proof of Lemma~\ref{lem:spikes_non_spikes} and \ref{lem:conv_mhat_moment_integral}, let us further denote the unique eigenvalues and corresponding eigenspaces of $\Sigma_p$ by $\lambda_1>\tlambda_2>\ldots>\tlambda_{\bar{p}}$ and $\tV_1,\ldots,\tV_{\bar{p}}$. Assume $M_j$ to be the multiplicity of $\tlambda_j$, $M_j=1$ for $j=1,\ldots,k^*$, and thus the dimension of the eigenspace matrix (vector) $\tV_j$ is given by $p\times M_j$. This change of notation was done to accommodate non-spikes having multiplicities more than one. Then, \cite{mestre2006asymptotic} provides the following result, 
\begin{lemma}\label{res:mestre_conv_bilinear}
Suppose $\widetilde{\mathbb{X}}\sim \mathrm{GSp}(\left\{(\lambda_j,v_j)\right\}_{j=1}^{k^*};\Gamma_H;n,d)$ and $\hat{\Sigma}=\frac1n \widetilde{\mathbb{X}}^\top \widetilde{\mathbb{X}}$ has spectral decomposition $\hat{\Sigma}=\sum_{j=1}^{n \wedge p} \hat{\lambda}_j \hat{v}_j \hat{v}^\top_j$.The following holds under Assumption~\ref{assn:spike_and_dimension}(a),
\begin{equation}
    \left | \hat{m}_p(z)-m_p(z)\right |\xrightarrow{a.s.} 0; \quad \forall z\in \mathbb{C}^+,
\end{equation}
where,
$$\hat{m}_p(z)=s_1^\top \left (\hat{\Sigma}-zI\right )^{-1}s_2=\sum_{j=1}^{p}{\frac{s_1^\top \hat{v}_j \hat{v}_j^\top s_2}{\hat{\lambda}_j-z}}; \quad \forall z\in \mathbb{C}^+,$$
and 
\begin{equation*}
    m_p(z)=s_1^\top (w(z)\Sigma-zI)^{-1}s_2=\sum_{j=1}^{\bar{p}}{\frac{s_1^\top \tV_j \tV_j^\top s_2}{\tlambda_jw(z)-z}}; \quad \forall z\in \mathbb{C}^+.
\end{equation*}
The function $w(z)=1-\gamma-\gamma z b_F(z)$, where for all $z\in\mathbb{C}^+$, $b_F(z)=b$ is the unique solution to the equation,
$$b=\frac{1}{p}\sum_{j=1}^{\bar{p}}{\frac{M_j}{\tlambda_j(1-\gamma-\gamma z b)-z}}$$
in the set $\{b\in\mathbb{C}:\gamma b-(1-\gamma)/z\in \mathbb{C}^+\}$.
\end{lemma}

The domain for $z\in \mathbb{C}^+$ can be extended to $\mathbb{C}^-$ for the functions $\hat{m}_p,m_p,b_F$ by defining $\hat{m}_p(z)=\hat{m}_p^*(z^*)$, $m_p(z)=m_p^*(z^*)$, $b_F(z)=b_F^*(z^*)$ for $z\in \mathbb{C}^-$. Here, $(.)^*$ represents the complex-conjugate. It is easy to see that Result~\ref{res:mestre_conv_bilinear} is also satisfied when $z\in\mathbb{C}^-$. Let us further define $f(z)=z/w(z)$ for $z\in \mathbb{C}\setminus \mathbb{R}$. Then, using Result~\ref{res:mestre_conv_bilinear}, we can derive the following fixed point equation,
\begin{align*}
    z=f(z)\left (1-\frac{\gamma}{p}\sum_{j=1}^{\bar{p}}{\frac{M_j\tlambda_j}{\tlambda_j-f(z)}}\right ).
\end{align*}
We can further extend this definition to the entire $\mathbb{C}$ by allowing $b(z)=\lim_{y\rightarrow 0}{z+iy}$ for $z\in \mathbb{R}\setminus \{0\}$ (the existence of such limits is shown in \cite{silverstein_choi}), and defining $f(0)=0$ for $\gamma\leq 1$ and $f(0)=f_0$ for $\gamma>1$, where $f_0$ is the unique negative solution to the equation
$$1+\frac{\gamma}{p}\sum_{j=1}^{\bar{p}}{\frac{M_j\tlambda_j}{f_0-\tlambda_j}}=0.$$

The following lemma then provides the convergence of integrals of $\hat{m}_p$ over different contours.
\begin{lemma}\label{lem:conv_mhat_moment_integral}
Suppose $\widetilde{\mathbb{X}}\sim \mathrm{GSp}(\left\{(\lambda_j,v_j)\right\}_{j=1}^{k^*};\Gamma_H;n,d)$ and $\hat{\Sigma}=\frac1n \widetilde{\mathbb{X}}^\top \widetilde{\mathbb{X}}$ has spectral decomposition $\hat{\Sigma}=\sum_{j=1}^{n \wedge p} \hat{\lambda}_j \hat{v}_j \hat{v}^\top_j$. The following holds under Assumption \ref{assn:spike_and_dimension}(a).
\begin{equation}
    \left |\frac{1}{2\pi i}\oint_{\partial\hat{S}_y^{-}(\hat{\mathcal{L}}_j)}{z^r\hat{m}_p(z)dz}-\frac{1}{2\pi i}\oint_{\partial S_y^{-}(\mathcal{L}_j)}{z^rm_p(z)dz} \right | \xrightarrow{a.s.} 0,
\end{equation}
for $r=0,1,2$, and $j=1,\ldots,k^*,\bar{p}$. The contours $\partial\hat{S}_y^{-}(\hat{\mathcal{L}}_j)$ and $\partial S_y^{-}(\mathcal{L}_j)$ are negatively oriented boundaries of the rectangles $\hat{S}_y(\hat{\mathcal{L}}_j)$ and $S_y(\mathcal{L}_j)$ of width $y>0$ (across imaginary axis) defined by,
\begin{align*}
    \hat{S}_y(\hat{\mathcal{L}}_j)&=\left \{z\in \mathbb{C}: \hat{a}_1(j)\leq Re(z) \leq \hat{a}_2(j), |Im(z)|\leq y \right \},\qquad \hat{\mathcal{L}}_j=\begin{cases}\{\hat{\lambda}_j\},\quad j=1,\ldots, k^*\\\{\hat{\lambda}_{k^*+1},\ldots,\hat{\lambda}_p\},\quad j=\bar{p}\end{cases}\\
    S_y(\mathcal{L}_j)&=\left \{z\in \mathbb{C}: a_1(j)\leq Re(z) \leq a_2(j), |Im(z)|\leq y \right \},\qquad \mathcal{L}_j=\begin{cases}\{\psi(\tlambda_j)\},\quad j=1,\ldots, k^*\\
    \Gamma_F,\quad j=\bar{p}
    \end{cases},
\end{align*}
where $\Gamma_F$ is the support of the LSD $F$ of the sample eigenvalues $\hat{\lambda}_1,\ldots,\hat{\lambda}_p$. The interval $[\hat{a}_1(j),\hat{a}_2(j)]$ is chosen such that $\hat{\mathcal{L}}_j\subset [\hat{a}_1(j),\hat{a}_2(j)]$ and $\hat{\mathcal{L}}_{j'}\cap [\hat{a}_1(j),\hat{a}_2(j)]=\emptyset$ for $j'\in\{1,\ldots,k^*,\bar{p}\}, j'\neq j$. Similarly, the interval $[a_1(j),a_2(j)]$ is chosen such that $\mathcal{L}_j\subset [a_1(j),a_2(j)]$ and $\mathcal{L}_{j'}\cap [a_1(j),a_2(j)]=\emptyset$ for $j'\in\{1,\ldots,k^*,\bar{p}\}, j'\neq j$.

\end{lemma}
\begin{remark}
The ESD of the sample covariance matrix converges weakly to the Mar\v{c}enko-Pastur distribution ~\cite{marcenkopastur,bai2012gsp}, and therefore $F$ can be considered as the Mar\v{c}enko-Pastur distribution as well. Since traditionally, Mar\v{c}enko-Pastur distribution is defined for the traditional spiked model where $\lambda_{k^*+1}=\ldots=\lambda_p=1$, for the generalized spiked model, we can call this distribution the ``Generalized Mar\v{c}enko-Pastur" distribution.
\end{remark}

\begin{remark}
The choices for intervals $[a_1(j),a_2(j)]$ exist due to the exact separation of eigenvalue clusters as described in Propositions 3.1 and 3.2 in \cite{bai2012gsp}.
\end{remark}

\begin{remark}
We have only derived results in this paper based on the assumption that all the generalized spikes are distant spikes, and larger than $\sup{\Gamma_H}$. However, other scenarios such as the existence of close spikes, or spikes smaller than $\sup{\Gamma_H}$ can also be accommodated by manipulating the definitions of the sets $\hat{\mathcal{L}}_j$ and $\mathcal{L}_j$. For example, if $\tlambda_j$-s are distant spikes for $j=1,\ldots,k_1$ and close spikes for $j=k_1+1,\ldots,k_2$, then the definitions of $\mathcal{L}_j$ and $\hat{\mathcal{L}}_j$ will remain the same as defined in Lemma~\ref{lem:conv_mhat_moment_integral} with $k^*=k_1$, which is equivalent to treating the close spikes as non-spikes.
\end{remark}

Thereafter, we note that,
\begin{align*}
    \frac{1}{2\pi i}\oint_{\partial\hat{S}_y^{-}(\hat{\mathcal{L}}_t)}{z^r\hat{m}_p(z)dz}=\begin{cases}
    \hat{\lambda}_t^rs_1^\top \hat{v}_t \hat{v}_t^\top s_2, \quad t=1,\ldots,k^*\\
    \sum_{j=k^*+1}^{p}{\hat{\lambda}_{j}^rs_1^\top \hat{v}_{j} \hat{v}_{j}^\top s_2}, \quad t=\bar{p}.
    \end{cases}
\end{align*}

Therefore, in order to complete the proof of the lemma, we need to find
\begin{align*}
    \frac{1}{2\pi i}\oint_{\partial S_y^{-}(\mathcal{L}_t)}{z^rm_p(z)dz}&=\frac{1}{2\pi i}\oint_{\partial S_y^{-}(\mathcal{L}_t)}{\sum_{j=1}^{\bar{p}}\frac{z^rs_1^\top \tV_j \tV_j^\top s_2}{\tlambda_jw(z)-z}dz}\\
 &=\sum_{j=1}^{\bar{p}}{g(\tlambda_j,\partial S_y^{-}(\mathcal{L}_t),r  )s_1^\top \tV_j \tV_j^\top s_2},
\end{align*}
where,
\begin{align*}
    g(\tlambda_j,\partial S_y^{-}(\mathcal{L}_t),r)&=\frac{1}{2\pi i}\oint_{\partial S_y^{-}(\mathcal{L}_t)}{\frac{z^r}{\tlambda_jw(z)-z}dz},\\
    &=\frac{1}{2\pi i}\oint_{\partial S_y^{-}(\mathcal{L}_t)}{\frac{f(z)z^{r-1}}{\tlambda_j-f(z)}dz},\qquad (\text{by replacing $w(z)=z/f(z)$}),\\
    &=\frac{1}{2\pi i}\oint_{\partial S_y^{-}(\mathcal{L}_t)}{\frac{f(z)\left \{f(z)\left (1-\frac{\gamma}{p}\sum_{q=1}^{\bar{p}}{\frac{M_q\tlambda_q}{\tlambda_q-f(z)}}\right ) \right \}^{r-1}}{\tlambda_j-f(z)}dz}\\
    &=\frac{1}{2\pi i}\oint_{f(\partial S_y^{-}(\mathcal{L}_t))}{\frac{f\left \{f\left (1-\frac{\gamma}{p}\sum_{q=1}^{\bar{p}}{\frac{M_q\tlambda_q}{\tlambda_q-f}}\right ) \right \}^{r-1}}{\tlambda_j-f}\left \{1-\frac{\gamma}{p}\sum_{q=1}^{\bar{p}}{M_q\left (\frac{\tlambda_q}{\tlambda_q-f}\right )^2} \right \}df}.   
\end{align*}
To compute the integral, we have employed a change of variable from $z$ to $f$, and the contour on which the integration is to be performed has been changed from $\partial S_y^{-}(\mathcal{L}_t)$ to $f(\partial S_y^{-}(\mathcal{L}_t))$. The justification behind using the $f(\partial S_y^{-}(\mathcal{L}_t))$ as an integration contour is given in \cite{mestre2008asymptotic} Appendix III.

For the rest of the proof, we refrain from mentioning $\mathcal{L}_t$ in $S_y(\mathcal{L}_t)$ where it is obvious for brevity of notations. We compute this integral using residue theorem case by case.
\begin{description}

\item[Case 1: $r=1$,]
\begin{align*}
    g(\tlambda_j,\partial S_y^{-},1)&=\frac{1}{2\pi i}\oint_{f(\partial S_y^{-})}{\frac{f}{\tlambda_j-f}\left \{1-\frac{\gamma}{p}\sum_{q=1}^{\bar{p}}{M_q\left (\frac{\tlambda_q}{\tlambda_q-f}\right )^2} \right \}df}.    
\end{align*}

The integrand here has poles at each $\tlambda_q$ for $q=1,\ldots,\bar{p}$. The multiplicity of each pole is two if $q\neq j$, and three if $q=j$. Therefore, the complex residues at these poles are,
\begin{align*}
    Res\left (g(\tlambda_j,\partial S_y^{-},1),\tlambda_q \right )=\begin{cases}
    \frac{\gamma\tlambda_j M_q}{p}\left ( \frac{\tlambda_q}{\tlambda_q-\tlambda_j}\right )^2,\quad q\neq j,\\
    \tlambda_j\left [1-\frac{\gamma}{p}\sum_{l\neq j}{M_l\left (\frac{\tlambda_l}{\tlambda_l-\tlambda_j}\right )^2} \right ],\quad q=j.
    \end{cases}
\end{align*}

\begin{description}
\item[Case 1a: For $t=1,\ldots, k^*$,]
the contour $f(\partial S_y^{-})$ only surrounds the eigenvalue $\tlambda_t$ and no other eigenvalue. Therefore, using the residue theorem,
\begin{align*}
    g(\tlambda_j,\partial S_y^{-},1)&=\begin{cases}
    \frac{\gamma\tlambda_j M_t}{p}\left ( \frac{\tlambda_t}{\tlambda_t-\tlambda_j}\right )^2,\quad t\neq j,\\
    \tlambda_j\left [1-\frac{\gamma}{p}\sum_{l\neq j}{M_l\left (\frac{\tlambda_l}{\tlambda_l-\tlambda_j}\right )^2} \right ],\quad t=j.
    \end{cases}
\end{align*}

The limit of the integral when $p\rightarrow \infty $,
\begin{align*}
    \lim_{p\rightarrow\infty}{g(\tlambda_j,\partial S_y^{-},1)}&=\begin{cases}
    0, \quad t\neq j\\
    \tlambda_j\psi '(\tlambda_j),\quad t=j.\end{cases}    
\end{align*}

\item[Case 1b: For $t=\bar{p}$,]
the contour  $f(\partial S_y^{-})$ surrounds all the non-spikes $\tlambda_{k^*+1},\ldots,\tlambda_{\bar{p}}$. Therefore, using the residue theorem,
\begin{align*}
    g(\tlambda_j,\partial S_y^{-},1)&=\begin{cases}
    \frac{\gamma\tlambda_j}{p}\sum_{q=k^*+1}^{\bar{p}}{M_q\left (\frac{\tlambda_q}{\tlambda_q-\tlambda_j}\right )^2},\quad j=1,\ldots,k^*\\
    \tlambda_j\left [1-\frac{\gamma}{p}\sum_{q=1}^{k^*}{M_q\left (\frac{\tlambda_q}{\tlambda_q-\tlambda_j}\right )^2} \right ],\quad j=k^*+1,\ldots,\bar{p}.\end{cases}
\end{align*}

The limit of the integral when $p\rightarrow \infty $,
\begin{align*}
    \lim_{p\rightarrow\infty}{g(\tlambda_j,\partial S_y^{-},1)}&=\begin{cases}
    \gamma\tlambda_j\int_{\Gamma_H}{\left (\frac{\lambda}{\tlambda_j-\lambda}\right )^2dH(\lambda)},\quad j=1,\ldots,k^*\\
    \tlambda_j,\quad j=k^*+1,\ldots,p.\end{cases}    
\end{align*}

\end{description}

\item[Case 2: $r=2$,]
\begin{align*}
    g(\tlambda_j,\partial S_y^{-},2)&=\frac{1}{2\pi i}\oint_{f(\partial S_y^{-})}{\frac{f^2}{\tlambda_j-f}\left (1-\frac{\gamma}{p}\sum_{q=1}^{\bar{p}}{\frac{M_q\tlambda_q}{\tlambda_q-f}}\right )\left \{1-\frac{\gamma}{p}\sum_{q=1}^{\bar{p}}{M_q\left (\frac{\tlambda_q}{\tlambda_q-f}\right )^2} \right \}df}.    
\end{align*}
The integrand has poles at each $\tlambda_q$ for $q=1,\ldots,\bar{p}$. The multiplicity of each pole is three if $q\neq j$, and four if $q=j$. Therefore, the complex residues at these poles are,
\begin{align*}
    &Res\left (g(\tlambda_j,\partial S_y^{-},2),\tlambda_q \right )\\
    &=\begin{cases}
    \frac{\gamma\tlambda_j M_q \tlambda_q^3}{p(\tlambda_j-\tlambda_q)^2}\left [1+\frac{\gamma\tlambda_j M_q }{p(\tlambda_j -\tlambda_q)}+\frac{\gamma}{p}\sum_{\substack{l=1\\l\neq q}}^{p}{\frac{M_l\tlambda_l}{\tlambda_q-\tlambda_l}} \right ],\quad q\neq j,\\
    \tlambda_j^2 \left [\frac{\gamma M_j}{p}\left (1+\frac{\gamma}{p}\sum_{\substack{l=1\\l\neq j}}^{p}{\frac{M_l\tlambda_l^3}{(\tlambda_j-\tlambda_l)^3}}\right )+\left (1+\frac{\gamma}{p}\sum_{\substack{l=1\\l\neq j}}^{p}{\frac{M_l\tlambda_l}{\tlambda_j-\tlambda_l}}\right )\left (1-\frac{\gamma}{p}\sum_{\substack{l=1\\l\neq j}}^{p}{\frac{M_l\tlambda_l^2}{(\tlambda_j-\tlambda_l)^2}}\right )\right ],\quad q=j.
    \end{cases}
\end{align*}

\begin{description}

\item[Case 2a: For $t=1,\ldots, k^*$,]
the contour $f(\partial S_y^{-})$ only surrounds the eigenvalue $\tlambda_t$ and no other eigenvalue. Therefore, using the residue theorem,
\begin{align*}
    &g(\tlambda_j,\partial S_y^{-},2)\\
    &=\begin{cases}
    \frac{\gamma\tlambda_j M_t \tlambda_t^3}{p(\tlambda_j-\tlambda_t)^2}\left [1+\frac{\gamma\tlambda_j M_t }{p(\tlambda_j -\tlambda_t)}+\frac{\gamma}{p}\sum_{\substack{l=1\\l\neq t}}^{p}{\frac{M_l\tlambda_l}{\tlambda_t-\tlambda_l}} \right ],\quad t\neq j,\\
    \tlambda_j^2 \left [\frac{\gamma M_j}{p}\left (1+\frac{\gamma}{p}\sum_{\substack{l=1\\l\neq j}}^{p}{\frac{M_l\tlambda_l^3}{(\tlambda_j-\tlambda_l)^3}}\right )+\left (1+\frac{\gamma}{p}\sum_{\substack{l=1\\l\neq j}}^{p}{\frac{M_l\tlambda_l}{\tlambda_j-\tlambda_l}}\right )\left (1-\frac{\gamma}{p}\sum_{\substack{l=1\\l\neq j}}^{p}{\frac{M_l\tlambda_l^2}{(\tlambda_j-\tlambda_l)^2}}\right )\right ],\quad t=j.
    \end{cases}
\end{align*}

The limit of the integral when $p\rightarrow \infty $,
\begin{align*}
    \lim_{p\rightarrow\infty}{g(\tlambda_j,\partial S_y^{-},2)}&=\begin{cases}
    0, \quad t\neq j\\
    \tlambda_j\psi (\tlambda_j) \psi '(\tlambda_j),\quad t=j.\end{cases}    
\end{align*}

\item[Case 2b: For $t=\bar{p}$,]
the contour  $f(\partial S_y^{-})$ surrounds all the non-spikes $\tlambda_{k^*+1},\ldots,\tlambda_{\bar{p}}$. Therefore, we first calculate the function $g(\tlambda_j,\partial S_y^{-},2)$ for $j=1,\ldots,k^*$ using the residue theorem,
\begin{align*}
   \ &  g(\tlambda_j,\partial S_y^{-},2)\\&=\sum_{q=k^*+1}^{p}{\frac{\gamma\tlambda_j M_q \tlambda_q^3}{p(\tlambda_j-\tlambda_q)^2}\left [1+\frac{\gamma\tlambda_j M_q }{p(\tlambda_j -\tlambda_q)}+\frac{\gamma}{p}\sum_{\substack{l=1\\l\neq q}}^{p}{\frac{M_l\tlambda_l}{\tlambda_q-\tlambda_l}} \right ]}\\
    &=\frac{\gamma\tlambda_j}{p}\sum_{q=k^*+1}^{p}{\frac{M_q\tlambda_q^3}{(\tlambda_j-\tlambda_q)^2}}+\frac{\gamma^2\tlambda_j^2}{p^2}\sum_{q=k^*+1}^{p}{\frac{M_q^2\tlambda_q^3}{(\tlambda_j-\tlambda_q)^3}}\\
    &+\frac{\gamma^2\tlambda_j}{p^2}\sum_{q=k^*+1}^{p}{\sum_{\substack{l=k^*+1\\l\neq q}}^{p}{\frac{M_q\tlambda_q^3 M_l \tlambda_l}{(\tlambda_j-\tlambda_q)^2(\tlambda_q-\tlambda_l)}}}+\frac{\gamma^2\tlambda_j}{p^2}\sum_{q=k^*+1}^{p}{\sum_{l=1}^{k^*}{\frac{M_q\tlambda_q^3 M_l \tlambda_l}{(\tlambda_j-\tlambda_q)^2(\tlambda_q-\tlambda_l)}}}.
\end{align*}
Now, for any pair of indices $q$ and $l$ such that $l\neq q$, one can show through some algebraic manipulation,
\begin{align*}
    \frac{M_q\tlambda_q^3 M_l \tlambda_l}{(\tlambda_j-\tlambda_q)^2(\tlambda_q-\tlambda_l)}+\frac{M_l\tlambda_l^3 M_q \tlambda_q}{(\tlambda_j-\tlambda_l)^2(\tlambda_l-\tlambda_q)}&=\frac{\tlambda_jM_q\tlambda_q^2 M_l \tlambda_l}{(\tlambda_j-\tlambda_q)^2(\tlambda_j-\tlambda_l)}+\frac{\tlambda_jM_l\tlambda_l^2 M_q \tlambda_q}{(\tlambda_j-\tlambda_l)^2(\tlambda_j-\tlambda_q)}
\end{align*}
Therefore,
\begin{align*}
     \ & g(\tlambda_j,\partial S_y^{-},2)\\
     &=\frac{\gamma\tlambda_j}{p}\sum_{q=k^*+1}^{p}{\frac{M_q\tlambda_q^3}{(\tlambda_j-\tlambda_q)^2}}+\frac{\gamma^2\tlambda_j^2}{p^2}\sum_{q=k^*+1}^{p}{\frac{M_q^2\tlambda_q^3}{(\tlambda_j-\tlambda_q)^3}}\\
    &+\frac{\gamma^2\tlambda_j^2}{p^2}\sum_{q=k^*+1}^{p}{\sum_{\substack{l=k^*+1\\l\neq q}}^{p}{\frac{M_q\tlambda_q^2 M_l \tlambda_l}{(\tlambda_j-\tlambda_q)^2(\tlambda_j-\tlambda_l)}}}+\frac{\gamma^2\tlambda_j}{p^2}\sum_{q=k^*+1}^{p}{\sum_{l=1}^{k^*}{\frac{M_q\tlambda_q^3 M_l \tlambda_l}{(\tlambda_j-\tlambda_q)^2(\tlambda_q-\tlambda_l)}}}\\
    &=\frac{\gamma\tlambda_j}{p}\sum_{q=k^*+1}^{p}{\frac{M_q\tlambda_q^3}{(\tlambda_j-\tlambda_q)^2}}+\frac{\gamma^2\tlambda_j^2}{p^2}\left (\sum_{q=k^*+1}^{p}{\frac{M_q\tlambda_q^2}{(\tlambda_j-\tlambda_q)^2}}\right )\left (\sum_{q=k^*+1}^{p}{\frac{M_q\tlambda_q}{\tlambda_j-\tlambda_q}}\right )+O(1/p).
\end{align*}

The limit of the integral when $p\rightarrow \infty $,
\begin{align*}
    \lim_{p\rightarrow\infty}{g(\tlambda_j,\partial S_y^{-},2)}&=\gamma\tlambda_j\int_{\Gamma_H}{ \frac{\lambda^3}{(\tlambda_j-\lambda)^2}dH(\lambda)} \\
    &+\gamma^2\tlambda_j^2\left [\int_{\Gamma_H}{ \frac{\lambda}{\tlambda_j-\lambda}dH(\lambda)}\right ]\left [\int_{\Gamma_H}{ \left (\frac{\lambda}{\tlambda_j-\lambda}\right )^2 dH(\lambda)}\right ].   
\end{align*}

Next, we calculate the function $g(\tlambda_j,\partial S_y^{-},2)$ for $j=k^*+1,\ldots,p$ using the residue theorem,
\begin{align*}
   \ &  g(\tlambda_j,\partial S_y^{-},2)\\
   &=\sum_{\substack{q=k^*+1\\q\neq j}}^{p}{\frac{\gamma\tlambda_j M_q \tlambda_q^3}{p(\tlambda_j-\tlambda_q)^2}\left [1+\frac{\gamma\tlambda_j M_q }{p(\tlambda_j -\tlambda_q)}+\frac{\gamma}{p}\sum_{\substack{l=1\\l\neq q}}^{p}{\frac{M_l\tlambda_l}{\tlambda_q-\tlambda_l}} \right ]}\\
    &+\tlambda_j^2 \left [\frac{\gamma M_j}{p}\left (1+\frac{\gamma}{p}\sum_{\substack{q=1\\q\neq j}}^{p}{\frac{M_q\tlambda_q^3}{(\tlambda_j-\tlambda_q)^3}}\right )+\left (1+\frac{\gamma}{p}\sum_{\substack{q=1\\q\neq j}}^{p}{\frac{M_q\tlambda_q}{\tlambda_j-\tlambda_q}}\right )\left (1-\frac{\gamma}{p}\sum_{\substack{q=1\\q\neq j}}^{p}{\frac{M_q\tlambda_q^2}{(\tlambda_j-\tlambda_q)^2}}\right )\right ]\\
    &=\sum_{\substack{q=k^*+1\\q\neq j}}^{p}{\frac{\gamma\tlambda_j M_q \tlambda_q^3}{p(\tlambda_j-\tlambda_q)^2}}+\sum_{\substack{q=k^*+1\\q\neq j}}^{p}{\frac{\gamma^2\tlambda_j^2 M_q^2 \tlambda_q^3}{p^2(\tlambda_j-\tlambda_q)^3}}+\frac{\gamma^2\tlambda_j}{p^2}\sum_{\substack{q=k^*+1\\q\neq j}}^{p}{\sum_{\substack{l=1\\l\neq q}}^{p}{\frac{M_q\tlambda_q^3 M_l \tlambda_l}{(\tlambda_j-\tlambda_q)^2(\tlambda_q-\tlambda_l)}}}\\
    &+\frac{\gamma}{p}M_j\tlambda_j^2+\frac{\gamma^2M_j\tlambda_j^2}{p^2}\sum_{\substack{q=1\\q\neq j}}^{p}{\frac{M_q\tlambda_q^3}{(\tlambda_j-\tlambda_q)^3}}+\tlambda_j^2+\frac{\gamma\tlambda_j^2}{p}\sum_{\substack{q=1\\q\neq j}}^{p}{\frac{M_q\tlambda_q}{\tlambda_j-\tlambda_q}}-\frac{\gamma\tlambda_j^2}{p}\sum_{\substack{q=1\\q\neq j}}^{p}{\frac{M_q\tlambda_q^2}{(\tlambda_j-\tlambda_q)^2}}\\
    &-\tlambda_j^2\left (\frac{\gamma}{p}\sum_{\substack{q=1\\q\neq j}}^{p}{\frac{M_q\tlambda_q}{\tlambda_j-\tlambda_q}}\right )\left (\frac{\gamma}{p}\sum_{\substack{q=1\\q\neq j}}^{p}{\frac{M_q\tlambda_q^2}{(\tlambda_j-\tlambda_q)^2}}\right )
\end{align*}
Terms $1,7,8$ yield,
\begin{align*}
    \sum_{\substack{q=k^*+1\\q\neq j}}^{p}{\frac{\gamma\tlambda_j M_q \tlambda_q^3}{p(\tlambda_j-\tlambda_q)^2}}+\frac{\gamma\tlambda_j^2}{p}\sum_{\substack{q=1\\q\neq j}}^{p}{\frac{M_q\tlambda_q}{\tlambda_j-\tlambda_q}}-\frac{\gamma\tlambda_j^2}{p}\sum_{\substack{q=1\\q\neq j}}^{p}{\frac{M_q\tlambda_q^2}{(\tlambda_j-\tlambda_q)^2}}=\frac{\gamma \tlambda_j}{p}\sum_{\substack{q=1\\q\neq j}}^{p}{M_q\tlambda_q}+{O}(1/p).
\end{align*}
Terms $2$ and $3$ yield,
\begin{align*}
    &\sum_{\substack{q=k^*+1\\q\neq j}}^{p}{\frac{\gamma^2\tlambda_j^2 M_q^2 \tlambda_q^3}{p^2(\tlambda_j-\tlambda_q)^3}}+\frac{\gamma^2\tlambda_j}{p^2}\sum_{\substack{q=k^*+1\\q\neq j}}^{p}{\sum_{\substack{l=1\\l\neq q}}^{p}{\frac{M_q\tlambda_q^3 M_l \tlambda_l}{(\tlambda_j-\tlambda_q)^2(\tlambda_q-\tlambda_l)}}}\\
    &=\sum_{\substack{q=1\\q\neq j}}^{p}{\frac{\gamma^2\tlambda_j^2 M_q^2 \tlambda_q^3}{p^2(\tlambda_j-\tlambda_q)^3}}+\frac{\gamma^2\tlambda_j}{p^2}\sum_{\substack{q=k^*+1\\q\neq j}}^{p}{\sum_{\substack{l=1\\l\neq q\\l\neq j}}^{p}{\frac{M_q\tlambda_q^3 M_l \tlambda_l}{(\tlambda_j-\tlambda_q)^2(\tlambda_q-\tlambda_l)}}}\\&-\frac{\gamma^2M_j\tlambda_j^2}{p^2}\sum_{\substack{q=k^*+1\\q\neq j}}^{p}{\frac{M_q \tlambda_q^3}{(\tlambda_j-\tlambda_q)^3}}+{O}(1/p)\\
    &=\tlambda_j^2\left (\frac{\gamma}{p}\sum_{\substack{q=1\\q\neq j}}^{p}{\frac{M_q\tlambda_q}{\tlambda_j-\tlambda_q}}\right )\left (\frac{\gamma}{p}\sum_{\substack{q=1\\q\neq j}}^{p}{\frac{M_q\tlambda_q^2}{(\tlambda_j-\tlambda_q)^2}}\right )-\frac{\gamma^2M_j\tlambda_j^2}{p^2}\sum_{\substack{q=k^*+1\\q\neq j}}^{p}{\frac{M_q \tlambda_q^3}{(\tlambda_j-\tlambda_q)^3}}+{O}(1/p),
\end{align*}
using similar algebraic manipulations as in Case 1. Therefore,
\begin{align*}
    g(\tlambda_j,\partial S_y^{-},2)&=\tlambda_j^2\left (1+\frac{\gamma}{p} M_j\right )+\frac{\gamma \tlambda_j}{p}\sum_{\substack{q=1\\q\neq j}}^{p}{M_q\tlambda_q}+{O}(1/p)\\
    &=\tlambda_j^2+\frac{\gamma \tlambda_j}{p}\sum_{q=1}^{p}{M_q\tlambda_q}+{O}(1/p). 
\end{align*}
The limit of the integral when $p\rightarrow \infty $,
\begin{align*}
    \lim_{p\rightarrow\infty}{g(\tlambda_j,\partial S_y^{-},2)}&=\tlambda_j^2+\gamma\tlambda_j\int_{\Gamma_H}{\lambda dH(\lambda)}.   
\end{align*}
\end{description}
And hence we have derived all the asymptotic quantities mentioned in Lemma~\ref{lem:spikes_non_spikes}. The final step of replacing the quantities $g(\tlambda_j,\partial S_y^{-},r)$ with their corresponding limits follows from the dominated convergence theorem by noting that $s_1, s_2$ have uniformly bounded norms, and thus $\sum_{j=1}^{\bar{p}}{s_1^\top \tilde{V}_j\tilde{V}_j^\top s_2}$ is bounded.
\end{description}
\end{proof}

\begin{proof}[Proof of Lemma~\ref{lem:conv_mhat_moment_integral}]
First, we show that $\hat{a}_1(j),\hat{a}_2(j),a_1(j),a_2(j)$ can be chosen to satisfy $\hat{a}_1(j)\rightarrow a_1(j)$ and $\hat{a}_2(j)\rightarrow a_2(j)$ for $j=1,\ldots,k^*,\bar{p}$.

For $j=1,\ldots,k^*$, the eigenvalues $\tlambda_j$ are distant spikes. Thus, according to \cite[Theorem 4.1]{bai2012gsp}, $\hat{\lambda}_j\xrightarrow{a.s.}\psi(\tlambda_j)$. Further, $\psi(\tlambda_j)$ is bounded away from $\Gamma_F$, the support of the sample LSD, as well as from any other $\psi(\tlambda_{j'})$ for $j\in\{1,\ldots,k^*\},j\neq j'$. Therefore, there exists an interval $I$ such that for large enough $p$ and $n$, $\{\hat{\lambda}_j,\psi(\tlambda_j)\}\subset I$ and $I\cap \mathcal{L}_{j'}=\emptyset,I\cap \hat{\mathcal{L}}_{j'}=\emptyset$ for $j'\in \{1,\ldots,k^*,\bar{p}\},j'\neq j$, with probability 1. We can then choose the intervals $[\hat{a}_1(j),\hat{a}_2(j)],[a_1(j),a_2(j)]\subset I$ to satisfy the claim above. Further, we can choose the interval boundaries to also be bounded away from $\tilde{\Gamma}_F=\Gamma_F\cup\{\psi(\lambda_1),\ldots,\psi(\lambda_{k^*})\}$.

For $j=\bar{p}$, the ESD of the sample eigenvalues corresponding to non-spikes converge to the sample LSD (or the generalized Mar\v{c}enko-Pastur distribution) $F$. Thus, there exists an interval $I$ such that for large enough $p$ and $n$, $\Gamma_H\subset I, \{\hat{\lambda}_{k^*+1},\ldots,\hat{\lambda}_p\}\subset I$ and $I\cap \mathcal{L}_{j}=\emptyset,I\cap \hat{\mathcal{L}}_{j}=\emptyset$ for $j\in \{1,\ldots,k^*\}$, with probability 1. We can then choose the intervals $[\hat{a}_1(j),\hat{a}_2(j)],[a_1(j),a_2(j)]\subset I$ to satisfy the claim above. Further, we can choose the interval boundaries to also be bounded away from $\tilde{\Gamma}_F$.

Then,
\begin{align*}
    &\left |\frac{1}{2\pi i}\oint_{\partial\hat{S}_y^{-}(\hat{\mathcal{L}}_j)}{z^r\hat{m}_p(z)dz}-\frac{1}{2\pi i}\oint_{\partial S_y^{-}(\mathcal{L}_j)}{z^rm_p(z)dz} \right |\\
    &\leq \left |\frac{1}{2\pi i}\oint_{\partial\hat{S}_y^{-}(\hat{\mathcal{L}}_j)}{z^r\hat{m}_p(z)dz}-\frac{1}{2\pi i}\oint_{\partial \hat{S}_y^{-}(\hat{\mathcal{L}}_j)}{z^r\hat{m}_p(z)dz} \right |+\left |\frac{1}{2\pi i}\oint_{\partial S_y^{-}(\mathcal{L}_j)}{\left (z^r\hat{m}_p(z)-z^rm_p(z)\right )dz} \right |\\
    &\leq \frac{1}{2\pi}\left \{\sup_{z\in\partial\hat{S}_y^{-}(\hat{\mathcal{L}}_j)\cup \partial S_y^{+}(\mathcal{L}_j)}{|z^r\hat{m}_p(z)|} \right \}(|\hat{a}_1(j)-a_1(j)|+|\hat{a}_2(j)-a_2(j)|)\\
    &+\frac{1}{2\pi} \oint_{\partial S_y^{-}(\mathcal{L}_j)}{|z^r||\hat{m}_p(z)-m_p(z)||dz|}.
\end{align*}
For the first term, using Cauchy-Schwartz inequality, we have,
\begin{align*}
    |z^r\hat{m}_p(z)|\leq \frac{\|z^r \| \|s_1\| \|s_2\|}{d\left (z,\tilde{\Gamma}_F \right )},
\end{align*}
where $d(z,S)=\inf_{y\in S}{|z-y|}$ for $S\subset \mathbb{C}$. Since $a_1(z),a_2(z)$ are chosen such that $d\left (a_1(z),\tilde{\Gamma}_F\right )>0,d\left (a_1(z),\tilde{\Gamma}_F\right )>0$, $\exists M,P$ large enough such that for $p>P$, with probability one,
\begin{align*}
    \sup_{z\in\partial\hat{S}_y^{-}(\hat{\mathcal{L}}_j)\cup \partial S_y^{+}(\mathcal{L}_j)}{|z^r\hat{m}_p(z)|}<M.
\end{align*}
Therefore, the convergence of the first term on the right-hand side above to zero is complete as $\hat{a}_1(z)\rightarrow a_1(z)$ and $\hat{a}_2(z)\rightarrow a_2(z)$.

For the second term, as $\hat{m}_p(z)$ and $m_p(z)$ are holomorphic on the compact set $\partial S_y^{-}(\mathcal{L}_j)$, we have $\sup_{z\in \partial S_y^{-}(\mathcal{L}_j)}{|z^r|}$ bounded above and
$\sup_{z\in \partial S_y^{-}(\mathcal{L}_j)}{|\hat{m}_p(z)-m_p(z)|}$ bounded above with probability one. Moreover, from Result~\ref{res:mestre_conv_bilinear}, $|\hat{m}_p(z)-m_p(z)|\xrightarrow{a.s.}0$ point-wise for $z\in\mathbb{C}\setminus \mathbb{R}$. Therefore, using Lebesgue's dominated convergence theorem, the second term on the right-hand side above converges to zero almost surely, and that completes the proof.
\end{proof}

\begin{lemma}\label{lem:lmnt}
Suppose $\widetilde{\mathbb{X}}\sim \mathrm{GSp}(\left\{(\lambda_j,v_j)\right\}_{j=1}^{k^*};\Gamma_H;n,p)$ and $\hat{\Sigma}=\frac1n \widetilde{\mathbb{X}}^\top \widetilde{\mathbb{X}}$ has spectral decomposition $\hat{\Sigma}=\sum_{j=1}^{n \wedge p} \hat{\lambda}_j \hat{v}_j \hat{v}^\top_j$. Under Assumption $2.6$(a), there exists $m_1, m_2>0$ depending on $\Gamma_H$ such that,
\begin{align}\label{eq:lambda_hat_moment}
     \frac{1}{p}\sum_{j=k+1}^{n \wedge p} \hat{\lambda}_j \rightarrow m_1, \\
     \frac{1}{p}\sum_{j=k+1}^{n \wedge p} \hat{\lambda}^2_j \rightarrow m_2.
\end{align}
\end{lemma}
\begin{proof}
See e.g. \cite{couillet2011random}, \cite{bai2010spectral}.
\end{proof}

\end{document}